\documentclass{amsart}
\usepackage{amssymb,enumitem,imakeidx,mathabx,mathrsfs,tikz-cd}
\usepackage[french,english]{babel}
\usepackage[T1]{fontenc}
\usepackage[colorlinks,linkcolor=blue,citecolor=cyan]{hyperref}

\newcommand{\newabstract}[1]{
  \par\bigskip
  \csname otherlanguage*\endcsname{#1}
  \csname captions#1\endcsname
  \item[\hskip\labelsep\scshape\abstractname.]
}
\newtheorem{Thm}{Theorem}[section]
\newtheorem{Cor}[Thm]{Corollary}
\newtheorem{Lem}[Thm]{Lemma}
\newtheorem{Prop}[Thm]{Proposition}
\newtheorem*{Thm*}{Theorem}
\theoremstyle{definition}
\newtheorem{Def}[Thm]{Definition}
\newtheorem{Ex}[Thm]{Example}
\theoremstyle{remark}
\newtheorem{Rmk}[Thm]{Remark}
\numberwithin{equation}{section}
\begin{document}

\title[Coherent cohomology and automorphic forms]{Coherent cohomology of Shimura varieties and automorphic forms}
\author{Jun Su}
\address{Department of Pure Mathematics and Mathematical Statistics, University of Cambridge, Wilberforce Road, Cambridge, CB3 0WB, UK}
\email{js2539@cantab.ac.uk}
\date{September,~2022}
\begin{abstract}
We show that the cohomology of canonical extensions of automorphic vector bundles over toroidal compactifications of Shimura varieties can be computed by relative Lie algebra cohomology of automorphic forms. Our result is inspired by and parallel to Borel and Franke's work on the cohomology of automorphic local systems on locally symmetric spaces, and also generalises a theorem of Mumford.
\newabstract{french}
Nous montrons que la cohomologie des extensions canoniques des fibr\'{e}s vectoriels automorphes sur les compactifications toro\"{i}dales des vari\'{e}t\'{e}s de Shimura peut \^{e}tre calcul\'{e}e par la cohomologie relative des alg\`{e}bres de Lie des formes automorphes. Notre r\'{e}sultat est inspir\'{e} par les travaux de Borel et Franke sur la cohomologie des syst\`{e}mes locaux automorphes sur les espaces localement sym\'{e}triques, et g\'{e}n\'{e}ralise \'{e}galement un th\'{e}or\`{e}me de Mumford.
\end{abstract}
\maketitle
\tableofcontents

\section*{Introduction}

The cohomology of Shimura varieties are of interest and importance in the Langlands program as there Galois representations and Hecke modules meet each other (for the first time). Let $(G,X)$ be a Shimura datum, $Sh(G,X)_\mathbb{K}$ be the associated Shimura variety at a neat level $\mathbb{K}$ and $\mathrm{Sh}_\mathbb{K}$ be its complex analytification, then there are the Betti cohomology $H^*(\mathrm{Sh}_\mathbb{K},\mathbb{C})$ and more generally cohomology of local systems on $\mathrm{Sh}_\mathbb{K}$ arising from algebraic representations of $G$. In this case the Galois representations spring from \'{e}tale cohomology, while on the other hand at first it is in general unknown whether the natural Hecke module structures on these cohomology come from automorphic representations of $G$. This question can actually be asked for every locally symmetric space: for simplicity we assume the locally symmetric space is of the form
\[S_\mathbb{K}:=G(\mathbb{Q})\backslash G(\mathbb{A})/K\mathbb{K},\]
where $G$ is a semisimple group over $\mathbb{Q}$, $K$ is an open subgroup of a maximal compact subgroup of $G(\mathbb{R})$ and $\mathbb{K}\subseteq G(\mathbb{A}_f)$ is a neat compact open subgroup, then every finite-dimensional representation $E$ of $G(\mathbb{R})$ defines a local system
\[\underline{E}:=G(\mathbb{Q})\backslash((G(\mathbb{A})/K\mathbb{K})\times E)\]
on $S_\mathbb{K}$ (which we will refer to as an \it automorphic local system\rm). There are standard Hecke-equivariant isomorphisms
\begin{equation}\label{CinftyB}
H^i(S_\mathbb{K},\underline{E})\cong H^i_{(\mathfrak{g},K)}(C^\infty(G(\mathbb{Q})\backslash G(\mathbb{A})/\mathbb{K})^{K\mbox{-}\mathrm{fin}}\otimes E),
\end{equation}
where $\mathfrak{g}$ is the complexified Lie algebra of $G$. Borel \cite{Borel84}, \cite{Borel83} conjectured that in (\ref{CinftyB}) $C^\infty(G(\mathbb{Q})\backslash G(\mathbb{A})/\mathbb{K})$ can be replaced by $\mathcal{A}(G)^\mathbb{K}$, i.e. we have
\begin{equation}\label{AGB}
H^i(S_\mathbb{K},\underline{E})\cong H^i_{(\mathfrak{g},K)}(\mathcal{A}(G)^\mathbb{K}\otimes E),
\end{equation}
where $\mathcal{A}(G)$ denotes the space of automorphic forms on $G$. Note that in the above settings there is nothing special about $\mathbb{Q}$ among all number fields: for $G$ over any number field $F$, the locally symmetric spaces and the automorphic sheaves thereon arising from $G$ or $\mathrm{Res}_{F/\mathbb{Q}}G$ are the same, while $\mathcal{A}(G)$ and $\mathcal{A}(\mathrm{Res}_{F/\mathbb{Q}}G)$ are canonically isomorphic as Hecke modules. The difficulty of the conjecture is rooted in the non-compactness of $S_\mathbb{K}$ and hence the cases where $G$ is anisotropic over $\mathbb{Q}$ are somewhat straightforward. Borel himself showed that in (\ref{CinftyB}) one can replace $C^\infty(G(\mathbb{Q})\backslash G(\mathbb{A})/\mathbb{K})^{K\mbox{-}\mathrm{fin}}$ with the space $C^\infty_{\mathrm{umg}}(G)^\mathbb{K}$ of $K$-finite smooth functions on $G(\mathbb{Q})\backslash G(\mathbb{A})/\mathbb{K}$ with uniformly moderate growth (see Definition \ref{mg}) \cite[3.2]{Borel83}, \cite[Theorem 1]{Borel90}, where the latter differs from $\mathcal{A}(G)^\mathbb{K}$ by a finiteness condition under differentiation by elements of the center $\mathfrak{Z}(\mathfrak{g})$ of the universal enveloping algebra of $\mathfrak{g}$. Harder and Casselman-Speh solved the rank-$1$ cases of the conjecture, and the full conjecture was eventually proved by Franke in his famous paper \cite{Franke98}.

For Shimura varieties it is also natural to consider coherent sheaf cohomology groups like $H^q(\mathrm{Sh}_\mathbb{K},\Omega^p)$. A suitable class of coefficients are the automophic vector bundles introduced by Harris and Milne in \cite{Harris85} and \cite{Milne88}. These vector bundles are parametrized by representations of a parabolic subgroup $P_h$ (arising from a $h\in X$) of $G_\mathbb{C}$ and they are analytifications of algebraic vector bundles defined over number fields specified by the representations. They have the name as every holomorphic or nearly holomorphic (vector-valued) automorphic form can be intepreted as a global section of one of them. The family of these vector bundles is closed under tensor operations and includes the sheaves $\Omega^p_{\mathrm{Sh}_\mathbb{K}}$ of holomorphic $(p,0)$-forms for all $p$. However, the cohomology of these vector bundles may not always be useful. The basic example is that in the modular curve case these cohomology groups are torsion-free modules over the ring of polynomials in the $j$-invariant and hence are either $0$ or infinite-dimensional. In general for an automorphic vector bundle $\widetilde{V}$ arising from a representation $V$ of $P_h$, analogous to (\ref{CinftyB}) we have
\[H^i(\mathrm{Sh}_\mathbb{K},\widetilde{V})\cong H^i_{(\mathfrak{p}_h,K_h)}(C^\infty(G(\mathbb{Q})\backslash G(\mathbb{A})/\mathbb{K})\otimes V)\]
(see (\ref{CinftyC})), where $\mathfrak{p}_h$ is the Lie algebra of $P_h$, $K_h\subseteq G(\mathbb{R})$ is the stabilizer of a point in $X$ and when $K_h$ is non-compact $H^i_{(\mathfrak{p}_h,K_h)}(-)$ abusively denotes cohomology of the complex \cite[2.127]{KV95} usually used to compute relative Lie algebra cohomology. In contrast to (\ref{AGB}) the right hand side above could be strictly larger than
\[H^i_{(\mathfrak{p}_h,K_h)}(\mathcal{A}(G)^\mathbb{K}\otimes V).\]
In fact by \cite[9.2,10.1]{Lan16} and the main theorem \ref{main} of this paper, when $i$ is $1$ less than the codimension of the boundary in the Baily-Borel compactification this happens for every $\widetilde{V}$ after twisted by a sufficiently high power of the canonical bundle. A remedy for this problem is built upon the theory of toroidal compactifications of locally symmetric varieties introduced by Ash, Mumford, Rapoport and Tai in \cite{AMRT75} (following the work of Igusa on Siegel modular varieties and Hirzebruch on Hilbert modular surfaces) and the canonical extensions of automorphic vector bundles over these compactifications introduced by Mumford and Harris in \cite{Mumford77} and \cite{Harris89}. The aim of this paper is to show that after applying these constructions we get favorable cohomology groups (see Theorem \ref{main}):

\begin{Thm*} Let $\widetilde{V}$ be an automorphic vector bundle over $\mathrm{Sh}_\mathbb{K}$ arising from a representation $V$ of $P_h$, $\widetilde{V}^{\mathrm{can}}$ be its canonical extension over an admissible toroidal compactification $\mathrm{Sh}_{\mathbb{K},\Sigma}$ of $\mathrm{Sh}_\mathbb{K}$, then there are Hecke-equivariant isomorphisms
\begin{equation}\label{AGC}
H^i(\mathrm{Sh}_{\mathbb{K},\Sigma},\widetilde{V}^{\mathrm{can}})\cong H^i_{(\mathfrak{p}_h,K_h)}(\mathcal{A}(G)^\mathbb{K}\otimes V).
\end{equation}
\end{Thm*}

We note that the insight to consider the left hand side above originated in Harris' paper \cite{Harris90}. Toroidal compactifications of a locally symmetric variety are usually not unique and form an inverse system indexed by some combinatorial data $\Sigma$, but the cohomology of canonical extensions of an automorphic vector bundle to different toroidal compactifications are naturally isomorphic. There are plenty of toroidal compactifications $\mathrm{Sh}_{\mathbb{K},\Sigma}$ of $\mathrm{Sh}_\mathbb{K}$ that are smooth and such that the boundary $Z:=\mathrm{Sh}_{\mathbb{K},\Sigma}\smallsetminus\mathrm{Sh}_\mathbb{K}$ is a normal crossings divisor. In this case $(\Omega^p_{\mathrm{Sh}_\mathbb{K}})^{\mathrm{can}}$ is the sheaf $\Omega^p_{\mathrm{Sh}_{\mathbb{K},\Sigma}}(\log{Z})$ of holomorphic $(p,0)$-forms on $\mathrm{Sh}_\mathbb{K}$ with log poles along $Z$, whose cohomology groups appear a lot in algebraic geometry. For instance, global sections of the log canonical bundle $\omega_{\mathrm{Sh}_{\mathbb{K},\Sigma}}(Z)$ and its powers make the log canonical ring of Iitaka of $\mathrm{Sh}_{\mathbb{K}}$. This led Mumford to compute the global sections of canonical extensions of general automorphic vector bundles and obtain the degree $0$ case of our theorem as \cite[Proposition 3.3]{Mumford77}. As another example, the groups $H^q(\mathrm{Sh}_{\mathbb{K},\Sigma},\Omega^p(\log{Z}))$ form the $E_1$ page of a spectral sequence computing $H^*(\mathrm{Sh}_\mathbb{K},\mathbb{C})$ and are thus related to the mixed Hodge structures attached to $\mathrm{Sh}_\mathbb{K}$. In fact, Faltings \cite{Faltings83}, \cite{Faltings84} had independently suggested studying the higher cohomology of canonical extensions and has since demonstrated their importance for the study of the Hodge structures attached to Siegel modular forms \cite{Faltings87}.

From the representation theoretic point of view, an advantage of coherent cohomology over the cohomology of automorphic local systems is that the right hand side of (\ref{AGC}) detects some more automorphic representations than the right hand side of (\ref{AGB}). The examples include weight $1$ modular forms and every cuspidal automorphic representation whose archimedean factor is a non-degenerate limit of discrete series. Our result in particular implies the algebraicity of Hecke eigenvalues arising from these automorphic representations. The attachment of Galois representations to those automorphic representations or Hecke eigenclasses in coherent cohomology is an interesting and important project. In degree $0$, Deligne-Serre \cite{DS74} did it for weight $1$ modular forms, and Taylor \cite{Taylor91} took care of Siegel modular forms of low weights. The interest in associating Galois representations to eigenclasses in higher coherent cohomology was recently refuelled by the work \cite{CG12} of Calegari-Geraghty, where the full power of their result partly relies on the existence of those Galois representations. Since then Emerton-Reduzzi-Xiao \cite{ERX14}, Boxer \cite{Boxer15}, Goldring-Koskivirta \cite{GK15} and Pilloni-Stroh \cite{PS16} have carried out the constructions for Hilbert modular varieties, PEL-type and Hodge-type Shimura varieties respectively. Along this line our result helps to verify the automorphy of these Galois representations. We refer the readers to \cite[1.3]{Boxer15} for some more reasons for studying higher coherent cohomology of Shimura varieties.

The proof of our theorem is built on the strategy and machinery developed by Borel and Franke in their work towards (\ref{AGB}). Their proof can be summarized as a trilogy:
\[\begin{aligned}
H^i(S_\mathbb{K},\underline{E})&\cong H^i_{(\mathfrak{g},K)}(C^\infty_{\mathrm{dmg}}(G)^\mathbb{K}\otimes E)\\
&\cong H^i_{(\mathfrak{g},K)}(C^\infty_{\mathrm{umg}}(G)^\mathbb{K}\otimes E)\\
&\cong H^i_{(\mathfrak{g},K)}(\mathcal{A}(G)^\mathbb{K}\otimes E),
\end{aligned}\]
where $C^\infty_{\mathrm{dmg}}(G)^\mathbb{K}$ is cut out in $C^\infty(G(\mathbb{Q})\backslash G(\mathbb{A})/\mathbb{K})^{K\mbox{-}\mathrm{fin}}$ by (non-uniformly) moderate growth conditions on all left-invariant derivatives (see Definition \ref{mg}). Borel proved the first two isomorphisms and Franke proved the third as well as gave a new proof of the second. The four proofs are based on very different ideas and techniques. In the first step Borel extended $\underline{E}$ to a local system on the Borel-Serre compactification of $S_\mathbb{K}$, on which he is able to construct a fine resolution of the local system whose global sections can be identified with the complex computing the right hand side. The key point here is to compare the global growth conditions defining $C^\infty_{\mathrm{dmg}}(G)^\mathbb{K}$ and the local growth conditions characterizing the resolution (which guarantee the fineness and exactness), and the comparison in turn reduces to estimations of the coefficients of invariant differential operators on $S_\mathbb{K}$ together with their derivatives in certain local coordinates. For the second step only Franke's method adapts to our case. His proof is to use certain endomorphisms of $C^\infty_{\mathrm{dmg}}(G)^\mathbb{K}$ to construct a homotopy inverse to the inclusion between the two complexes computing cohomology. To show that this approach works one needs to apply Kuga's formula, and the desired endomorphisms are contructed using the theory of semigroups of operators. The third step is acknowledged to be the hardest, for which Franke considered a left-exact functor $\mathfrak{Fin}_\mathcal{J}$ from the category of $(\mathfrak{g},K)$-modules to itself which preserves injectives, cuts out the same subspace in $C^\infty_{\mathrm{umg}}(G)^\mathbb{K}$ and $\mathcal{A}(G)^\mathbb{K}$ and the composition with which leaves the functor $H^0_{(\mathfrak{g},K)}(-\otimes E)$ invariant. The Grothendieck spectral sequence then suggests that it suffices to show that $C^\infty_{\mathrm{umg}}(G)^\mathbb{K}$ and $\mathcal{A}(G)^\mathbb{K}$ are both $\mathfrak{Fin}_\mathcal{J}$-acyclic, where the latter is trivial. The $\mathfrak{Fin}_\mathcal{J}$-acyclicity of $C^\infty_{\mathrm{umg}}(G)^\mathbb{K}$ is the main result of \cite{Franke98}, whose proof involves a delicate study of the structure of $C^\infty_{\mathrm{umg}}(G)^\mathbb{K}$ based on Langlands' theory of Eisenstein series \cite{Langlands76}.

Our proof of (\ref{AGC}) proceeds in three parallel steps
\[\begin{aligned}
H^i(\mathrm{Sh}_{\mathbb{K},\Sigma},\widetilde{V}^{\mathrm{can}})&\cong H^i_{(\mathfrak{p}_h,K_h)}(C^\infty_{\mathrm{dmg}}(G)^\mathbb{K}\otimes V)\\
&\cong H^i_{(\mathfrak{p}_h,K_h)}(C^\infty_{\mathrm{umg}}(G)^\mathbb{K}\otimes V)\\
&\cong H^i_{(\mathfrak{p}_h,K_h)}(\mathcal{A}(G)^\mathbb{K}\otimes V).
\end{aligned}\]
In the first step the role of the Borel-Serre compactification is now played by a toroidal compactification. As such a compactification is usually not canonical, the local coordinates appearing in the corresponding growth conditions comparison, Proposition \ref{dmglocal}, are somewhat more subtle. Note that there are growth conditions on infinitely many derivatives to compare. We first deal with the growth conditions on functions themselves and prove Proposition \ref{mglocal}, where the key estimates come from Krivine's Positivstellensatz in real algebraic geometry. Then we use a simple ring-theoretic lemma \ref{DerCor} to reduce the remaining comparison to the estimate (\ref{GLdlg}), which essentially follows from the relation (\ref{DDF}) between two decompositions (\ref{D}) and (\ref{DF}) of a Hermitian symmetric domain. The comparison ensures the fineness of the resolution we construct, and then we prove a Dolbeault lemma \ref{dlgr} for the exactness. For the second step we first reduce the problem to the existence of certain endomorphisms of $C^\infty_{\mathrm{umg}}(G)^\mathbb{K}$ in Proposition \ref{RcTc} using the Okamoto-Ozeki formula \ref{OO} in place of Kuga's formula. The endomorphisms we need should be the same as what Franke needed, but as Borel's proof of this step had already existed Franke didn't give a fully-detailed verification of the required properties of the endomorphisms he constructed, and we carry this out in Sections \ref{reg2}-\ref{reg3}. In particular one needs some sufficient conditions making linear operators intertwine semigroups of operators, which we don't find in the literature, so we develop them ourselves in Section \ref{reg2}. Besides, we hope Proposition \ref{Sobolev2} exhibits the contrast between $C^\infty_{\mathrm{dmg}}(G)^\mathbb{K}$ and $C^\infty_{\mathrm{umg}}(G)^\mathbb{K}$ clearly. Finally in the third step we feel fortunate to find out that though the pair $(\mathfrak{g},K)$ changes to $(\mathfrak{p}_h,K_h)$, we are still able to apply Franke's $\mathfrak{Fin}_\mathcal{J}$-acyclicity result without going into its proof. What we do is to find a suitable ideal $\mathcal{J}\subseteq\mathfrak{Z}(\mathfrak{g})$ of finite codimension, bring the forgetful functor into play and show that it behaves well.

Here is an outline of the paper. In section \ref{section1} we recall the construction of automorphic vector bundles, toroidal compactifications and canonical extensions and collect facts about them. Sections \ref{section2}-\ref{section3} make the first step of the proof of the main theorem, where in \ref{sheaves1} we define for every automorphic vector bundle a complex of sheaves on the toroidal compactification and perform some first analysis of it; in \ref{sheaves2} we prove the comparison Proposition \ref{dmglocal}; in \ref{sheaves3} we deal with growth conditions on differential forms; in \ref{fineresol1} we connect the left end of the complex with the canonical extension; in \ref{fineresol2} we prove the exactness of the complex; and in \ref{fineresol3} we clarify the Hecke action. Sections \ref{section4} and \ref{section5} are devoted to the remaining two steps of the proof respectively.

\section*{Notations and conventions}

In this paper, rings are commutative rings with units, complexes are cochain complexes. $\mathbb{A}$\index{$\mathbb{A}$} and $\mathbb{A}_f$\index{$\mathbb{A}_f$} are the ad\`{e}les and finite ad\`{e}les of $\mathbb{Q}$. $\mathcal{A}^i_M$ and $\mathcal{C}^\infty_M:=\mathcal{A}^0_M$ are the sheaves of smooth $i$-forms and functions on a smooth manifold $M$, $\mathcal{A}^{p,q}_M$ is the sheaf of smooth $(p,q)$-forms when $M$ is a complex manifold.

Lie algebras of algebraic groups and Lie groups are denoted by small gothic letters, e.g. if a group $G$ is defined over a field $F$ and $E/F$ is an extension then $\mathfrak{g}_E:=\mathrm{Lie}(G)\otimes_F E\index{$\mathfrak{g}_\mathbb{R}$}$, and $E=\mathbb{C}$\index{$\mathfrak{g}$} will be omitted. The universal enveloping algebra of a Lie algebra $\mathfrak{g}$ and its center are denoted as $\mathfrak{U}(\mathfrak{g})$\index{$\mathfrak{U}(\mathfrak{g})$} and $\mathfrak{Z}(\mathfrak{g})$\index{$\mathfrak{Z}(-)$}.

For a reductive group $G$ over $\mathbb{Q}$, let $A_G$\index{$A_G$} be the maximal split torus in the center of $G$ and $X^*(G)$ be the group of $\mathbb{Q}$-characters of $G$, we denote
\[[G]\index{$[G]$}:=G(\mathbb{Q})\backslash G(\mathbb{A})/A_G(\mathbb{R})^\circ,\]
\[G(\mathbb{A})^1\index{$G(\mathbb{A})^1$}:=\big\{(g_v)\in G(\mathbb{A}):\forall\chi\in X^*(G),\textstyle\prod|\chi(g_v)|_v=1\big\}\]
and $G(\mathbb{R})^1:=G(\mathbb{R})\cap G(\mathbb{A})^1$\index{$G(\mathbb{R})^1$}. \it Smooth functions\rm\label{sm} on adelic groups are defined as follows: a function on an open subset $U\subseteq G(\mathbb{A})$ is smooth if for some compact open subgroup $\mathbb{K}\subseteq G(\mathbb{A}_f)$\index{$\mathbb{K}$} it is pulled back from a smooth function on an open set in the manifold $G(\mathbb{A})/\mathbb{K}$ (in particular, $U$ should be $\mathbb{K}$-invariant). In this paper we often view invariant functions and functions on quotients as the same, like in the above case we may identify $C^\infty(U/\mathbb{K})$ and $C^\infty(U)^\mathbb{K}$.

For a $(\mathfrak{g},K)$-module $M$, we denote
\[C^i(M)\index{$C^i(-)$}=C^i_{(\mathfrak{g},K)}(M):=\mathrm{Hom}_K(\textstyle\bigwedge^i(\mathfrak{g}/\mathfrak{k}),M),\]
where $\bigwedge^i$ means the $i$-th exterior power. We let $d:C^i(M)\rightarrow C^{i+1}(M)$ be the maps given in \cite[2.127b]{KV95}, then the complex $(C^*(M),d)$ computes the $(\mathfrak{g},K)$-cohomology of $M$ and in this paper we call it the \it relative Chevalley-Eilenberg (abbreviated C-E) complex \rm for $H^*_{(\mathfrak{g},K)}(M)$.

In a finite-dimensional real vector space, by a \it cone \rm we mean a convex subset $C$ such that $\mathbb{R}_{>0}C=C$ and that $C$ does not contain an entire line.

The unitary dual of a compact group $K$ is denoted by $\widehat{K}$\index{$\widehat{K}$}\index{$\widehat{K}_o$}. For every $\pi\in\widehat{K}$ we denote the $\pi$-isotypic component of a $K$-representation $M$ as $M_\pi$\index{$(-)_\pi$}, and for every $S\subseteq\widehat{K}$ we write $M_S:=\bigoplus_{\pi\in S}M_\pi\index{$(-)_S$}$ and call it the \it $S$-isotypic part \rm of $M$.

As for discs, we denote
\[\Delta(z_0,\rho)\index{$\Delta(\cdot,\cdot)$}:=\{z\in\mathbb{C}:|z-z_0|<\rho\},\Delta^*(\rho)\index{$\Delta^*(\cdot)$}:=\Delta(0,\rho)\backslash\{0\},\]
\[\Delta^{n,r}(\rho)\index{$\Delta^{n,r}(\cdot)$}:=\Delta^*(\rho)^r\times\Delta(0,\rho)^{n-r},\]
and $\Delta:=\Delta(0,\frac{1}{3})$\index{$\Delta$}, $\Delta^{n,r}:=\Delta^{n,r}(\frac{1}{3})$\index{$\Delta^{n,r}$} for $\rho>0$ and integers $0\leq r\leq n$.

\section{Coherent cohomology of Shimura varieties}\label{section1}

In this section we review the constructions involved in the definition of coherent cohomology of Shimura varieties: automorphic vector bundles, toroidal compactifications and the canonical extensions of automorphic vector bundles over them, with an emphasise on facts that will be needed.

\subsection{Automorphic vector bundles}\label{avb}

Let $(G,X)$ be a Shimura datum. Let $G=A_GM_G$ be the Langlands decomposition, i.e. $A_G$ is the maximal $\mathbb{Q}$-split torus in the center $Z_G$\index{$Z_G$} of $G$, $M_G$\index{$M_G$} is the intersection of the kernels of all $\mathbb{Q}$-characters of $G$. To avoid technical complications, we assume that $G$ is connected and
\[A_G\textrm{ is also the maximal }\mathbb{R}\textrm{-split torus in }Z_G.\label{assum}\tag{$\dagger$}\]
For each compact open subgroup $\mathbb{K}\subseteq G(\mathbb{A}_f)$, the Shimura variety $Sh(G,X)_\mathbb{K}$\index{$Sh(G,X)_\mathbb{K}$} associated to $(G,X)$ at level $\mathbb{K}$ is a quasi-projective variety over a number field $E(G,X)$\index{$E(G,X)$} whose complex analytification is
\[\mathrm{Sh}_\mathbb{K}\index{$\mathrm{Sh}_\mathbb{K}$}:=G(\mathbb{Q})\backslash(X\times G(\mathbb{A}_f)/\mathbb{K}).\]
Constructions related to Shimura varieties are usually indexed by the level $\mathbb{K}$, which we often omit when it is understood. In this paper we only consider levels $\mathbb{K}$ that are \it neat \rm as in \cite[1.1]{Harris89}, in which case $Sh(G,X)_\mathbb{K}$ is smooth and $\mathrm{Sh}_\mathbb{K}$ is a complex manifold.

Fix a point $h\in X$. Recall that $h$ stands for a morphism from $\mathbb{S}=\mathrm{Res}_{\mathbb{C}/\mathbb{R}}\mathbb{G}_m$ to $G_\mathbb{R}$ which induces a Hodge structure on $\mathfrak{g}_\mathbb{R}$:
\[\mathfrak{g}=\mathfrak{g}_\mathbb{C}=\mathfrak{g}^{-1,1}\oplus\mathfrak{g}^{0,0}\oplus\mathfrak{g}^{1,-1},\]
and $X$ is its $G(\mathbb{R})$-conjugacy class. Let $K_h\subseteq G(\mathbb{R})$ be the stabilizer of $h$, then $X\cong G(\mathbb{R})/K_h$ and $\mathfrak{g}^{0,0}=\mathfrak{k}_h$. Let $D$ be the connected component of $X$ containing $h$, $G(\mathbb{R})^+\subseteq G(\mathbb{R})$ and $G(\mathbb{Q})^+\subseteq G(\mathbb{Q})$ be the subgroups that stabilize $D$, then $K_h\subseteq G(\mathbb{R})^+$ and $D\cong G(\mathbb{R})^+/K_h$. We know $G(\mathbb{Q})$ acts on $\pi_0(X)$ transitively \cite[2.1.2]{Deligne77}, hence $\mathrm{Sh}_\mathbb{K}\cong G(\mathbb{Q})^+\backslash(D\times G(\mathbb{A}_f)/\mathbb{K})$. Let $\{\gamma\}$ be a set of representatives of $G(\mathbb{Q})^+\backslash G(\mathbb{A}_f)/\mathbb{K}$, then
\begin{equation}\label{gamma}
\mathrm{Sh}_\mathbb{K}\cong\coprod_{\{\gamma\}}\Gamma(\gamma)\backslash D
\end{equation}
where $\Gamma(\gamma):=G(\mathbb{Q})^+\cap\gamma\mathbb{K}\gamma^{-1}$. $\mathbb{K}$ being neat implies that each $\Gamma(\gamma)$ is \it neat \rm as in \cite[17.1]{Borel69}.

Denote $\mathfrak{p}_+:=\mathfrak{g}^{-1,1},\mathfrak{p}_-:=\mathfrak{g}^{1,-1}$ and $\mathfrak{p}_{h}:=\mathfrak{k}_h\oplus\mathfrak{p}_-$, then $\mathfrak{p}_{h}$ is a parabolic subalgebra of $\mathfrak{g}$ with nilpotent radical $\mathfrak{p}_-$. Let $P_h$ be the parabolic subgroup of $G_\mathbb{C}$ with Lie algebra $\mathfrak{p}_{h}$, then $K_h\subseteq P_h(\mathbb{C})$ and $\widecheck{D}:=G(\mathbb{C})/P_h(\mathbb{C})$ is isomorphic to the compact dual of $D$. We know that the unique $G(\mathbb{R})$-equivariant map $\beta:X\rightarrow\widecheck{D}$ sending $h$ to the coset $P_h(\mathbb{C})$ is an open immersion of complex manifolds.

Let $(\rho,V)$ be any finite-dimensional holomorphic representation of $P_h(\mathbb{C})$, then it defines a $G(\mathbb{C})$-homogeneous holomorphic vector bundle
\[\mathcal{E}_V:=G(\mathbb{C})\times^{P_h(\mathbb{C})}V\rightarrow\widecheck{D},\]
and $\beta^*(\mathcal{E}_V)\times(G(\mathbb{A}_f)/\mathbb{K})$ is a $G(\mathbb{A})$-homogeneous holomorphic vector bundle over $X\times(G(\mathbb{A}_f)/\mathbb{K})$. Under the assumption (\ref{assum}), when $\mathbb{K}$ is neat $G(\mathbb{Q})$ acts freely on $X\times(G(\mathbb{A}_f)/\mathbb{K})$, therefore
\begin{equation}\label{autovb}
\widetilde{V}=\widetilde{V}_\mathbb{K}:=G(\mathbb{Q})\backslash(\beta^*(\mathcal{E}_V)\times G(\mathbb{A}_f)/\mathbb{K})
\end{equation}
is a holomorphic vector bundle over $\mathrm{Sh}_\mathbb{K}$, which is called an \it automorphic vector bundle\rm. It is shown in \cite[4.8]{Harris85} that $\widetilde{V}$ is the analytification of an algebraic vector bundle defined over a number field $k_V$ determined by $V$.

\begin{Ex}
$V\cong\bigwedge^p(\mathfrak{g}/\mathfrak{p}_h)^*\Rightarrow\mathcal{E}_V\cong\Omega^p_{\widecheck{D}}$ and $\widetilde{V}\cong\Omega^p_{\mathrm{Sh}_\mathbb{K}}$.
\end{Ex}

If $g\in G(\mathbb{A}_f)$ and $\mathbb{K}'\subseteq g\mathbb{K}g^{-1}$ is an open subgroup, denote by $t_g$ the composition of the map $\mathrm{Sh}_{g\mathbb{K}g^{-1}}\rightarrow\mathrm{Sh}_\mathbb{K}$ induced by right translation by $g$ and the covering map $\mathrm{Sh}_{\mathbb{K}'}\rightarrow\mathrm{Sh}_{g\mathbb{K}g^{-1}}$, then by construction there is a natural isomorphism $t_g^*\widetilde{V}_\mathbb{K}\cong\widetilde{V}_{\mathbb{K}'}$.

\subsubsection{Automorphic vector bundles as sheaves}

For computational purpose it is favorable to write $\widetilde{V}$ as a sheaf. Let $\pi_\mathbb{C}:G(\mathbb{C})\rightarrow\widecheck{D}$ be the quotient map, then by definition the smooth and holomorphic sections of $\mathcal{E}_V$ over an open subset $U\subseteq\widecheck{D}$ are $(C^\infty(\pi_\mathbb{C}^{-1}(U))\otimes V)^{P_h(\mathbb{C})}$ and $(\mathcal{O}(\pi_\mathbb{C}^{-1}(U))\otimes V)^{P_h(\mathbb{C})}$ respectively, where $\mathcal{O}(\pi_\mathbb{C}^{-1}(U))$ stands for holomorphic functions on $\pi_\mathbb{C}^{-1}(U)$. When $U$ is small enough so that the principal $P_h(\mathbb{C})$-bundle $\pi_\mathbb{C}$ is trivial over $U$, a trivialisation
\[\pi_\mathbb{C}^{-1}(U)\xrightarrow[\sim]{(t_\mathbb{C},\pi_\mathbb{C})}P_h(\mathbb{C})\times U\]
gives rise to a $P_h(\mathbb{C})$-equivariant holomorphic map $t_\mathbb{C}:\pi_\mathbb{C}^{-1}(U)\rightarrow P_h(\mathbb{C})$. Then for each open subset $W\subseteq U$ there is an isomorphism
\begin{equation}\label{tc}
\begin{gathered}
(C^\infty(\pi_\mathbb{C}^{-1}(W))\otimes V)^{P_h(\mathbb{C})}\xrightarrow{\sim}C^\infty(\pi_\mathbb{C}^{-1}(W)/P_h(\mathbb{C}))\otimes V,\\
f\mapsto t_\mathbb{C}\cdot f:=(g\mapsto\rho(t_\mathbb{C}(g))f(g)),
\end{gathered}
\end{equation}
wherein $(\mathcal{O}(\pi_\mathbb{C}^{-1}(W))\otimes V)^{P_h(\mathbb{C})}\xrightarrow{\sim}\mathcal{O}(\pi_\mathbb{C}^{-1}(W)/P_h(\mathbb{C}))\otimes V$ since $t_\mathbb{C}$ is holomorphic. Let $\mathfrak{g}$ act on $C^\infty(\pi_\mathbb{C}^{-1}(W))$ by right differentiation, then $\mathcal{O}(\pi_\mathbb{C}^{-1}(W)/P_h(\mathbb{C}))$ consists of smooth functions $h$ on $\pi_\mathbb{C}^{-1}(W)/P_h(\mathbb{C})$ such that
\[Xh+iX'h=0,\forall X\in\mathfrak{g}/\mathfrak{p}_h,X'\in (iX+\mathfrak{p}_h)/\mathfrak{p}_h.\]
As $\mathfrak{g}_\mathbb{R}\twoheadrightarrow\mathfrak{g}_\mathbb{R}/\mathfrak{k}_h\xrightarrow{\sim}\mathfrak{g}/\mathfrak{p}_h$, the above condition is equivalent to
\begin{equation}\label{hol1}
Xh+iX'h=0,\forall X\in\mathfrak{g}_\mathbb{R},X'\in (iX+\mathfrak{p}_h)\cap\mathfrak{g}_\mathbb{R}.
\end{equation}
Now suppose $U\subseteq D$ and let $\pi_\mathbb{R}$ be the quotient map $G(\mathbb{R})^+\rightarrow D$, then since $\pi_\mathbb{C}$ and $\pi_\mathbb{R}$ are principal $P_h(\mathbb{C})$ and $K_h$-bundles respectively, the restriction map induces an isomorphism
\[(C^\infty(\pi_\mathbb{C}^{-1}(W))\otimes V)^{P_h(\mathbb{C})}\cong(C^\infty(\pi_\mathbb{R}^{-1}(W))\otimes V)^{K_h}.\]
By (\ref{tc}) and (\ref{hol1}), elements of $(\mathcal{O}(\pi_\mathbb{C}^{-1}(W))\otimes V)^{P_h(\mathbb{C})}$ in the left hand side correspond to functions $f$ in the right hand side such that
\begin{equation}\label{hol2}
X(t_\mathbb{C}\cdot f)+iX'(t_\mathbb{C}\cdot f)=0,\forall X\in\mathfrak{g}_\mathbb{R},X'\in (iX+\mathfrak{p}_h)\cap\mathfrak{g}_\mathbb{R},
\end{equation}
where $t_\mathbb{C}$ is temporarily restricted to $\pi_\mathbb{R}^{-1}(W)$. As $X+iX'\in\mathfrak{p}_h$, we have
\begin{equation}\label{hol3}
\begin{gathered}
X(t_\mathbb{C}\cdot f)+iX'(t_\mathbb{C}\cdot f)=t_\mathbb{C}\cdot (Xf+iX'f)+(Xt_\mathbb{C}+iX't_\mathbb{C})\cdot f,\\
Xt_\mathbb{C}+iX't_\mathbb{C}=(X+iX')t_\mathbb{C},\\
((X+iX')t_\mathbb{C})\cdot f=t_\mathbb{C}\cdot ((d\rho(X+iX'))f),
\end{gathered}
\end{equation}
where in the second equality $(X+iX')$ is view as a real tangent vector field, and the equality holds as $t_\mathbb{C}$ is holomorphic; in the third equality $d\rho:\mathfrak{p}_h\rightarrow\mathrm{End}(V)$ is the derivative of $\rho:P_h(\mathbb{C})\rightarrow\mathrm{GL}(V)$, $d\rho(X+iX')$ acts on the value of $f$ and the equality holds as $t_\mathbb{C}$ is $P_h(\mathbb{C})$-equivariant. If we define the $\mathfrak{g}$-action on $C^\infty(\pi_\mathbb{R}^{-1}(W))$ as the $\mathbb{C}$-linear extension of the $\mathfrak{g}_\mathbb{R}$-action by right differentiation, then (\ref{hol2}) means that $f$ is annihilated by $\mathfrak{p}_h$ under the diagonal action, i.e. we have
\[\mathcal{E}_V(W)\cong(C^\infty(\pi_\mathbb{R}^{-1}(W))\otimes V)^{(\mathfrak{p}_h,K_h)}.\]
This isomorphism is independent of the trivialisation, so by gluing we get $\mathcal{E}_V|_D$ is the sheaf $U\mapsto(C^\infty(\pi_\mathbb{R}^{-1}(U))\otimes V)^{(\mathfrak{p}_h,K_h)}$. Doing the same for other connected components of $X$, we see that $\beta^*(\mathcal{E}_V)$ is the sheaf $U\mapsto(C^\infty(\pi_\mathbb{R}^{-1}(U))\otimes V)^{(\mathfrak{p}_h,K_h)}$, where $\pi_\mathbb{R}$ has been extended to the quotient map $G(\mathbb{R})\rightarrow X$. Combining with (\ref{autovb}) we get $\widetilde{V}$ is the sheaf
\[U\mapsto(C^\infty(\pi^{-1}(U))\otimes V)^{(\mathfrak{p}_h,K_h)},\]
where $\pi=\pi_\mathbb{K}:G(\mathbb{Q})\backslash G(\mathbb{A})/\mathbb{K}\rightarrow\mathrm{Sh}_\mathbb{K}$ is the quotient map.

With this expression, the morphism $\widetilde{V}_\mathbb{K}\rightarrow t_{g*}\widetilde{V}_{\mathbb{K}'}$ adjoint to $t_g^*\widetilde{V}_\mathbb{K}\cong\widetilde{V}_{\mathbb{K}'}$ is induced from maps
\[\begin{aligned}
C^\infty(\pi_\mathbb{K}^{-1}(U))\otimes V&\hookrightarrow C^\infty(\pi_{\mathbb{K}'}^{-1}(t_g^{-1}(U)))\otimes V,\\
f&\mapsto f(\cdot\,g),
\end{aligned}\]
where $f$ is viewed as a function on an open subset of $G(\mathbb{A})$. Now let
\[\pi_o=\pi_{o,\mathbb{K}}\textrm{ be the quotient map }[G]/\mathbb{K}=G(\mathbb{Q})\backslash G(\mathbb{A})/A_G(\mathbb{R})^\circ\mathbb{K}\rightarrow\mathrm{Sh}_\mathbb{K},\]
\[\mathfrak{p}_o:=\mathfrak{p}_h\cap\mathfrak{m}_G,\]
\[K_o\textrm{ be the maximal compact subgroup of }K_h,\]
then $\mathfrak{p}_h=\mathfrak{p}_o\oplus\mathfrak{a}_G$ and thanks to (\ref{assum}), $K_h=K_o A_G(\mathbb{R})^\circ$. For every $g\in G(\mathbb{A})$, let $a_g$ be the unique element in $A_G(\mathbb{R})^\circ$ such that $g\in a_g G(\mathbb{A})^1$. Let $G(\mathbb{A}_f)$ act on $V$ as $gv:=a_g^{-1}v$ and denote the $P_h(\mathbb{C})\times G(\mathbb{A}_f)$-module we get as $V_o$. Then there are isomorphisms
\[\begin{gathered}
(C^\infty(\pi^{-1}(U))\otimes V)^{A_G(\mathbb{R})^\circ}\xrightarrow{\sim}C^\infty(\pi_o^{-1}(U))\otimes V_o,\\
f\mapsto(g\mapsto f(a_g^{-1}g)),
\end{gathered}\]
which induce
\[\widetilde{V}_\mathbb{K}(U)\cong(C^\infty(\pi^{-1}(U))\otimes V)^{(\mathfrak{p}_h,K_h)}\cong(C^\infty(\pi_o^{-1}(U))\otimes V_o)^{(\mathfrak{p}_o,K_o)}.\]
These isomorphisms are compatible with restrictions, and with the new expression $\widetilde{V}_\mathbb{K}\rightarrow t_{g*}\widetilde{V}_{\mathbb{K}'}$ is induced from maps
\begin{equation}\label{Vo}
\begin{gathered}
C^\infty(\pi_{o,\mathbb{K}}^{-1}(U))\otimes V_o\hookrightarrow C^\infty(\pi_{o,\mathbb{K}'}^{-1}(t_g^{-1}(U)))\otimes V_o,\\
f\mapsto gf(\cdot\,g).
\end{gathered}
\end{equation}

\subsubsection{Cohomology of $\widetilde{V}$}

To demonstrate the usage of the expressions above, we compute the cohomology of $\widetilde{V}$, which is a toy model of coherent cohomology of Shimura varieties. We first note that $\widetilde{V}_\mathbb{K}\rightarrow t_{g*}\widetilde{V}_{\mathbb{K}'}$ induces maps
\[t_g^*:H^i(\mathrm{Sh}_\mathbb{K},\widetilde{V}_\mathbb{K})\rightarrow H^i(\mathrm{Sh}_{\mathbb{K}'},\widetilde{V}_{\mathbb{K}'}),\]
and hence $\varinjlim_{t_e^*}H^i(\mathrm{Sh}_\mathbb{K},\widetilde{V}_\mathbb{K})$ has a $G(\mathbb{A}_f)$-module structure.

Fix a level $\mathbb{K}$. Tensoring $\widetilde{V}$ with the Dolbeault complex $(\mathcal{A}^{0,*}_{\mathrm{Sh}_\mathbb{K}},\overline\partial)$ of $\mathrm{Sh}_\mathbb{K}$ gives a fine resolution of it. The previous analysis exhibits that $\widetilde{V}\otimes_{\mathcal{O}_{\mathrm{Sh}_\mathbb{K}}}\mathcal{C}^\infty_{\mathrm{Sh}_\mathbb{K}}$ is the sheaf
\[U\mapsto(C^\infty(\pi^{-1}(U))\otimes V)^{K_h}\cong(C^\infty(\pi_o^{-1}(U))\otimes V)^{K_o}.\]
For any $i\geq 0$, $\varpi\in(\widetilde{V}\otimes_{\mathcal{O}_{\mathrm{Sh}_\mathbb{K}}}\mathcal{A}^i_{\mathrm{Sh}_\mathbb{K}})(U)$ and tangent vectors $X_1,...,X_i\in(T_x U)_\mathbb{C}$ at a point $x$, $\varpi(\bigwedge_{j=1}^iX_j)$ can be viewed as an element of $(C^\infty(\pi_o^{-1}(x))\otimes V)^{K_o}$. Define
\[\varphi:(\widetilde{V}\otimes_{\mathcal{O}_{\mathrm{Sh}_\mathbb{K}}}\mathcal{A}^i_{\mathrm{Sh}_\mathbb{K}})(U)\rightarrow\mathrm{Hom}(\textstyle\bigwedge^i(\mathfrak{g}/\mathfrak{a}_G),C^\infty(\pi_o^{-1}(U))\otimes V)\]
as follows: for every $y\in\pi_o^{-1}(U)$, identify $T_y\pi_o^{-1}(U)$ with $\mathfrak{g}_\mathbb{R}/\mathfrak{a}_{G,\mathbb{R}}$ via left-invariant vector fields on $G(\mathbb{R})$, then for any $\varpi$ and $Y_1,...,Y_i\in\mathfrak{g}/\mathfrak{a}_G\cong(T_y\pi_o^{-1}(U))_\mathbb{C}$, define
\begin{equation}\label{Ai}
\varphi(\varpi)(Y_1\!\wedge\!...\!\wedge\!Y_i)(y):=\varpi(d\pi_{o,y}Y_1\!\wedge\!...\!\wedge\!d\pi_{o,y}Y_i)(y)\in V.
\end{equation}
We see that if some $Y_j\in\mathfrak{k}_h/\mathfrak{a}_G$, then $d\pi_{o,y}Y_j\equiv 0$ and hence $\varphi(\varpi)(\bigwedge_{j=1}^i Y_j)=0$, so the image of $\varphi$ belongs to
\[\mathrm{Hom}(\textstyle\bigwedge^i(\mathfrak{g}/\mathfrak{k}_h),C^\infty(\pi_o^{-1}(U))\otimes V).\]
Moreover for any $y\in\pi_o^{-1}(U),Y\in\mathfrak{g}/\mathfrak{a}_G$ and $k\in K_o$,
\[Y\in(T_y\pi_o^{-1}(U))_\mathbb{C}\textrm{ and }\mathrm{Ad}(k)Y\in(T_{yk^{-1}}\pi_o^{-1}(U))_\mathbb{C}\]
push forward to a same tangent vector in $(T_{\pi_o(y)}U)_\mathbb{C}$, so the image of $\varphi$ belongs to
\[\mathrm{Hom}_{K_o}(\textstyle\bigwedge^i(\mathfrak{g}/\mathfrak{k}_h),C^\infty(\pi_o^{-1}(U))\otimes V).\]
As $\pi_o$ is a principal $K_o$-bundle, $\varphi$ actually induces an isomorphism
\[(\widetilde{V}\otimes_{\mathcal{O}_{\mathrm{Sh}_\mathbb{K}}}\mathcal{A}^i_{\mathrm{Sh}_\mathbb{K}})(U)\cong\mathrm{Hom}_{K_o}(\textstyle\bigwedge^i(\mathfrak{g}/\mathfrak{k}_h),C^\infty(\pi_o^{-1}(U))\otimes V).\]
At each $y\in\pi_o^{-1}(U)$ the decomposition $\mathfrak{g}/\mathfrak{k}_h\cong\mathfrak{p}_+\oplus\mathfrak{p}_-$ corresponds to
\[(T_{\pi_o(y)}U)_\mathbb{C}\cong T^{1,0}_{\pi_o(y)}U\oplus T^{0,1}_{\pi_o(y)}U,\]
so $\varphi$ also induces an isomorphism
\begin{equation}\label{A0i}
\begin{aligned}
(\widetilde{V}\otimes_{\mathcal{O}_{\mathrm{Sh}_\mathbb{K}}}\mathcal{A}^{0,i}_{\mathrm{Sh}_\mathbb{K}})(U)&\cong\mathrm{Hom}_{K_o}(\textstyle\bigwedge^i\mathfrak{p}_-,C^\infty(\pi_o^{-1}(U))\otimes V)\\
&\cong\mathrm{Hom}_{K_o}(\textstyle\bigwedge^i(\mathfrak{p}_o/\mathfrak{k}_o),C^\infty(\pi_o^{-1}(U))\otimes V),
\end{aligned}
\end{equation}
where the right hand side is exactly the $i$-th term of the relative Chevalley-Eilenberg complex for $H^*_{(\mathfrak{p}_o,K_o)}(C^\infty(\pi_o^{-1}(U))^{K_o\mbox{-}\mathrm{fin}}\otimes V)$, and by \cite[4.2.3]{Harris90} we have the following commutative diagram:
\begin{equation}\label{partial}
\begin{tikzcd}
(\widetilde{V}\otimes_{\mathcal{O}_{\mathrm{Sh}_\mathbb{K}}}\mathcal{A}^{0,i}_{\mathrm{Sh}_\mathbb{K}})(U)\rar{\overline\partial}\dar{\sim}&(\widetilde{V}\otimes_{\mathcal{O}_{\mathrm{Sh}_\mathbb{K}}}\mathcal{A}^{0,i+1}_{\mathrm{Sh}_\mathbb{K}})(U)\dar{\sim}\\
\mathrm{Hom}_{K_o}(\bigwedge^i(\mathfrak{p}_o/\mathfrak{k}_o),C^\infty\otimes V)\rar{d}&\mathrm{Hom}_{K_o}(\bigwedge^{i+1}(\mathfrak{p}_o/\mathfrak{k}_o),C^\infty\otimes V),
\end{tikzcd}
\end{equation}
where $C^\infty=C^\infty(\pi_o^{-1}(U))^{K_o\mbox{-}\mathrm{fin}}$ and $d$ is the differential in the relative C-E complex. Therefore we get
\begin{equation}\label{CinftyC}
\begin{aligned}
H^i(\mathrm{Sh}_\mathbb{K},\widetilde{V})&\cong H^i_{(\mathfrak{p}_o,K_o)}(C^\infty(\pi_o^{-1}(\mathrm{Sh}_\mathbb{K}))^{K_o\mbox{-}\mathrm{fin}}\otimes V)\\
&=H^i_{(\mathfrak{p}_o,K_o)}(C^\infty([G]/\mathbb{K})^{K_o\mbox{-}\mathrm{fin}}\otimes V).
\end{aligned}
\end{equation}
Taking direct limit, we get an isomorphism
\[\varinjlim_{t_e^*}H^i(\mathrm{Sh}_\mathbb{K},\widetilde{V}_\mathbb{K})\cong H^i_{(\mathfrak{p}_o,K_o)}(C^\infty([G])^{K_o\mbox{-}\mathrm{fin}}\otimes V_o)\]
which by (\ref{Vo}) is $G(\mathbb{A}_f)$-equivariant when $C^\infty([G])^{K_o\mbox{-}\mathrm{fin}}\otimes V_o$ is endowed with the diagonal $G(\mathbb{A}_f)$-action.

\subsection{Toroidal compactifications}

Toroidal compactifications were initially constructed for general locally symmetric varieties in \cite{AMRT75} (following the work of Igusa on Siegel modular varieties and Hirzebruch on Hilbert modular surfaces), and in the Shimura variety case \cite{Harris89} and \cite{Pink90} constructed canonical models of them over the reflex field. In this paper we mostly work around a single point, so here we mainly follow \cite{AMRT10} to describe the local structure of these compactifications.

Let $\Gamma\backslash D$ be a connected component of $\mathrm{Sh}_\mathbb{K}$ as in (\ref{gamma}). Recall that $D$ is $G(\mathbb{R})^\circ$-equivariantly embedded in $\widecheck{D}=G(\mathbb{C})/P_h(\mathbb{C})$ by the Borel embedding $\beta$, let $\overline{D}$ be the closure of $D$ in $\widecheck{D}$, then a \it boundary component \rm $F$ of $D$ is a maximal connected analytic submanifold of $\overline{D}$ (in particular, $D$ is a boundary component of itself). By \cite[Proposition III.3.6]{AMRT10} there is a parabolic subgroup $\mathscr{P}(F)\subseteq G_\mathbb{R}$ such that
\[\mathscr{P}(F)(\mathbb{R})\cap G(\mathbb{R})^\circ=P(F):=\big\{g\in G(\mathbb{R})^\circ:gF=F\big\}.\]
$F$ is said to be \it rational \rm if $\mathscr{P}(F)$ is defined over $\mathbb{Q}$; when $F$ runs through rational boundary components of $D$, $\mathscr{P}(F)$ runs through products of maximal $\mathbb{Q}$-parabolics of the $\mathbb{Q}$-factors of $G$ (we refer to a group itself also as a maximal parabolic).

In \cite[Theorem III.3.10]{AMRT10} a Levi decomposition of $\mathscr{P}(F)$ according to the choice of $h\in X$ is defined, which we denote as $\mathscr{P}(F)=\mathscr{W}(F)\rtimes\mathscr{L}(F)$. As in \cite[III.4.1]{AMRT10} the Levi subgroup $\mathscr{L}(F)$ decomposes further as an almost direct product of connected subgroups:
\begin{equation}\label{GhGl}
\mathscr{L}(F)=\mathscr{G}_h(F)\cdot\mathscr{G}_l(F)\cdot\mathscr{M}(F)\cdot Z_{G_\mathbb{R}}^\circ.
\end{equation}
This decomposition can be characterized as follows: let $G_h(F),G_l(F)$ and $M(F)$ be the identity component of the real points of $\mathscr{G}_h(F),\mathscr{G}_l(F)$ and $\mathscr{M}(F)$ respectively, then $G_l(F),M(F)$ act trivially on $F$ while $G_h(F)$ maps onto $\mathrm{Aut}(F)^\circ$ with a finite kernel; $M(F)$ is compact while $G_l(F),G_h(F)$ have no compact factors. Write $W(F):=\mathscr{W}(F)(\mathbb{R})$, then the map
\[\begin{gathered}
W(F)\rtimes(G_h(F)\cdot G_l(F))\rightarrow D,\\
wg_hg_l\mapsto wg_hg_l\cdot h
\end{gathered}\]
induces a diffeomorphism
\begin{equation}\label{D}
D\cong (W(F)\rtimes(G_h(F)\cdot G_l(F)))/((G_h(F)\cap K_o)\cdot(G_l(F)\cap K_o)),
\end{equation}
where $G_h(F)\cap K_o$ and $G_l(F)\cap K_o$ are maximal compact subgroups of $G_h(F)$ and $G_l(F)$ \cite[III.4.3]{AMRT10}.

Let $\mathscr{U}(F)$ be the center of $\mathscr{W}(F)$, then $U(F):=\mathscr{U}(F)(\mathbb{R}),U(F)_\mathbb{C}:=\mathscr{U}(F)(\mathbb{C})$ are vector spaces. By \cite[Theorem III.4.1(1)]{AMRT10}, $\mathscr{G}_h(F),\mathscr{M}(F)$ commute with $\mathscr{U}(F)$. Let $\omega_F\in\mathfrak{u}(F)_\mathbb{R}$ be the element defined in \cite[III.4.2]{AMRT10}, then by \cite[Theorem III.4.1(2)]{AMRT10} the centralizer of $\exp{\omega_F}\in U(F)$ in $G_l(F)$ equals $G_l(F)\cap K_o$ and its $G_l(F)$-orbit $C(F)$ is an open cone in $U(F)$. Furthermore by \cite[Proposition III.3.11]{AMRT10}, $\exp(-i\omega_F)h\in\widecheck{D}$ is fixed by $G_l(F)$. Set
\[D(F):=U(F)_\mathbb{C}D=\bigcup_{g\in U(F)_\mathbb{C}}gD\subseteq\widecheck{D},\]
then the map
\[\begin{gathered}
(iU(F)\times W(F))\rtimes G_h(F)\rightarrow D(F),\\
uwg_h\mapsto uwg_h\cdot\exp(-i\omega_F)h
\end{gathered}\]
induces a diffeomorphism
\begin{equation}\label{DF}
D(F)\cong((iU(F)\times W(F))\rtimes G_h(F))/(G_h(F)\cap K_o)
\end{equation}
\cite[Lemma III.4.6]{AMRT10}. Particularly $U(F)_\mathbb{C}$ acts freely on $D(F)$, and the decomposition (\ref{DF}) is good for exhibiting this principal homogeneous structure. If we compose (\ref{D}) with the inclusion $D\subseteq D(F)$, then the image of $wg_hg_l$ is
\begin{equation}\label{DDF}
\begin{split}
wg_hg_l\cdot h&=wg_hg_l\exp(i\omega_F)\cdot\exp(-i\omega_F)h\\
&=\exp(i\mathrm{Ad}(g_l)\omega_F)wg_hg_l\cdot\exp(-i\omega_F)h\\
&=\exp(i\mathrm{Ad}(g_l)\omega_F)wg_h\cdot\exp(-i\omega_F)h,
\end{split}
\end{equation}
i.e. the image of $\exp(i\mathrm{Ad}(g_l)\omega_F)wg_h\in(iU(F)\times W(F))\rtimes G_h(F)$ under (\ref{DF}).

When $F$ is rational, $U(F)_\mathbb{Z}:=\Gamma\cap U(F)$ is an arithmetic subgroup of $U(F)$ and hence a lattice in the real vector space, $T(F):=U(F)_\mathbb{Z}\backslash U(F)_\mathbb{C}$ is a complex torus with cocharacter group $U(F)_\mathbb{Z}$, and $U(F)_\mathbb{Z}\backslash D(F)$ is a principal $T(F)$-bundle over $D(F)':=U(F)_\mathbb{C}\backslash D(F)$. Let $\sigma\subseteq U(F)$ be a \it rational polyhedral cone\rm, i.e. a cone generated by finitely many elements of $U(F)_\mathbb{Z}$, then as in \cite[I.1]{AMRT10} it gives rise to a torus embedding $T(F)\hookrightarrow T_\sigma$ which induces an embedding of fiber bundles over $D(F)'$:
\begin{equation}\label{p}
\begin{tikzcd}
U(F)_\mathbb{Z}\backslash D\rar[hook]{j_\sigma}\drar{p_F}&T_\sigma\times^{T(F)}(U(F)_\mathbb{Z}\backslash D(F))\dar{p_\sigma}\\
&D(F)'.
\end{tikzcd}
\end{equation}
Let $(U(F)_\mathbb{Z}\backslash D)_\sigma$ be the interior of closure of $U(F)_\mathbb{Z}\backslash D$ in the right hand side, then the main theorem of \cite{AMRT10} asserts that for suitable collections of $\sigma$ these $(U(F)_\mathbb{Z}\backslash D)_\sigma$ glue together to give compactifications of $\Gamma\backslash D$. Denote by $\overline{\Gamma}_F$ the group of automorphisms of $U(F)$ induced by $\Gamma\cap P(F)$, then it acts on $C(F)$. Let $\Sigma:=\coprod_F\Sigma_F$, where $F$ runs through all rational boundary components of $D$, each $\Sigma_F$ is a \it $\overline{\Gamma}_F$-admissible polyhedral decomposition \rm of $C(F)$ as in \cite[Definition II.4.10]{AMRT10} and they satisfy the compatibility conditions in \cite[Definition II.5.1]{AMRT10}, then $\Sigma$ is called a \it $\Gamma$-admissible collection of polyhedra \rm and we have:

\begin{Prop}\cite[Theorem III.5.2]{AMRT10}
There exists a unique Hausdorff complex analytic space $(\Gamma\backslash D)_\Sigma$ containing $\Gamma\backslash D$ as an open dense subset such that, for every rational boundary component $F$ of $D$, there are open analytic morphisms $(\phi_\sigma)_{\sigma\in\Sigma_F}$ making the following diagrams commutative:
\begin{equation}\label{phi}
\begin{tikzcd}
U(F)_\mathbb{Z}\backslash D\rar[hook]{j_\sigma}\dar{\phi_F}&(U(F)_\mathbb{Z}\backslash D)_\sigma\dar[dashed]{\phi_\sigma}\\
\Gamma\backslash D\rar[hook]{j}&(\Gamma\backslash D)_\Sigma
\end{tikzcd}
\end{equation}
and such that every point of $(\Gamma\backslash D)_\Sigma$ is in the image of one of the maps $\phi_\sigma$. Furthermore, $(\Gamma\backslash D)_\Sigma$ is compact.
\end{Prop}

Such a $(\Gamma\backslash D)_\Sigma$ is called a \it toroidal compactification \rm of $\Gamma\backslash D$. When $\Gamma$ is neat, each $\phi_\sigma$ is a local homeomorphism \cite[III.7]{AMRT10}, therefore if every $\sigma\in\Sigma_F\subseteq\Sigma$ is generated by a subset of a basis of $U(F)_\mathbb{Z}$ then $(\Gamma\backslash D)_\Sigma$ is smooth and $(\Gamma\backslash D)_\Sigma-\Gamma\backslash D$ is a normal crossings divisor (i.e. around each point in $(\Gamma\backslash D)_\Sigma$ there exist local coordinates $z_1,...,z_n$ such that $(\Gamma\backslash D)_\Sigma-\Gamma\backslash D$ is locally defined by the equation $z_1...z_r=0$ for some $0\leq r\leq n$), in which case $\Sigma$ and the compactification will be called \it SNC\rm.

Let $\Sigma'$ be a \it refinement \rm of $\Sigma$, i.e. $\Sigma'$ is a $\Gamma$-admissible collection of polyhedra such that every $\sigma'\in\Sigma'_F$ is contained in a $\sigma\in\Sigma_F$, then the identity on $\Gamma\backslash D$ extends to a proper surjection $(\Gamma\backslash D)_{\Sigma'}\rightarrow(\Gamma\backslash D)_\Sigma$ \cite[2.3.3]{Harris89}. We see that for any two $\Gamma$-admissible collections of polyhedra $\Sigma$ and $\Sigma'$, $\Sigma''_F:=\{\sigma\cap\sigma':\sigma\in\Sigma_F,\sigma'\in\Sigma'_F\}$ is a $\overline{\Gamma}_F$-admissible polyhedral decomposition of $C(F)$ for each $F$, and $\Sigma'':=\coprod_F\Sigma''_F$ is a common refinement of $\Sigma$ and $\Sigma'$. \cite[Corollary III.7.6]{AMRT10} shows that every $\Gamma$-admissible collection of polyhedra $\Sigma$ admits an SNC refinement $\Sigma'$ such that $(\Gamma\backslash D)_{\Sigma'}\rightarrow(\Gamma\backslash D)_\Sigma$ is projective. \cite[IV.2.1]{AMRT10} shows that there exists a $\Sigma$ such that $(\Gamma\backslash D)_\Sigma$ is projective, hence $\Gamma\backslash D$ has a projective SNC toroidal compactification.

\subsubsection{$\sigma$-neighborhoods and $\sigma$-coordinates}

Suppose $(\Gamma\backslash D)_\Sigma$ is SNC, then each point of it is in the image of $\phi_\sigma$ for a $\sigma$ generated by a basis of $U(F)_\mathbb{Z}$, say $\{\xi_i\}_{1\leq i\leq r}$. Let $\{l_j:U(F)_\mathbb{C}\rightarrow\mathbb{C}\}_{1\leq j\leq r}$ be dual to $\{\xi_i\}_{1\leq i\leq r}$, $q_j:=e^{2\pi il_j}(1\leq j\leq r)$, then $T(F)\hookrightarrow T_\sigma$ is the embedding $\mathbb{C}_{q_1}^\times\times...\times\mathbb{C}_{q_r}^\times\hookrightarrow\mathbb{C}_{q_1}\times...\times\mathbb{C}_{q_r}$. Let $T(F)\rightarrow iU(F)$ be the map descended from taking imaginary parts of elements in $U(F)_\mathbb{C}$, then it induces a $T(F)$-action on $iU(F)$ and we have $T_\sigma\times^{T(F)}iU(F)\cong\mathbb{R}_{\geq 0}^r$ via the following commutative diagram:
\[\begin{tikzcd}
T(F)\cong\mathbb{C}_{q_1}^\times\times...\times\mathbb{C}_{q_r}^\times\rar[hook]\dar&T_\sigma\cong\mathbb{C}_{q_1}\times...\times\mathbb{C}_{q_r}\dar{(|q_1|,...,|q_r|)}\\
iU(F)\rar[hook]{j_q:=(q_1,...,q_r)}&\mathbb{R}_{\geq 0}^r.
\end{tikzcd}\]
Let $y_F:U(F)_\mathbb{Z}\backslash D(F)\rightarrow iU(F)$ be descended from projection to the first factor in (\ref{DF}). Apply $T_\sigma\times^{T(F)}$ to it, we get
\begin{equation}\label{y}
\begin{tikzcd}
U(F)_\mathbb{Z}\backslash D(F)\rar[hook]{j_\sigma}\dar{y_F}&T_\sigma\times^{T(F)}(U(F)_\mathbb{Z}\backslash D(F))\dar{y_\sigma}\\
iU(F)\rar[hook]{j_q}&\mathbb{R}_{\geq 0}^r.
\end{tikzcd}
\end{equation}

\begin{Lem}
Every point in $(U(F)_\mathbb{Z}\backslash D)_\sigma$ has a neighborhood $\Omega$ such that
\[y_F(j_\sigma^{-1}(\Omega))=j_q^{-1}(y_\sigma(\Omega))\subseteq u+i\sigma\]
for some $u\in iC(F)$.
\end{Lem}

\begin{proof}
Notice that in every neighborhood $U$ of a point $y\in\mathbb{R}_{\geq 0}^r$ we can find a $c=(c_1,...,c_r)$ such that
\[U_c:=\big\{(x_1,...,x_r):0\leq x_j\leq c_j,j=1,...,r\big\}\subseteq\mathbb{R}_{\geq 0}^r\]
is also a neighborhood of $y$, and we have $j_q^{-1}(U_c)=j_q^{-1}(c)+i\sigma$. When $y$ lies in $y_\sigma((U(F)_\mathbb{Z}\backslash D)_\sigma)$, we can take
\[U=\textrm{the interior of closure of }j_q(y_F(U(F)_\mathbb{Z}\backslash D))\textrm{ in }\mathbb{R}_{\geq 0}^r,\]
then $j_q^{-1}(c)$ lies in $y_F(U(F)_\mathbb{Z}\backslash D)$, which equals $iC(F)$ by (\ref{DDF}).
\end{proof}

The free action of $U(F)_\mathbb{C}$ on $D(F)\cong\pi_\mathbb{C}^{-1}(D(F))/P_h(\mathbb{C})$ implies that $\pi_\mathbb{C}^{-1}(D(F))$ is a principal $U(F)_\mathbb{C}\times P_h(\mathbb{C})$ bundle over
\[D(F)'\cong U(F)_\mathbb{C}\backslash\pi_\mathbb{C}^{-1}(D(F))/P_h(\mathbb{C}),\]
and
\[\Psi_F:U(F)_\mathbb{Z}\backslash\pi_\mathbb{C}^{-1}(D(F))\rightarrow D(F)'\]
is a principal $T(F)\times P_h(\mathbb{C})$ bundle. Let $\Omega'\subseteq D(F)'$ be an open subset over which $\Psi_F$ is trivial and take a trivialisation
\begin{equation}\label{t1}
\Psi_F^{-1}(\Omega')\xrightarrow[\sim]{(\Psi_F,\widetilde{t}_T,t_P)}\Omega'\times T(F)\times P_h(\mathbb{C}),
\end{equation}
then $t_P:\Psi_F^{-1}(\Omega')\rightarrow P_h(\mathbb{C})$ is a $T(F)$-invariant, $P_h(\mathbb{C})$-equivariant holomorphic map and $\widetilde{t}_T$ descends to a $T(F)$-equivariant holomorphic map
\begin{equation}\label{t2}
t_T:p_F^{-1}(\Omega')\cong\Psi_F^{-1}(\Omega')/P_h(\mathbb{C})\rightarrow T(F)
\end{equation}
which extends to a map $t_\sigma:p_\sigma^{-1}(\Omega')\rightarrow T_\sigma$.

\begin{Def}\label{sigma}
(i) An open subset $\Omega\subseteq(U(F)_\mathbb{Z}\backslash D)_\sigma$ is a \it $\sigma$-neighborhood \rm of a point $x\in(\Gamma\backslash D)_\Sigma$ if
\begin{itemize}
\item it is relatively compact in $(U(F)_\mathbb{Z}\backslash D)_\sigma$,
\item it maps homeomorphically to a neighborhood of $x$ in $(\Gamma\backslash D)_\Sigma$ under $\phi_\sigma$,
\item $y_F(j_\sigma^{-1}(\Omega))\subseteq u+i\sigma$ for some $u\in iC(F)$ and
\item the bundle $\Psi_F$ is trivial over a neighborhood $\Omega'$ of $\overline{p_\sigma(\Omega)}$.
\end{itemize}
(ii) Coordinates $\underline{z}=(z_1,...,z_n)$ on $\Omega$ are \it $\sigma$-coordinates \rm if
\begin{itemize}
\item $z_i=c\cdot q_i\circ t_\sigma,1\leq i\leq r$ for some $t_\sigma$ as above and $c\in\mathbb{R}^\times$,
\item $(z_{r+1},...,z_n)$ are pulled back from (holomorphic) coordinates on $\Omega'$ by $p_\sigma$.
\end{itemize}
For technical convenience we also assume that
\begin{itemize}
\item $z_1,...,z_n$ take values of magnitude $<\frac{1}{3}$.
\end{itemize}
\end{Def}

We see that a sufficiently small open neighborhood of a point in $\phi_\sigma^{-1}(x)$ is a $\sigma$-neighborhood of $x$ and admits $\sigma$-coordinates. For every open subset $U\subseteq\Omega$ it holds that $j^{-1}(\phi_\sigma(U))=\phi_F(j_\sigma^{-1}(U))$, therefore for every sheaf $\mathcal{F}$ on $\Gamma\backslash D$ there is an natural isomorphism
\begin{equation}\label{sheafisom1}
(\phi_\sigma^*j_*\mathcal{F})|_\Omega\cong(j_{\sigma*}\phi_F^*\mathcal{F})|_\Omega.
\end{equation}
Similarly, if we view $\underline{z}$ as an embedding $\Omega\hookrightarrow\Delta^n$, then
\[\forall U\subseteq\Omega,\underline{z}(j_\sigma^{-1}(U))=j_{n,r}^{-1}(\underline{z}(U)),\]
where $j_{n,r}$ is the inclusion of $\Delta^{n,r}$ into $\Delta^n$. For clarity we denote by $\underline{z}_{n,r}$ and $j_\Omega$ the restrictions of $\underline{z}$ and $j_\sigma$ to $\underline{z}^{-1}(\Delta^{n,r})=j_\sigma^{-1}(\Omega)$, then we have
\begin{equation}\label{sheafisom2}
j_{\Omega*}\underline{z}_{n,r}^*\mathcal{F}\cong\underline{z}^*j_{n,r*}\mathcal{F}
\end{equation}
for every sheaf $\mathcal{F}$ on $\Delta^{n,r}$.

\subsubsection{Toroidal compactifications of $\mathrm{Sh}_\mathbb{K}$}

In \cite{Harris89} a family of compactifications of $\mathrm{Sh}_\mathbb{K}$ with nice functorial properties were introduced. They are parameterized by simplicial data $\Sigma$ as in \cite[2.5]{Harris89}, where to each $\Sigma$ and each $\gamma$ as in (\ref{gamma}) a $\Gamma(\gamma)$-admissible collection of polyhedra $\Sigma(\gamma)$ is associated, and the compactification will be $\mathrm{Sh}_{\mathbb{K},\Sigma}:=\coprod_{\{\gamma\}}(\Gamma(\gamma)\backslash D)_{\Sigma(\gamma)}$. Such a $\Sigma$ is called a \it $\mathbb{K}$-admissible collection of polyhedra \rm and $\mathrm{Sh}_{\mathbb{K},\Sigma}$ is called an \it admissible toroidal compactifications \rm of $\mathrm{Sh}_\mathbb{K}$ which is \it SNC \rm if each $\Sigma(\gamma)$ is SNC.

A $\mathbb{K}$-admissible collection of polyhedra $\Sigma$ is automatically admissible for every open subgroup $\mathbb{K}'\subseteq\mathbb{K}$, and the covering map $\mathrm{Sh}_{\mathbb{K}'}\rightarrow\mathrm{Sh}_\mathbb{K}$ extends to a surjection $\mathrm{Sh}_{\mathbb{K}',\Sigma}\rightarrow\mathrm{Sh}_{\mathbb{K},\Sigma}$ with finite fibers \cite[2.5.7 (c)]{Harris89}. Meanwhile, for every $g\in G(\mathbb{A}_f)$ there is a $g\mathbb{K}g^{-1}$-admissible collection of polyhedra $\Sigma^g$ such that $t_g:\mathrm{Sh}_{g\mathbb{K}g^{-1}}\rightarrow\mathrm{Sh}_\mathbb{K}$ extends to an isomorphism $\mathrm{Sh}_{g\mathbb{K}g^{-1},\Sigma^g}\rightarrow\mathrm{Sh}_{\mathbb{K},\Sigma}$ \cite[4.3.3]{Harris89}. Therefore for every open subgroup $\mathbb{K}'\subseteq g\mathbb{K}g^{-1}$ there is a map $t_{g,\Sigma}:\mathrm{Sh}_{\mathbb{K}',\Sigma^g}\rightarrow\mathrm{Sh}_{\mathbb{K},\Sigma}$.

\cite[2.8]{Harris89} shows that when $\Sigma$ is \it projective \rm and \it equivariant \rm as in \cite[2.7]{Harris89}, there is a projective variety $Sh(G,X)_{\mathbb{K},\Sigma}$ over $E(G,X)$ containing $Sh(G,X)_\mathbb{K}$ whose complex analytification gives $\mathrm{Sh}_{\mathbb{K},\Sigma}$. In this case $\Sigma^g$ is a projective equivariant $\mathbb{K'}$-admissible collection of polyhedra and $t_{g,\Sigma}$ is defined over $E(G,X)$ \cite[1.3.4]{Harris90}. It is also indicated in \cite[2.8]{Harris89} that there exists a projective equivariant $\Sigma$ such that $\mathrm{Sh}_{\mathbb{K},\Sigma}$ is SNC.

For $\mathbb{K}$-admissible collections of polyhedra there is an analogous notion of \it refinements\rm, and if $\Sigma'$ is a refinement of $\Sigma$ then the identity on $\mathrm{Sh}_\mathbb{K}$ extends to a proper surjection $r_{\Sigma',\Sigma}:\mathrm{Sh}_{\mathbb{K},\Sigma'}\rightarrow\mathrm{Sh}_{\mathbb{K},\Sigma}$ \cite[2.5.7 (b)]{Harris89}. We see that when $\Sigma$ and $\Sigma'$ are both projective and equivariant, $r_{\Sigma',\Sigma}$ is defined over $E(G,X)$. Finally, any two projective equivariant $\mathbb{K}$-admissible collections of polyhedra admits a common refinement that is also projective and equivariant \cite[2.7]{Harris90}.

\subsection{Canonical extensions of automorphic vector bundles}

Let $\widetilde{V}$ be an automorphic vector bundle over $\mathrm{Sh}_\mathbb{K}$ arising from a finite dimensional holomorphic representation $V$ of $P_h(\mathbb{C})$, $F$ be a rational boundary component of $D$ and $j_\sigma,p_F,p_\sigma,\phi_F,\phi_\sigma$ be the maps in (\ref{p}) and (\ref{phi}). Recall that $V$ defines a $G(\mathbb{C})$-homogeneous vector bundle $\mathcal{E}_V=G(\mathbb{C})\times^{P_h(\mathbb{C})}V$ on $\widecheck{D}$. As $U(F)_\mathbb{C}$ acts freely on $D(F)$, $\mathcal{E}_V|_{D(F)}$ descends to a holomorphic vector bundle
\begin{equation}\label{EVF}
\mathcal{E}_{V,F}:=U(F)_\mathbb{C}\backslash(\mathcal{E}_V|_{D(F)})
\end{equation}
on $D(F)'$. Then $p_F^*\mathcal{E}_{V,F}\cong\phi_F^*\widetilde{V}$ and over $(U(F)_\mathbb{Z}\backslash D)_\sigma$ the natural morphism
\[p_\sigma^*\mathcal{E}_{V,F}\rightarrow j_{\sigma*}j_\sigma^*p_\sigma^*\mathcal{E}_{V,F}=j_{\sigma*}p_F^*\mathcal{E}_{V,F}\]
is an embedding of sheaves.
\begin{Def}\label{Vcan}
Let $j:\mathrm{Sh}_\mathbb{K}\hookrightarrow\mathrm{Sh}_{\mathbb{K},\Sigma}$ be an admissible toroidal compactification. The \it canonical extension \rm of $\widetilde{V}$ over $\mathrm{Sh}_{\mathbb{K},\Sigma}$, if it existed, is the subsheaf
\[\widetilde{V}^{\mathrm{can}}=\widetilde{V}^{\mathrm{can}}_{\mathbb{K},\Sigma}\subseteq j_*\widetilde{V}\]
such that for every $\sigma\in\Sigma(\gamma)_F$ of full dimensions and every $\sigma$-neighborhood $\Omega$ of a point in $(\Gamma(\gamma)\backslash D)_{\Sigma(\gamma)}\subseteq\mathrm{Sh}_{\mathbb{K},\Sigma}$, inside
\[(\phi_\sigma^*j_*\widetilde{V})|_\Omega\cong(j_{\sigma*}\phi_F^*\widetilde{V})|_\Omega\cong(j_{\sigma*}p_F^*\mathcal{E}_{V,F})|_\Omega\]
the subsheaf $(\phi_\sigma^*\widetilde{V}^{\mathrm{can}})|_\Omega$ coincides with $(p_\sigma^*\mathcal{E}_{V,F})|_\Omega$, where the first isomorphism above is (\ref{sheafisom1}).
\end{Def}

The definition implies that $\widetilde{V}^{\mathrm{can}}$, if it existed, is a vector bundle over $\mathrm{Sh}_{\mathbb{K},\Sigma}$. The vector bundle $\widetilde{V}^{\mathrm{can}}$ was first constructed in \cite{Mumford77} when $\mathrm{Sh}_{\mathbb{K},\Sigma}$ is SNC and $\mathfrak{p}_-$ acts trivially on $V$, while the general definition as above was formulated in \cite[4.1]{Harris89}. \cite[Theorem 3.1]{Mumford77} actually implies the existence of $\widetilde{V}^{\mathrm{can}}$ in all cases, while \cite[4.2]{Harris89} proves that in a different way by reducing to the case where $\mathrm{Sh}_{\mathbb{K},\Sigma}$ is SNC and $V$ comes from a representation of $G(\mathbb{R})$ and applying Deligne's existence theorem \cite{Deligne70} (which is the origin of the name ``canonical extensions''; \cite{FC90} gave a similar proof in the case of Siegel modular varieties). \cite[4.2]{Harris89} also shows that when $\Sigma$ is projective and equivariant $\widetilde{V}^{\mathrm{can}}$ is the analytification of an algebraic vector bundle defined over $k_V$, in which case $H^i(\mathrm{Sh}_{\mathbb{K},\Sigma},\widetilde{V}^{\mathrm{can}})$ is equipped with a $k_V$-rational structure via GAGA.

\begin{Ex}\cite[Proposition 3.4 a)]{Mumford77}
When $\mathrm{Sh}_{\mathbb{K},\Sigma}$ is SNC, denote $Z:=\mathrm{Sh}_{\mathbb{K},\Sigma}\smallsetminus\mathrm{Sh}_\mathbb{K}$, then $(\Omega^p_{\mathrm{Sh}_\mathbb{K}})^{\mathrm{can}}=\Omega^p_{\mathrm{Sh}_{\mathbb{K},\Sigma}}(\log{Z})\subseteq j_*\Omega^p_{\mathrm{Sh}_\mathbb{K}}$.
\end{Ex}

It follows from the definition that if $\Sigma'$ is a refinement of $\Sigma$ then the identity map on $\widetilde{V}$ extends to an isomorphism $i_{\Sigma',\Sigma}:r_{\Sigma',\Sigma}^*\widetilde{V}^{\mathrm{can}}_\Sigma\cong\widetilde{V}^{\mathrm{can}}_{\Sigma'}$. Moreover, we have:

\begin{Prop}\label{cohisom}\cite[2.4]{Harris90}
The morphism $\widetilde{V}^{\mathrm{can}}_\Sigma\rightarrow r_{\Sigma',\Sigma*}\widetilde{V}^{\mathrm{can}}_{\Sigma'}$ adjoint to $i_{\Sigma',\Sigma}$ is also an isomorphism, and $R^i r_{\Sigma',\Sigma*}\widetilde{V}^{\mathrm{can}}_{\Sigma'}$ vanishes for $i>0$. Hence there are natural isomorphisms
\[H^i(\mathrm{Sh}_{\mathbb{K},\Sigma},\widetilde{V}^{\mathrm{can}}_\Sigma)\cong H^i(\mathrm{Sh}_{\mathbb{K},\Sigma'},\widetilde{V}^{\mathrm{can}}_{\Sigma'}).\]
\end{Prop}

When $\Sigma$ and $\Sigma'$ are projective and equivariant, $i_{\Sigma',\Sigma}$ is defined over $k_V$ and hence so is the isomorphism above. Then thanks to the existence of common refinements, $H^i(\mathrm{Sh}_{\mathbb{K},\Sigma},\widetilde{V}^{\mathrm{can}}_\Sigma)$ for all projective equivariant $\Sigma$ are naturally isomorphic over $k_V$ and will be denoted as $H^i(\widetilde{V}^{\mathrm{can}}_\mathbb{K})$. We will refer to them as \it coherent cohomology of Shimura varieties\rm.

For every $t_{g,\Sigma}:\mathrm{Sh}_{\mathbb{K}',\Sigma^g}\rightarrow\mathrm{Sh}_{\mathbb{K},\Sigma}$, by \cite[4.3]{Harris89} there is a natural isomorphism $i_g:t_{g,\Sigma}^*\widetilde{V}^{\mathrm{can}}_{\mathbb{K},\Sigma}\cong\widetilde{V}^{\mathrm{can}}_{\mathbb{K}',\Sigma^g}$, which induces maps
\[t_{g,\Sigma}^*:H^i(\mathrm{Sh}_{\mathbb{K},\Sigma},\widetilde{V}^{\mathrm{can}}_{\mathbb{K},\Sigma})\rightarrow H^i(\mathrm{Sh}_{\mathbb{K}',\Sigma^g},\widetilde{V}^{\mathrm{can}}_{\mathbb{K}',\Sigma^g}).\]
When $\Sigma$ is projective and equivariant, $t_{g,\Sigma}^*$ is defined over $k_V$, and
since $i_{\Sigma',\Sigma}$ and $i_g$ always commute, all $t_{g,\Sigma}^*$ give rise to the same map
\begin{equation}\label{tg*}
t_g^*:H^i(\widetilde{V}^{\mathrm{can}}_\mathbb{K})\rightarrow H^i(\widetilde{V}^{\mathrm{can}}_{\mathbb{K}'}).
\end{equation}
These maps make $\varinjlim_{t_e^*}H^i(\widetilde{V}^{\mathrm{can}}_\mathbb{K})$ an admissible $G(\mathbb{A}_f)$-module defined over $k_V$.

\section{The sheaf \texorpdfstring{$\mathcal{C}^\infty_{\mathrm{dmg}}$}{} on SNC toroidal compactifications}\label{section2}

Let $j:\mathrm{Sh}_\mathbb{K}\hookrightarrow\mathrm{Sh}_{\mathbb{K},\Sigma}$ be an SNC admissible toroidal compactification, $V$ be a finite-dimensional holomorphic representation of $P_h(\mathbb{C})$. In this and the next section we prove that
\begin{equation}\label{dmg1}
H^i(\mathrm{Sh}_{\mathbb{K},\Sigma},\widetilde{V}^{\mathrm{can}})\cong H^i_{(\mathfrak{p}_o,K_o)}(C^\infty_{\mathrm{dmg}}(G)^\mathbb{K}\otimes V),
\end{equation}
where the $(\mathfrak{p}_o,K_o)$-module $C^\infty_{\mathrm{dmg}}(G)$ will be defined below. The method is to find a fine resolution of $\widetilde{V}^{\mathrm{can}}$ whose global sections form the relative Chevalley-Eilenberg complex computing the right hand side above. In this section we construct a subcomplex of sheaves
\[(\mathcal{I}_V^*,j_*\overline\partial)\subseteq(j_*(\widetilde{V}\otimes_{\mathcal{O}_{\mathrm{Sh}_\mathbb{K}}}\mathcal{A}^{0,*}_{\mathrm{Sh}_\mathbb{K}}),j_*\overline\partial)\]
with desired global sections and show that the subsheaf $\mathcal{C}^\infty_{\mathrm{dmg}}:=\mathcal{I}_\mathbb{C}^0\subseteq j_*\mathcal{C}^\infty_{\mathrm{Sh}_\mathbb{K}}$ can be characterized as follows:

\begin{Def}\label{lgsheaves}
Recall that $j_{n,r}$ denotes the inclusion $\Delta^{n,r}\subseteq\Delta^n$ and on $\Delta^n$, $j_{n,r*}\mathcal{C}^\infty_{\Delta^{n,r}}$ is the sheaf $U\mapsto C^\infty(j_{n,r}^{-1}(U))$. Let $\mathcal{C}^\infty_{\mathrm{lg},r}$ be the sheafification of the subpresheaf of $j_{n,r*}\mathcal{C}^\infty_{\Delta^{n,r}}$ consisting of functions $f\in C^\infty(j_{n,r}^{-1}(\cdot))$ that are bounded by a polynomial in $\log{\frac{1}{|z_1|}},...,\log{\frac{1}{|z_r|}}$, and $\mathcal{C}^\infty_{\mathrm{dlg},r}$ be the subsheaf
\begin{equation}\label{dlg}
U\mapsto\big\{f\in\mathcal{C}^\infty_{\mathrm{lg},r}(U):\forall\,\textrm{superposition }D\textrm{ of elements of }\underline{D}_z,Df\in\mathcal{C}^\infty_{\mathrm{lg},r}(U)\big\},
\end{equation}
where $\underline{D}_z=(z_i\frac{\partial}{\partial z_i},\overline{z}_i\frac{\partial}{\partial\overline{z}_i},\frac{\partial}{\partial z_k},\frac{\partial}{\partial\overline{z}_k})_{1\leq i\leq r<k\leq n}$.
\end{Def}

\begin{Prop}\label{dmglocal}
Let $\Omega\subseteq(U(F)_\mathbb{Z}\backslash D)_\sigma$ be a $\sigma$-neighborhood of a point in $\mathrm{Sh}_{\mathbb{K},\Sigma}$ with $\sigma$-coordinates $\underline{z}=(z_1,...,z_n)$ as in Definition \ref{sigma}. Then inside
\begin{equation}\label{sheafisom4}
(\phi_\sigma^*j_*\mathcal{C}^\infty_{\mathrm{Sh}_\mathbb{K}})|_\Omega\cong(j_{\sigma*}\phi_F^*\mathcal{C}^\infty_{\mathrm{Sh}_\mathbb{K}})|_\Omega\cong j_{\Omega*}\underline{z}_{n,r}^*\mathcal{C}^\infty_{\Delta^{n,r}}\cong\underline{z}^*j_{n,r*}\mathcal{C}^\infty_{\Delta^{n,r}}
\end{equation}
the subsheaves $(\phi_\sigma^*\mathcal{C}^\infty_{\mathrm{dmg}})|_\Omega$ and $(\underline{z}^*\mathcal{C}^\infty_{\mathrm{dlg},r})|_\Omega$ coincide, where the first and third isomorphisms above are (\ref{sheafisom1}) and (\ref{sheafisom2}) respectively.
\end{Prop}

The proposition together with the fact that the sheaves $\mathcal{I}_V^i$ are $\mathcal{C}^\infty_{\mathrm{dmg}}$-modules will imply the fineness of these sheaves, and we will show that $(\mathcal{I}_V^*,j_*\overline\partial)$ forms a resolution of $\widetilde{V}^{\mathrm{can}}$ in the next section.

\begin{Rmk}
A key ingredient in the proof of Proposition \ref{dmglocal} is the estimate (\ref{GLdlg}) of the coefficients of left-invariant differential operators as well as their derivatives in $\sigma$-coordinates, which is analogous to what Borel \cite{Borel90} did for local coordinates on the Borel-Serre compactification. The strategy to prove (\ref{dmg1}) has been indicated in \cite[2.4]{Harris88}, despite that it was inaccurately rested on \cite{Borel90}.
\end{Rmk}

\subsection{The sheaves \texorpdfstring{$\mathcal{I}_V^i$}{} and \texorpdfstring{$\mathcal{C}^\infty_{\mathrm{mg}}$}{}}\label{sheaves1}

\subsubsection{Growth conditions}

\begin{Def}\label{mg}
(i) Let $G$ be a linear algebraic group over $\mathbb{R}$, $U\subseteq G(\mathbb{R})$ be an open subset, then $\mathfrak{g}_\mathbb{R}$ and hence $\mathfrak{U}(\mathfrak{g})$ act on $C^\infty(U)$ by right differentiation. Define
\[C^\infty_{\mathrm{mg}}(U):=\big\{f\in C^\infty(U):\exists F\in\mathcal{O}_G\textrm{ s.t. }|f|\leq F\big\},\]
\[C^\infty_{\mathrm{dmg}}(U):=\big\{f\in C^\infty(U):\forall D\in\mathfrak{U}(\mathfrak{g}),Df\in C^\infty_{\mathrm{mg}}(U)\big\},\]
\[C^\infty_{\mathrm{umg}}(U):=\big\{f\in C^\infty(U):\exists F\in\mathcal{O}_G\textrm{ s.t. }\forall D\in\mathfrak{U}(\mathfrak{g}),\exists c_D>0\textrm{ s.t. }|Df|\leq c_D F\big\}.\]
(ii) Let $G$ be a linear algebraic group over $\mathbb{Q}$, $U\subseteq G(\mathbb{A})$ be an open subset right-invariant by some compact open subgroup of $G(\mathbb{A}_f)$. For $\star=mg,dmg$ and $umg$, define
\[C^\infty_\star(U):=\big\{f\in C^\infty(U):\forall y\in G(\mathbb{A}_f),f(\cdot\,y)\in C^\infty_\star(Uy^{-1}\cap G(\mathbb{R}))\big\}.\]
When $G(\mathbb{R})$ has a prescribed maximal compact subgroup $K$, denote
\[C^\infty_\star(G):=C^\infty_\star([G])^{K\mbox{-}\mathrm{fin}}\]
for $\star=dmg$ and $umg$.
\end{Def}

\begin{Rmk}\label{rho}
(i) It follows from the definition that $C^\infty_{\mathrm{umg}}(U)\subseteq C^\infty_{\mathrm{dmg}}(U)\subseteq C^\infty_{\mathrm{mg}}(U)$ and $C^\infty_{\mathrm{umg}}(U),C^\infty_{\mathrm{dmg}}(U)$ are $\mathfrak{g}$-submodules of $C^\infty(U)$.\\
(ii) The growth condition for $C^\infty_{\mathrm{mg}}(G(\mathbb{R}))$ is equivalent to the one in \cite[7]{Borel66}, since the norm $\|\cdot\|$ there belongs to $\mathcal{O}_G$ while every element of $\mathcal{O}_G$ is bounded by a multiple of some $\|\cdot\|^N$.\\
(iii) On arithmetic quotients there are usually few functions descended from polynomials on the group. However some other nice functions play the role of $\|\cdot\|$: by \cite[p.187 Proposition]{Franke98}, for any reductive group $G$ over $\mathbb{Q}$ and maximal compact subgroup $\mathcal{K}\subseteq G(\mathbb{A})$, there exists a {\hypersetup{hidelinks}\hyperref[sm]{smooth function}} $\rho:[G]/\mathcal{K}\rightarrow\mathbb{R}_{>0}$ with positive infimum such that
\begin{equation}\label{dlog}
\forall D\in\mathfrak{U}(\mathfrak{g}),|\frac{D\rho}{\rho}|\textrm{ is bounded}
\end{equation}
and
\[C^\infty_{\mathrm{mg}}([G])=\big\{f\in C^\infty([G]):\exists c>0,N\in\mathbb{Z}_{>0}\textrm{ s.t. }|f|\leq c\rho^N\big\}.\]
Such a $\rho$ will be called a \it weight function\rm.
\end{Rmk}

\begin{Prop}\label{ring}
$C^\infty_{\mathrm{mg}}(U)$ and $C^\infty_{\mathrm{dmg}}(U)$ are subrings of $C^\infty(U)$.
\end{Prop}

\begin{proof}
In the real case $C^\infty_{\mathrm{mg}}(U)$ is a ring because $\mathcal{O}_G$ is, and then the Leibniz rule implies that $C^\infty_{\mathrm{dmg}}(U)$ is also a ring. The adelic case follows from the real case by definition.
\end{proof}

In practice we only consider these growth conditions for $K_o$-finite functions, in which case the following observation is essential:

\begin{Lem}\label{Borel}
Let $B\subseteq G$ be linear algebraic groups over $\mathbb{R}$, $K\subseteq G(\mathbb{R})$ be a compact subgroup, $U\subseteq B(\mathbb{R})$ be an open subset such that $UK$ is open in $G(\mathbb{R})$ and $U\times K\rightarrow UK,(u,k)\mapsto uk$ is bijective, then we have commutative diagrams
\[\begin{tikzcd}
C^\infty_\star(U)\otimes C^\infty(K)^{K\mbox{-}\mathrm{fin}}\rar{\sim}\dar[hook]&C^\infty_\star(UK)^{K\mbox{-}\mathrm{fin}}\dar[hook]\\
C^\infty(U)\otimes C^\infty(K)^{K\mbox{-}\mathrm{fin}}\rar{\sim}&C^\infty(UK)^{K\mbox{-}\mathrm{fin}},
\end{tikzcd}\]
where $\star=mg,dmg$ and the bottom horizontal map is
\begin{equation}\label{fugk}
f\otimes g\mapsto(uk\mapsto f(u)g(k)).
\end{equation}
\end{Lem}

\begin{proof}
We first show that (\ref{fugk}) is an isomorphism by giving the inverse. For every $h\in C^\infty(UK)^{K\mbox{-}\mathrm{fin}}$, $\mathrm{span}\{h(\cdot\,k):k\in K\}$ is a finite-dimensional representation of $K$. Let $S\subseteq\widehat{K}$ be the finite set of $K$-types it contains, then the image of
\[F_h:U\rightarrow C^\infty(K)^{K\mbox{-}\mathrm{fin}},\quad u\mapsto h(u\,\cdot)\]
is contained in $C^\infty(K)_S$, which is finite-dimensional. As evaluations at points in $K$ span the dual of $C^\infty(K)_S$, we see that $F_h$ is smooth, i.e.
\[F_h\in C^\infty(U)\otimes C^\infty(K)_S\subseteq C^\infty(U)\otimes C^\infty(K)^{K\mbox{-}\mathrm{fin}}.\]
It is clear that $h\mapsto F_h$ and (\ref{fugk}) are inverse to each other.

The above proof indicates that if $g_1,...,g_n\in C^\infty(K)^{K\mbox{-}\mathrm{fin}}$ are linearly independent and $h=f_1\otimes g_1+...+f_n\otimes g_n$ via (\ref{fugk}), then
\[\textrm{each }f_m\textrm{ is a linear combination of }u\mapsto h(uk_i)\textrm{ for finitely many }k_i\in K.\label{lincomb}\tag{$\ddagger$}\]
Therefore if $h\in C^\infty_{\mathrm{mg}}(UK)^{K\mbox{-}\mathrm{fin}}$ then every $f_m$ is bounded by some element of $\mathcal{O}_G$, which restricts to an element of $\mathcal{O}_B$, so $f_m\in C^\infty_{\mathrm{mg}}(U)$ and we see that
\[C^\infty_{\mathrm{mg}}(UK)^{K\mbox{-}\mathrm{fin}}\subseteq C^\infty_{\mathrm{mg}}(U)\otimes C^\infty(K)^{K\mbox{-}\mathrm{fin}}.\]
Conversely, if $f\in C^\infty_{\mathrm{mg}}(U)$, it is bounded by some element of $\mathcal{O}_G$ since $\mathcal{O}_G\rightarrow\mathcal{O}_B$ is surjective, and without loss of generality we can assume the element is some $c\|\cdot\|^N$. As the norm $\|\cdot\|$ is continuous and satisfies $\|xy\|\leq\|x\|\|y\|$ \cite[7]{Borel66} and $K$ is compact, there exists a $C>0$ such that
\[\|uk\|\geq C\|u\|,\forall u\in U,k\in K,\]
so $f\otimes 1$ and hence any $f\otimes g$ are bounded by another multiple of $\|\cdot\|^N$. Thus it also holds that
\[C^\infty_{\mathrm{mg}}(U)\otimes C^\infty(K)^{K\mbox{-}\mathrm{fin}}\subseteq C^\infty_{\mathrm{mg}}(UK)^{K\mbox{-}\mathrm{fin}}.\]

Next since we have
\[Y_1...Y_l(h(\cdot\,k))=(\mathrm{Ad}(k)^{-1}Y_1...\mathrm{Ad}(k)^{-1}Y_lh)(\cdot\,k),\forall k\in K,Y_1,...,Y_l\in\mathfrak{b},\]
(\ref{lincomb}) also implies that
\[C^\infty_{\mathrm{dmg}}(UK)^{K\mbox{-}\mathrm{fin}}\subseteq C^\infty_{\mathrm{dmg}}(U)\otimes C^\infty(K)^{K\mbox{-}\mathrm{fin}}.\]
On the other hand, note that $\mathfrak{g}=\mathfrak{b}\oplus\mathfrak{k}$ by assumption, let $Y_1,...,Y_p$ be a basis of $\mathfrak{b}$, $Z_1,...,Z_q$ be a basis of $\mathfrak{k}$, then for every $X\in\mathfrak{g}$, there exist $c_1,...,c_p,d_1,...,d_q$ in $C^\infty(K)^{K\mbox{-}\mathrm{fin}}$ such that
\[\mathrm{Ad}(k)X=\sum_{i=1}^p c_i(k)Y_i+\mathrm{Ad}(k)\sum_{j=1}^q d_j(k)Z_j,\forall k\in K,\]
and then for $h=f\otimes g$,
\[\begin{aligned}
Xh(uk)&=\left.\frac{\mathrm{d}}{\mathrm{d}t}\right|_{t=0}h(uke^{tX})\\
&=\sum_{i=1}^p\left.\frac{\mathrm{d}}{\mathrm{d}t}\right|_{t=0}h(ue^{tc_i(k)Y_i}k)+\sum_{j=1}^q\left.\frac{\mathrm{d}}{\mathrm{d}t}\right|_{t=0}h(uke^{td_j(k)Z_j})\\
&=\left.\frac{\mathrm{d}}{\mathrm{d}t}\right|_{t=0}(\sum_{i=1}^p f(ue^{tc_i(k)Y_i})g(k)+\sum_{j=1}^q f(u)g(ke^{td_j(k)Z_j}))\\
&=\sum_{i=1}^p Y_if(u)c_i(k)g(k)+\sum_{j=1}^q f(u)d_j(k)Z_jg(k),
\end{aligned}\]
that is,
\[Xh=\sum_{i=1}^p Y_if\otimes c_ig+\sum_{j=1}^q f\otimes d_jZ_jg.\]
By iterating this formula we see for every $D\in\mathfrak{U}(\mathfrak{g})$, $Dh$ is a finite sum $\sum D'f\otimes g'$, where every $D'\in\mathfrak{U}(\mathfrak{b})$ and every $g'\in C^\infty(K)^{K\mbox{-}\mathrm{fin}}$. Therefore
\[C^\infty_{\mathrm{dmg}}(U)\otimes C^\infty(K)^{K\mbox{-}\mathrm{fin}}\subseteq C^\infty_{\mathrm{dmg}}(UK)^{K\mbox{-}\mathrm{fin}}.\qedhere\]
\end{proof}

\subsubsection{Definition of the sheaves}

\begin{Def}\label{mgsheaves}
For each finite-dimensional holomorphic representation $V$ of $P_h(\mathbb{C})$ and $i\geq 0$, according to (\ref{A0i}) $j_*(\widetilde{V}\otimes_{\mathcal{O}_{\mathrm{Sh}_\mathbb{K}}}\mathcal{A}^{0,i}_{\mathrm{Sh}_\mathbb{K}})$ is the sheaf
\[U\mapsto\mathrm{Hom}_{K_o}(\textstyle\bigwedge^i(\mathfrak{p}_o/\mathfrak{k}_o),C^\infty(\pi_o^{-1}(U))\otimes V).\]
Let $\mathcal{I}_{V}^i=\mathcal{I}_{V,\mathbb{K},\Sigma}^i$ be the sheafification of the subpresheaf
\[U\mapsto\mathrm{Hom}_{K_o}(\textstyle\bigwedge^i(\mathfrak{p}_o/\mathfrak{k}_o),C^\infty_{\mathrm{dmg}}(\pi_o^{-1}(U))\otimes V),\]
where elements of $C^\infty(\pi_o^{-1}(U))$ are viewed as functions on an open subset of $G(\mathbb{A})$. When $V=\mathbb{C}$ is the trivial representation, write
\[\mathcal{A}^{0,i}_{\mathrm{dmg}}:=\mathcal{I}_\mathbb{C}^i\textrm{ and }\mathcal{C}^\infty_{\mathrm{dmg}}:=\mathcal{I}_\mathbb{C}^0,\]
so $\mathcal{C}^\infty_{\mathrm{dmg}}$ is the sheafification of the subpresheaf $U\mapsto C^\infty_{\mathrm{dmg}}(\pi_o^{-1}(U))^{K_o}$ of
\[j_*\mathcal{C}^\infty_{\mathrm{Sh}_\mathbb{K}}:U\mapsto C^\infty(j^{-1}(U))\cong C^\infty(\pi_o^{-1}(U))^{K_o}.\]
Analogously, let $\mathcal{C}^\infty_{\mathrm{mg}}$ be the sheafification of subpresheaf $U\mapsto C_{\mathrm{mg}}^\infty(\pi_o^{-1}(U))^{K_o}$.
\end{Def}

\begin{Prop}
$\mathcal{C}^\infty_{\mathrm{mg}},\mathcal{C}^\infty_{\mathrm{dmg}}$ are sheaves of rings, and every $\mathcal{I}_V^i$ is a $\mathcal{C}^\infty_{\mathrm{dmg}}$-module.
\end{Prop}

\begin{proof}
Proposition \ref{ring} implies that the presheaves defining these sheaves satisfy the claims, and sheafification preserves these properties.
\end{proof}

\begin{Prop}\label{mgsections}
For every open subset $U\subseteq\mathrm{Sh}_{\mathbb{K},\Sigma}$, $\mathcal{C}^\infty_{\mathrm{mg}}(U)$ consists of functions $f\in C^\infty(\pi_o^{-1}(U))^{K_o}$ such that for every relatively compact open subset $W\subseteq U$, $f|_{\pi_o^{-1}(W)}\in C^\infty_{\mathrm{mg}}(\pi_o^{-1}(W))^{K_o}$. Similarly, $\mathcal{I}_V^i(U)$ consists of homomorphisms
\[\varphi\in\mathrm{Hom}_{K_o}(\textstyle\bigwedge^i(\mathfrak{p}_o/\mathfrak{k}_o),C^\infty(\pi_o^{-1}(U))\otimes V)\]
such that for any $Y_1,...,Y_i\in\mathfrak{p}_o/\mathfrak{k}_o$ and $W$ as above,
\[\varphi(Y_1\!\wedge\!...\!\wedge\!Y_i)|_{\pi_o^{-1}(W)}\in C^\infty_{\mathrm{dmg}}(\pi_o^{-1}(W))\otimes V.\]
In particular, since $\mathrm{Sh}_{\mathbb{K},\Sigma}$ is compact,
\[\begin{aligned}
\mathcal{I}_V^i(\mathrm{Sh}_{\mathbb{K},\Sigma})&=\mathrm{Hom}_{K_o}(\textstyle\bigwedge^i(\mathfrak{p}_o/\mathfrak{k}_o),C^\infty_{\mathrm{dmg}}(\pi_o^{-1}(\mathrm{Sh}_{\mathbb{K},\Sigma}))\otimes V)\\
&=\mathrm{Hom}_{K_o}(\textstyle\bigwedge^i(\mathfrak{p}_o/\mathfrak{k}_o),C^\infty_{\mathrm{dmg}}(G)^\mathbb{K}\otimes V).
\end{aligned}\]
\end{Prop}

\begin{proof}
By definition for every $f\in\mathcal{C}^\infty_{\mathrm{mg}}(U)$ there is an open cover $\{U_\alpha\}$ of $U$ such that
\[f|_{\pi_o^{-1}(U_\alpha)}\in C^\infty_{\mathrm{mg}}(\pi_o^{-1}(U_\alpha))^{K_o}\]
for each $\alpha$. As $W$ is relatively compact in $U$, it is covered by finitely many $U_\alpha$, so $f|_{\pi_o^{-1}(W)}\in C^\infty_{\mathrm{mg}}(\pi_o^{-1}(W))^{K_o}$. Conversely as manifolds are locally compact, every $f$ as described in the proposition belongs to $\mathcal{C}^\infty_{\mathrm{mg}}(U)$. The proof for $\mathcal{I}_V^i$ is similar.
\end{proof}

By the same arguments we have:

\begin{Prop}\label{lgsections}
For every open subset $U\subseteq\Delta^n$, $\mathcal{C}^\infty_{\mathrm{lg},r}(U)$ consists of functions $f\in C^\infty(j_{n,r}^{-1}(U))$ such that for every relatively compact open subset $W\subseteq U$, $f|_{j_{n,r}^{-1}(W)}$ is bounded by a polynomial in $\log{\frac{1}{|z_1|}},...,\log{\frac{1}{|z_r|}}$.
\end{Prop}

Consider the differentials
\[j_*\overline\partial:j_*(\widetilde{V}\otimes_{\mathcal{O}_{\mathrm{Sh}_\mathbb{K}}}\mathcal{A}^{0,i}_{\mathrm{Sh}_\mathbb{K}})\rightarrow j_*(\widetilde{V}\otimes_{\mathcal{O}_{\mathrm{Sh}_\mathbb{K}}}\mathcal{A}^{0,i+1}_{\mathrm{Sh}_\mathbb{K}}).\]
By (\ref{partial}), the formula \cite[2.127b]{KV95} and Definition \ref{mg}, $\mathcal{I}_{V}^i$ is mapped into $\mathcal{I}_{V}^{i+1}$, so we get a complex of sheaves
\[(\mathcal{I}_{V}^*,j_*\overline\partial).\]
Then Proposition \ref{mgsections} implies that the global sections of this complex can be identified with the relative Chevalley-Eilenberg complex for
\[H^*_{(\mathfrak{p}_o,K_o)}(C^\infty_{\mathrm{dmg}}(G)^\mathbb{K}\otimes V).\]

\subsubsection{$(\phi_\sigma^*\mathcal{C}^\infty_{\mathrm{mg}})|_\Omega$ and $(\phi_\sigma^*\mathcal{A}^{0,i}_{\mathrm{dmg}})|_\Omega$}

To establish Proposition \ref{dmglocal} we want to treat $\mathcal{C}^\infty_{\mathrm{dmg}}$ and $\mathcal{C}^\infty_{\mathrm{dlg},r}$ in parallel. We've already seen the parallelism between Propositions \ref{mgsections} and \ref{lgsections}, while a relation between $\mathcal{C}^\infty_{\mathrm{mg}}$ and $\mathcal{C}^\infty_{\mathrm{dmg}}$ similar to (\ref{dlg}) would also be desirable. Here we formulate such a relation locally.

Let $\Omega\subseteq(U(F)_\mathbb{Z}\backslash D)_\sigma$ be a $\sigma$-neighborhood of a point in $\mathrm{Sh}_{\mathbb{K},\Sigma}$ for a $\sigma\in\Sigma(\gamma)_F$. When $\gamma$ is understood, denote by
\[\pi_F:U(F)_\mathbb{Z}\backslash G(\mathbb{R})^+/A_G(\mathbb{R})^\circ\rightarrow U(F)_\mathbb{Z}\backslash D\]
the quotient map, then we have a commutative diagram
\[\begin{tikzcd}
U(F)_\mathbb{Z}\backslash G(\mathbb{R})^+/A_G(\mathbb{R})^\circ\arrow[rr]\dar{\pi_F}&&{[G]/\mathbb{K}}\dar{\pi_o}\\
U(F)_\mathbb{Z}\backslash D\rar{\phi_F}&\Gamma(\gamma)\backslash D\rar[hook]&\mathrm{Sh}_\mathbb{K},
\end{tikzcd}\]
where the top horizontal map is induced from $G(\mathbb{R})^+\cong G(\mathbb{R})^+\times\{\gamma\}\subseteq G(\mathbb{A})$ and both vertical maps are principal $K_o$-bundles. As $\phi_\sigma$ maps $\Omega$ homeomorphically to $\phi_\sigma(\Omega)$, the map from $\pi_F^{-1}(\Omega)$ to $\pi_o^{-1}(\phi_\sigma(\Omega))$ is a bijection and induces $(\mathfrak{g},K_o)$-module isomorphisms
\begin{equation}\label{ringisom1}
\begin{tikzcd}
C^\infty_\star(\pi_o^{-1}(\phi_\sigma(U)))^{K_o\mbox{-}\mathrm{fin}}\rar{\sim}\dar[hook]&C^\infty_\star(\pi_F^{-1}(U))^{K_o\mbox{-}\mathrm{fin}}\dar[hook]\\
C^\infty(\pi_o^{-1}(\phi_\sigma(U)))^{K_o\mbox{-}\mathrm{fin}}\rar{\sim}&C^\infty(\pi_F^{-1}(U))^{K_o\mbox{-}\mathrm{fin}}
\end{tikzcd}
\end{equation}
for every open subset $U\subseteq\Omega$, where $\star=mg$ or $dmg$, elements of $C^\infty(\pi_F^{-1}(U))$ are viewed as functions on an open subset of $G(\mathbb{R})$ and the identification of growth conditions follows from definition.

Take a subgroup $\mathscr{B}$ of $G_\mathbb{R}$ as follows: in (\ref{GhGl}), take a minimal $\mathbb{R}$-parabolic $\mathscr{P}_l\subseteq\mathscr{G}_l(F)$ and a Langlands decomposition $\mathscr{P}_l=\mathscr{N}_l\mathscr{A}_l\mathscr{M}_l$ such that $\mathscr{M}_l(\mathbb{R})^\circ\subseteq G_l(F)\cap K_o$, then $\mathscr{B}_l:=\mathscr{N}_l\mathscr{A}_l$ is an $\mathbb{R}$-split solvable subgroup of $\mathscr{G}_l(F)$ satisfying that $G_l(F)=\mathscr{B}_l(\mathbb{R})^\circ(G_l(F)\cap K_o)$. Take an $\mathbb{R}$-split solvable subgroup $\mathscr{B}_h\subseteq\mathscr{G}_h(F)$ such that $G_h(F)=\mathscr{B}_h(\mathbb{R})^\circ(G_h(F)\cap K_o)$ in the same way and let
\[\mathscr{B}:=\mathscr{W}(F)\rtimes(\mathscr{B}_h\cdot\mathscr{B}_l\cdot A_G),\]
then $G(\mathbb{R})^+=\mathscr{B}(\mathbb{R})^\circ K_o$ and $\pi_F$ restricts to a homeomorphism
\[\pi_\mathscr{B}:U(F)_\mathbb{Z}\backslash\mathscr{B}(\mathbb{R})^\circ/A_G(\mathbb{R})^\circ\xrightarrow{\sim}U(F)_\mathbb{Z}\backslash D,\]
where the domain has a decomposition
\begin{equation}\label{B}
U(F)_\mathbb{Z}\backslash\mathscr{B}(\mathbb{R})^\circ/A_G(\mathbb{R})^\circ\cong U(F)_\mathbb{Z}\backslash(W(F)\rtimes(\mathscr{B}_h(\mathbb{R})^\circ\times\mathscr{B}_l(\mathbb{R})^\circ)).
\end{equation}
For each open subset $U\subseteq\Omega$, let $U_\mathscr{B}$ be the preimage of $\pi_\mathscr{B}^{-1}(U)$ in $\mathscr{B}(\mathbb{R})^\circ$, then $U_\mathscr{B}K_o$ is the preimage of $\pi_F^{-1}(U)$ in $G(\mathbb{R})^+$ which is open and $U_\mathscr{B}\subseteq\mathscr{B}(\mathbb{R})^\circ$ intersects $K_o$ trivially, thus Lemma \ref{Borel} gives
\begin{equation}\label{ringisom2}
\begin{tikzcd}
C^\infty_{\mathrm{mg}}(\pi_F^{-1}(U))^{K_o}\rar{\sim}\dar[hook]&C^\infty_{\mathrm{mg}}(\pi_\mathscr{B}^{-1}(U))\dar[hook]\\
C^\infty(\pi_F^{-1}(U))^{K_o}\rar{\sim}&C^\infty(\pi_\mathscr{B}^{-1}(U))
\end{tikzcd}
\end{equation}
and
\begin{equation}\label{modisom}
\begin{tikzcd}
\mathrm{Hom}_{K_o}(\bigwedge^i(\mathfrak{p}_o/\mathfrak{k}_o),C^\infty_{\mathrm{dmg}}(\pi_F^{-1}(U)))\rar{\sim}\dar[hook]&C^\infty_{\mathrm{dmg}}(\pi_\mathscr{B}^{-1}(U))\otimes(\bigwedge^i(\mathfrak{p}_o/\mathfrak{k}_o))^*\dar[hook]\\
\mathrm{Hom}_{K_o}(\bigwedge^i(\mathfrak{p}_o/\mathfrak{k}_o),C^\infty(\pi_F^{-1}(U)))\rar{\sim}&C^\infty(\pi_\mathscr{B}^{-1}(U))\otimes(\bigwedge^i(\mathfrak{p}_o/\mathfrak{k}_o))^*,
\end{tikzcd}
\end{equation}
where elements of $C^\infty(\pi_\mathscr{B}^{-1}(U))$ are thought of as functions on $U_\mathscr{B}\subseteq\mathscr{B}(\mathbb{R})^\circ$ and in the second diagram we've incorporated that
\[\mathrm{Hom}_{K_o}(\pi,C^\infty(K_o)^{K_o\mbox{-}\mathrm{fin}})\cong\pi^*\]
for every finite-dimensional representation $\pi$ of the compact group $K_o$.

Combining (\ref{ringisom2}), (\ref{modisom}) and (\ref{ringisom1}) we get:

\begin{Prop}\label{sheafB}
$(i)$ $(\phi_\sigma^*j_*\mathcal{C}^\infty_{\mathrm{Sh}_\mathbb{K}})|_\Omega$ is isomorphic to the sheaf
\[U\mapsto C^\infty(\pi_\mathscr{B}^{-1}(U)),\]
in which $(\phi_\sigma^*\mathcal{C}^\infty_{\mathrm{mg}})|_\Omega$ is the sheafification of the subpresheaf
\[U\mapsto C^\infty_{\mathrm{mg}}(\pi_\mathscr{B}^{-1}(U)).\]
$(ii)$ Similarly, $(\phi_\sigma^*j_*\mathcal{A}^{0,i}_{\mathrm{Sh}_\mathbb{K}})|_\Omega$ is isomorphic to the sheaf
\[U\mapsto C^\infty(\pi_\mathscr{B}^{-1}(U))\otimes(\textstyle\bigwedge^i)^*,\]
in which $(\phi_\sigma^*\mathcal{A}^{0,i}_{\mathrm{dmg}})|_\Omega$ is the sheafification of the subpresheaf
\[U\mapsto C^\infty_{\mathrm{dmg}}(\pi_\mathscr{B}^{-1}(U))\otimes(\textstyle\bigwedge^i)^*.\]
\end{Prop}

Now consider the following constructions: for a ring $A$, denote by $\mathrm{Der}(A)$ the set of $\mathbb{Z}$-derivations $D:A\rightarrow A$.

\begin{Def}
For every collection $\underline{D}$ of elements of $\mathrm{Der}(A)$, denote by $\mathfrak{U}(\underline{D})$ the subring of $\mathrm{End}_\mathbb{Z}(A)$ generated by the derivations in $\underline{D}$.
\end{Def}

\begin{Def}
For $A,\underline{D}$ as above and every subring $B\subseteq A$, define
\[B_{\underline{D}}:=\big\{b\in B:\forall D\in\mathfrak{U}(\underline{D}),Db\in B\big\}.\]
It is the largest subring of $B$ preserved by the derivations in $\underline{D}$.
\end{Def}

\begin{Ex}
In Definition \ref{mg}, let $\underline{D}$ be a basis of $\mathfrak{g}_\mathbb{R}$, then
\[C^\infty_{\mathrm{dmg}}(U)=C^\infty_{\mathrm{mg}}(U)_{\underline{D}}.\]
\end{Ex}

\begin{Ex}\label{Dz}
(\ref{dlg}) can also be written as
\[\mathcal{C}^\infty_{\mathrm{dlg},r}(U)=\mathcal{C}^\infty_{\mathrm{lg},r}(U)_{\underline{D}_z}.\]
\end{Ex}

Elements of $\mathfrak{b}_\mathbb{R}$ give rise to left-invariant vector fields on $\mathscr{B}(\mathbb{R})^\circ$ and act on the rings $C^\infty(\pi_\mathscr{B}^{-1}(U))$, where $\mathfrak{a}_{G,\mathbb{R}}$ acts trivially since functions in $C^\infty(\pi_\mathscr{B}^{-1}(U))$ are all $A_G(\mathbb{R})^\circ$-invariant. Note that $\dim{\mathfrak{b}_\mathbb{R}/\mathfrak{a}_{G,\mathbb{R}}}=\dim_{\mathbb{R}}{\mathrm{Sh}_\mathbb{K}}=2n$.

\begin{Prop}\label{Db1}
Let $\underline{D}_\mathfrak{b}$ be a basis of $\mathfrak{b}_\mathbb{R}/\mathfrak{a}_{G,\mathbb{R}}$, then for every open subset $U\subseteq\Omega$ we have
\[\phi_\sigma^*\mathcal{C}^\infty_{\mathrm{dmg}}(U)=(\phi_\sigma^*\mathcal{C}^\infty_{\mathrm{mg}}(U))_{\underline{D}_\mathfrak{b}}.\]
\end{Prop}

\begin{proof}
Combining Proposition \ref{sheafB} and the proof of Proposition \ref{mgsections}, we see that for $\star=mg,dmg$, $\phi_\sigma^*\mathcal{C}^\infty_\star(U)$ consists of $f\in C^\infty(\pi_\mathscr{B}^{-1}(U))$ such that for every relatively compact open subset $W\subseteq U$, $f|{\pi_\mathscr{B}^{-1}(W)}\in C^\infty_\star(\pi_\mathscr{B}^{-1}(W))$. Thus
\[\begin{aligned}
f\in\phi_\sigma^*\mathcal{C}^\infty_{\mathrm{dmg}}(U)&\Leftrightarrow\forall W,\forall D\in\mathfrak{U}(\underline{D}_\mathfrak{b}),D(f|_{\pi_\mathscr{B}^{-1}(W)})\in C^\infty_{\mathrm{mg}}(\pi_\mathscr{B}^{-1}(W))\\
&\Leftrightarrow\forall W,\forall D\in\mathfrak{U}(\underline{D}_\mathfrak{b}),(Df)|_{\pi_\mathscr{B}^{-1}(W)}\in C^\infty_{\mathrm{mg}}(\pi_\mathscr{B}^{-1}(W))\\\
&\Leftrightarrow\forall D\in\mathfrak{U}(\underline{D}_\mathfrak{b}),Df\in\phi_\sigma^*\mathcal{C}^\infty_{\mathrm{mg}}(U)\\
&\Leftrightarrow f\in(\phi_\sigma^*\mathcal{C}^\infty_{\mathrm{mg}}(U))_{\underline{D}_\mathfrak{b}},
\end{aligned}\]
where $\forall W$ means ``for every relatively compact open subset $W\subseteq U$''.
\end{proof}

\subsection{Proof of Proposition \ref{dmglocal}}\label{sheaves2}

Our strategy to prove Proposition \ref{dmglocal} is to first show that $(\phi_\sigma^*\mathcal{C}^\infty_{\mathrm{mg}})|_\Omega=(\underline{z}^*\mathcal{C}^\infty_{\mathrm{lg},r})|_\Omega$, and then prove that the collections $\underline{D}_\mathfrak{b}$ and $\underline{D}_z$ of derivations define the same subsheaf of rings.

\subsubsection{Identification of $(\phi_\sigma^*\mathcal{C}^\infty_{\mathrm{mg}})|_\Omega$ and $(\underline{z}^*\mathcal{C}^\infty_{\mathrm{lg},r})|_\Omega$}

\begin{Prop}\label{mglocal}
$(\phi_\sigma^*\mathcal{C}^\infty_{\mathrm{mg}})|_\Omega=(\underline{z}^*\mathcal{C}^\infty_{\mathrm{lg},r})|_\Omega$.
\end{Prop}

Recall that $\pi_\mathscr{B}:U(F)_\mathbb{Z}\backslash\mathscr{B}(\mathbb{R})^\circ/A_G(\mathbb{R})^\circ\rightarrow U(F)_\mathbb{Z}\backslash D$ is a homeomorphism. We wish to verify that on $\pi_\mathscr{B}^{-1}(\Omega)\cong j_\sigma^{-1}(\Omega)$, the bounds given by elements of $\mathcal{O}_\mathscr{B}$ and by $\log{\frac{1}{|z_1|}},...,\log{\frac{1}{|z_r|}}$ are equivalent.

For a point in $U(F)_\mathbb{Z}\backslash\mathscr{B}(\mathbb{R})^\circ/A_G(\mathbb{R})^\circ$, by its $w$, $b_h$ and $b_l$-factors we mean its image in $U(F)_\mathbb{Z}\backslash W(F)$, $\mathscr{B}_h(\mathbb{R})^\circ$ and $\mathscr{B}_l(\mathbb{R})^\circ$ under (\ref{B}) respectively, then

\begin{Prop}
The $w$ and $b_h$-factors of points in $\pi_\mathscr{B}^{-1}(\Omega)$ range in relatively compact sets.
\end{Prop}

\begin{proof}
From (\ref{DDF}) we see that the $w$ and $b_h$-factors of a point in $U(F)_\mathbb{Z}\backslash D$ is determined by its image in $(iU(F)\times U(F)_\mathbb{Z})\backslash D(F)$, which is a $(S^1)^r$-bundle over $D(F)'$. By definition $\Omega$ is relatively compact in $(U(F)_\mathbb{Z}\backslash D)_\sigma$, thus $p_\sigma(\Omega)$ is relatively compact in $D(F)'$ and so is its preimage in $(iU(F)\times U(F)_\mathbb{Z})\backslash D(F)$. So $w$ and $b_h$-factors of points in $\pi_\mathscr{B}^{-1}(\Omega)$ vary in relatively compact ranges.
\end{proof}

By the proposition, on $\pi_\mathscr{B}^{-1}(\Omega)$ the bound given by elements of $\mathcal{O}_\mathscr{B}$ is equivalent to the one given by elements of $\mathcal{O}_{\mathscr{B}_l}$.

We see from (\ref{DDF}) that $y_F$ in (\ref{y}) is the composition of taking $b_l$-factor and the map
\[\mu:\mathscr{B}_l(\mathbb{R})^\circ\rightarrow iU(F),\quad b_l\mapsto\exp(i\mathrm{Ad}(b_l)\omega_F),\]
which is injective since $\mathscr{B}_l(\mathbb{R})^\circ$ intersects $K_o$ trivially and has image $iC(F)$. Therefore elements of $\mathcal{O}_{\mathscr{B}_l}$ can be viewed as functions on $iC(F)$. On the other hand let $\Omega'$ be as in Definition \ref{sigma}, $y_t$ be the composition of $t_T$ as in (\ref{t2}) with $T(F)\rightarrow iU(F)$, then
\begin{equation}\label{logz}
\log{\frac{1}{|z_j|}}=-2\pi il_j\circ y_t-\log{c},1\leq j\leq r.
\end{equation}
By definition both $y_F$ and $y_t$ are smooth and $T(F)$-equivariant, so there is a smooth map $y':\Omega'\rightarrow iU(F)$ such that
\begin{equation}\label{y'}
y_F=y_t+y'\circ p_F.
\end{equation}
Since $p_\sigma(\Omega)$ is relatively compact in $\Omega'$, $y'$ is bounded on it, so a function on $j_\sigma^{-1}(\Omega)$ is bounded by a polynomial in $\log{\frac{1}{|z_1|}},...,\log{\frac{1}{|z_r|}}$, i.e. a polynomial on $iU(F)$ pulled back by $y_t$ if and only if it is bounded by a polynomial on $iU(F)$ pulled back by $y_F$. Therefore Proposition \ref{mglocal} will follow from:

\begin{Prop}\label{OB}
$\mathcal{O}_{\mathscr{B}_l}$ contains all the polynomial functions on $iC(F)\subseteq iU(F)$. On the other hand, on $y_F(j_\sigma^{-1}(\Omega))$ elements of $\mathcal{O}_{\mathscr{B}_l}$ as well as their partial derivatives are all bounded by polynomials.
\end{Prop}

\begin{proof}
First notice that $\mu$ is defined via an algebraic representation of $\mathscr{B}_l$, under which the orbit of $\exp(i\omega_F)$ is open since the tangent map of $\mu$ at the identity is surjective, and the stabilizer $\mathscr{F}$ of $\exp(i\omega_F)$ is finite, say of order $d$, since it intersects $\mathscr{B}_l(\mathbb{R})^\circ$ trivially. Denote $\overline{\mathscr{B}_l}:=\mathscr{B}_l/\mathscr{F}$, then it is affine and embeds into $i\mathscr{U}(F)$ as an open subscheme.

\begin{Lem}
When $A$ is a unique factorization domain, every affine open subscheme $U\subseteq\mathrm{Spec}\,A$ is of the form $D(f)\cong\mathrm{Spec}\,A_f$ for some $f\in A$.
\end{Lem}

\begin{proof}
Granted the affineness, it suffices to see that $\mathcal{O}(U)\subseteq\mathrm{Frac}(A)$ equals some $A_f$. As $U$ is quasi-compact it's of the form $D(f_1)\cup...\cup D(f_n)$, let $f:=\mathrm{gcd}(f_1,...,f_n)$, then clearly $A_f\subseteq\mathcal{O}(U)$. On the other hand take any $g/h\in \mathcal{O}(U)$ with $g$ and $h$ being coprime and let $h_0$ be any irreducible factor of $h$, then $h_0$ is a factor of every $f_i$ and hence a factor of $f$. Therefore $\mathcal{O}(U)=A_f$.
\end{proof}

By the lemma there exists a polynomial $f$ such that
\[\mathcal{O}_{\overline{\mathscr{B}_l}}=\mathcal{O}_{i\mathscr{U}(F)}[f^{-1}],\]
and we can take $f$ to be a product $f_1...f_m$ of distinct irreducible polynomials $f_1,...,f_m\in\mathcal{O}_{i\mathscr{U}(F)}$. Then
\[\mathcal{O}_{\overline{\mathscr{B}_l}}^\times=\big\{af_1^{k_1}...f_m^{k_m}:a\in\mathbb{R}^\times,k_1,...,k_m\in\mathbb{Z}\big\}.\]

Next consider the Levi subgroup $\mathscr{A}_l\subseteq\mathscr{B}_l$, which is an $\mathbb{R}$-split torus, as a quotient of $\mathscr{B}_l$. $\mathscr{F}$ maps injectively into $\mathscr{A}_l$ since unipotent groups over a field of characteristic $0$ have no finite subgroups. Denote the quotient of $\mathscr{A}_l$ as $\overline{\mathscr{A}_l}$, then $\mathcal{O}_{\mathscr{A}_l}$ is generated over $\mathcal{O}_{\overline{\mathscr{A}_l}}$ by some fractional powers of elements in $\mathcal{O}_{\overline{\mathscr{A}_l}}^\times\,$. Now we have the following commutative diagram, in which every vertical map is an \'{e}tale cover of degree $d$:
\[\begin{tikzcd}
\mathscr{B}_l\rar[hook]\dar&\overline{\mathscr{B}_l}\times_{\overline{\mathscr{A}_l}}\mathscr{A}_l\rar\dar&\mathscr{A}_l\dar\\
\overline{\mathscr{B}_l}\rar[equal]&\overline{\mathscr{B}_l}\rar&\overline{\mathscr{A}_l}.
\end{tikzcd}\]
It follows that $\mathscr{B}_l\cong\overline{\mathscr{B}_l}\times_{\overline{\mathscr{A}_l}}\mathscr{A}_l$, and thus
\[\mathcal{O}_{\mathscr{B}_l}\cong\mathcal{O}_{\overline{\mathscr{B}_l}}\otimes_{\mathcal{O}_{\overline{\mathscr{A}_l}}}\mathcal{O}_{\mathscr{A}_l}.\]
Therefore $\mathcal{O}_{\mathscr{B}_l}$ is generated over $\mathcal{O}_{\overline{\mathscr{B}_l}}=\mathcal{O}_{i\mathscr{U}(F)}[f_1^{-1},...,f_m^{-1}]$ by some products of fractional powers of $f_1,...,f_m$. For any such product $f_1^{q_1}...f_m^{q_m}$ and partial derivative $\partial_{\vec{v}}$ along a vector $\vec{v}\in iU(F)$, we have
\[\partial_{\vec{v}}f_1^{q_1}...f_m^{q_m}=f_1^{q_1}...f_m^{q_m}\sum_{i=1}^m q_i\frac{\partial_{\vec{v}}f_i}{f_i}\in f_1^{q_1}...f_m^{q_m}\mathcal{O}_{\overline{\mathscr{B}_l}},\]
therefore $\mathcal{O}_{\mathscr{B}_l}$ contains all partial derivatives of its elements.

It remains to show that on $y_F(j_\sigma^{-1}(\Omega))$, elements of $\mathcal{O}_{\mathscr{B}_l}$ are bounded by polynomials, and it suffices to prove this for $f_1^{-1},...,f_m^{-1}$. By definition $y_F(j_\sigma^{-1}(\Omega))$ can be enlarged to some $W=u+i\sigma\subseteq iC(F)$, which is a \it basic semialgebraic set\rm, i.e. a set of the form
\[W=\big\{x\in\mathbb{R}^r:g_1(x)\geq 0,...,g_s(x)\geq 0\big\}\]
for some polynomials $g_1,...,g_s$, and then we apply the following corollary of Krivine's Positivstellensatz:

\begin{Lem}
Let $W\subseteq\mathbb{R}^r$ be a basic semialgebraic set, $f$ be a polynomial taking positive values on $W$, then $f^{-1}$ is bounded by a polynomial $g$ on $W$.
\end{Lem}

\begin{proof}
In \cite[4.4.3 (ii)]{BCR98} take $V=\mathbb{R}^r$, then it asserts that there exist polynomials $g,h$ that are non-negative on $W$ and satisfy $fg=1+h$, thus on $W$,
\[fg\geq 1\Rightarrow f^{-1}\leq g.\qedhere\]
\end{proof}

Since $f_1,...,f_m$ lie in $\mathcal{O}_{\overline{\mathscr{B}_l}}^\times$, they are nonzero on $W\subseteq iC(F)$ and without loss of generality we may assume that they are positive, then the lemma concludes that $f_1^{-1},...,f_m^{-1}$ are bounded by polynomials on $W$.
\end{proof}

\subsubsection{A ring-theoretic lemma}

\begin{Lem}\label{Der}
Let $C\subseteq A$ be rings, $D_1,...,D_N\in\mathrm{Der}(A)$ be derivations preserving $C$, $D'_i=\sum_{j=1}^N c_{ij}D_j(i=1,...,N)$ be linear combinations of them with coefficients in $C$, then\\
$(i)$ every superposition of $D'_1,...,D'_N$ is a linear combination of superpositions of $D_1,...,D_N$ with coefficients in $C$;\\
$(ii)$ if $(c_{ij})_{1\leq i,j\leq N}\in\mathrm{GL}_N(C)$, then every superposition of $D_1,...,D_N$ is also a linear combination of superpositions of $D'_1,...,D'_N$ with coefficients in $C$.
\end{Lem}

\begin{proof} (i) is clear by induction; as for (ii), we notice that $D'_1,...,D'_N$ also preserve $C$, so when $D_1,...,D_N$ are linear combinations of $D'_1,...,D'_N$ with coefficients in $C$ we can apply (i) with the roles of $D_1,...,D_N$ and $D'_1,...,D'_N$ exchanged.
\end{proof}

\begin{Cor}
In (ii), for every subring $B\subseteq A$ containing $C$, we have
\[B_{\underline{D}}=B_{\underline{D}'},\]
where $\underline{D}=(D_1,...,D_N)$, $\underline{D}'=(D'_1,...,D'_N)$.
\end{Cor}

\begin{proof}
From the proof of (ii) we see that under the assumption there, the roles of $\underline{D}$ and $\underline{D}'$ are symmetric, so it suffices to show $B_{\underline{D}}\subseteq B_{\underline{D}'}$. Now by part (i), every element of $\mathfrak{U}(\underline{D}')$ is a linear combination of elements of $\mathfrak{U}(\underline{D})$ with coefficients in $C$, so $b\in B_{\underline{D}}$ implies that $b\in B_{\underline{D}'}$.
\end{proof}

Notice that $C=B_{\underline{D}}$ satisfies the condition of Lemma \ref{Der}, in this special case the corollary states that

\begin{Lem}\label{DerCor}
Let $B\subseteq A$ be rings, $\underline{D}=(D_1,...,D_N)$ be a collection of derivations of $A$. If $D'_i=\sum_{j=1}^N c_{ij}D_j(i=1,...,N)$ satisfies that $(c_{ij})_{1\leq i,j\leq N}\in\mathrm{GL}_N(B_{\underline{D}})$, then
\[B_{\underline{D}}=B_{\underline{D}'},\]
where $\underline{D'}=(D'_1,...,D'_N)$.
\end{Lem}

\begin{Cor}\label{dmglocalpf}
If we have
\begin{equation}\label{GLdlg}
\underline{D}_\mathfrak{b}\in\underline{D}_z\mathrm{GL}_{2n}(\underline{z}^*\mathcal{C}^\infty_{\mathrm{dlg},r}(\Omega)),
\end{equation}
where $\underline{D}_z$ and $\underline{D}_\mathfrak{b}$ are both viewed as tuples of vector fields on $j_\sigma^{-1}(\Omega)$, then Proposition \ref{dmglocal} holds.
\end{Cor}

\begin{proof}[Proof of Corollary \ref{dmglocalpf}]
(\ref{GLdlg}) implies that
\begin{equation}\label{GLdlgU}
\underline{D}_\mathfrak{b}\in\underline{D}_z\mathrm{GL}_{2n}(\underline{z}^*\mathcal{C}^\infty_{\mathrm{dlg},r}(U))
\end{equation}
for every open subset $U\subseteq\Omega$, and then
\[\begin{aligned}
\phi_\sigma^*\mathcal{C}^\infty_{\mathrm{dmg}}(U)&=(\phi_\sigma^*\mathcal{C}^\infty_{\mathrm{mg}}(U))_{\underline{D}_\mathfrak{b}}&\textrm{(Proposition \ref{Db1})}\\
&=(\underline{z}^*\mathcal{C}^\infty_{\mathrm{lg},r}(U))_{\underline{D}_\mathfrak{b}}&\textrm{(Proposition \ref{mglocal})}\\
&=(\underline{z}^*\mathcal{C}^\infty_{\mathrm{lg},r}(U))_{\underline{D}_z}&\textrm{((\ref{GLdlgU})+Lemma \ref{DerCor})}\\
&=\underline{z}^*\mathcal{C}^\infty_{\mathrm{dlg},r}(U)&\textrm{(Example \ref{Dz})}.
\end{aligned}\]
\end{proof}

\subsubsection{Vector fields on a $\sigma$-neighborhood}

It remains to verify the condition (\ref{GLdlg}) in Corollary \ref{dmglocalpf}. We first note that in Definition \ref{sigma} $z_1,...,z_n$ actually can be defined as functions on $p_\sigma^{-1}(\Omega')$, and then $\underline{D}_z=(z_i\frac{\partial}{\partial z_i},\overline{z}_i\frac{\partial}{\partial\overline{z}_i},\frac{\partial}{\partial z_k},\frac{\partial}{\partial\overline{z}_k})_{1\leq i\leq r<k\leq n}$ is
extended to a collection of $T(F)$-invariant vector fields on $p_\sigma^{-1}(\Omega')$, which gives a basis of the tangent space at every point of $p_F^{-1}(\Omega')$.

Meanwhile, (\ref{DF}) implies that
\begin{equation}\label{DF2}
D(F)\cong(iU(F)\times W(F))\rtimes\mathscr{B}_h(\mathbb{R})^\circ,
\end{equation}
which means each point of $D(F)$ can be written as $uwb_h\cdot\exp(-i\omega_F)h$ with $u\in iU(F)$, $w\in W(F)$ and $b_h\in\mathscr{B}_h(\mathbb{R})^\circ$ in a unique way. Let $X_1,...,X_r$ be a basis of $i\mathfrak{u}_\mathbb{R}$, $X_{r+1},...,X_{2n}$ be a basis of $\mathfrak{w}_\mathbb{R}\oplus\mathfrak{b}_{h,\mathbb{R}}$. Consider the vector fields
\begin{equation}\label{uwbh}
uwb_h\cdot\exp(-i\omega_F)h\mapsto\left.\frac{\mathrm{d}}{\mathrm{d}t}\right|_{t=0}uwb_he^{tX_j}\cdot\exp(-i\omega_F)h,1\leq j\leq 2n
\end{equation}
on $D(F)$, we see from (\ref{DF2}) that they give a basis of the tangent space at each point, and as (\ref{DF2}) is $U(F)_\mathbb{C}$-equivariant, these vector fields are $U(F)_\mathbb{C}$-invariant and hence descends to $T(F)$-invariant vector fields on $U(F)_\mathbb{Z}\backslash D(F)$. Denote them as $v(X_1),...,v(X_{2n})$ and let $\underline{D}_v:=(v(X_1),...,v(X_{2n}))$, then analogous to (\ref{y'}) there is a smooth map $g':\Omega'\rightarrow\mathrm{GL}_{2n}(\mathbb{R})$ such that
\[\underline{D}_v(x)=\underline{D}_z(x) g'(p_F(x)),\forall x\in p_F^{-1}(\Omega').\]

\begin{Prop}\label{Dv}
\begin{equation}\label{g'GLdlg}
(g'\circ p_F)|_{j_\sigma^{-1}(\Omega)}\in\mathrm{GL}_{2n}(z^*\mathcal{C}^\infty_{\mathrm{dlg},r}(\Omega)),
\end{equation}
and hence
\[\underline{D}_v\in\underline{D}_z\mathrm{GL}_{2n}(\underline{z}^*\mathcal{C}^\infty_{\mathrm{dlg},r}(\Omega)).\]
\end{Prop}

\begin{proof}
(\ref{g'GLdlg}) means that $(f\circ g'\circ p_F)|_{j_\sigma^{-1}(\Omega)}\in z^*\mathcal{C}^\infty_{\mathrm{dlg},r}(\Omega)$ for every $f\in\mathcal{O}_{\mathrm{GL}_{2n/\mathbb{R}}}$. In fact, $f\circ g'\circ p_F$ is pulled back from a smooth function on a neighborhood of $\overline{p_\sigma(\Omega)}$, so it is bounded on $j_\sigma^{-1}(\Omega)$ and in particular $(f\circ g'\circ p_F)|_{j_\sigma^{-1}(\Omega)}\in z^*\mathcal{C}^\infty_{\mathrm{lg},r}(\Omega)$. Moreover, every derivative of $f\circ g'\circ p_F$ is the same kind of function, so $(f\circ g'\circ p_F)|_{j_\sigma^{-1}(\Omega)}$ actually lies in $z^*\mathcal{C}^\infty_{\mathrm{dlg},r}(\Omega)$.
\end{proof}

Similarly, (\ref{D}) implies that
\begin{equation}\label{D2}
D\cong W(F)\rtimes(\mathscr{B}_h(\mathbb{R})^\circ\times\mathscr{B}_l(\mathbb{R})^\circ)
\end{equation}
and each point of $D$ can be uniquely written as $wb_hb_l\cdot h$ with $w\in W(F)$, $b_h\in\mathscr{B}_h(\mathbb{R})^\circ$ and $b_l\in\mathscr{B}_l(\mathbb{R})^\circ$. For every $X\in\mathfrak{w}_\mathbb{R}\oplus\mathfrak{b}_{h,\mathbb{R}}\oplus\mathfrak{b}_{l,\mathbb{R}}$, the vector field
\[wb_hb_l\cdot h\mapsto\left.\frac{\mathrm{d}}{\mathrm{d}t}\right|_{t=0}wb_hb_le^{tX}\cdot h\]
on $D$ is $U(F)_\mathbb{Z}$-invariant and thus descends to a vector field $\widetilde{v}(X)$ on $U(F)_\mathbb{Z}\backslash D$. Without loss of generality assume the basis $\underline{D}_\mathfrak{b}$ of $\mathfrak{b}_\mathbb{R}/\mathfrak{a}_{G,\mathbb{R}}\cong\mathfrak{w}_\mathbb{R}\oplus\mathfrak{b}_{h,\mathbb{R}}\oplus\mathfrak{b}_{l,\mathbb{R}}$ consists of a basis $\widetilde{X}_1,...,\widetilde{X}_r$ of $\mathfrak{b}_{l,\mathbb{R}}$ and the above basis $X_{r+1},...,X_{2n}$ of $\mathfrak{w}_\mathbb{R}\oplus\mathfrak{b}_{h,\mathbb{R}}$, then the corresponding collection of vector fields is $(\widetilde{v}(\widetilde{X}_1),...,\widetilde{v}(\widetilde{X}_r),\widetilde{v}(X_{r+1}),...,\widetilde{v}(X_{2n}))$. By (\ref{D2}) it also gives a basis of the tangent space at each point.

We want to compare $\underline{D}_\mathfrak{b}$ with $\underline{D}_v$. By (\ref{DDF}), at the point
\[wb_hb_l\cdot h=\mu(b_l)wb_h\cdot\exp(-i\omega_F)h,\]
the vector field given by $X$ takes the value
\begin{equation}\label{wbh}
\begin{aligned}
\left.\frac{\mathrm{d}}{\mathrm{d}t}\right|_{t=0}wb_hb_le^{tX}\cdot h&=\left.\frac{\mathrm{d}}{\mathrm{d}t}\right|_{t=0}wb_he^{t\mathrm{Ad}(b_l)X}b_l\cdot h\\
&=\left.\frac{\mathrm{d}}{\mathrm{d}t}\right|_{t=0}\mu(b_l)wb_he^{t\mathrm{Ad}(b_l)X}\cdot\exp(-i\omega_F)h
\end{aligned}
\end{equation}
when $X\in\mathfrak{w}_\mathbb{R}\oplus\mathfrak{b}_{h,\mathbb{R}}$, and
\begin{equation}\label{bl1}
\left.\frac{\mathrm{d}}{\mathrm{d}t}\right|_{t=0}wb_hb_le^{tX}\cdot h=\left.\frac{\mathrm{d}}{\mathrm{d}t}\right|_{t=0}\mu(b_le^{tX})wb_h\cdot\exp(-i\omega_F)h
\end{equation}
when $X\in\mathfrak{b}_{l,\mathbb{R}}$.

\begin{Prop}
$X\mapsto\mathrm{Ad}(e^{-i\omega_F})X-X$ defines a map $\vartheta:\mathfrak{b}_{l,\mathbb{R}}\rightarrow i\mathfrak{u}_\mathbb{R}$, and the right hand side of (\ref{bl1}) equals
\begin{equation}\label{bl2}
\left.\frac{\mathrm{d}}{\mathrm{d}t}\right|_{t=0}\mu(b_l)wb_he^{t\mathrm{Ad}(b_l)\vartheta(X)}\cdot\exp(-i\omega_F)h.
\end{equation}
\end{Prop}

\begin{proof}
First note that since $\mathscr{B}_l$ normalizes $\mathscr{U}$, $[\mathfrak{b}_{l,\mathbb{R}},i\mathfrak{u}_\mathbb{R}]\subseteq i\mathfrak{u}_\mathbb{R}$, so $\vartheta$ has image in $i\mathfrak{u}_\mathbb{R}$. Next, we have
\[\begin{aligned}
\mu(b_le^{tX})&=b_le^{tX}e^{i\omega_F}e^{-tX}b_l^{-1}\\
&=b_le^{i\omega_F}e^{t\mathrm{Ad}(e^{-i\omega_F})X}e^{-tX}b_l^{-1}\\
&=\mu(b_l)e^{t\mathrm{Ad}(b_le^{-i\omega_F})X}e^{-t\mathrm{Ad}(b_l)X}.
\end{aligned}\]
Hence
\[\begin{aligned}
\left.\frac{\mathrm{d}}{\mathrm{d}t}\right|_{t=0}\mu(b_le^{tX})wb_h&=\left.\frac{\mathrm{d}}{\mathrm{d}t}\right|_{t=0}\mu(b_l)e^{t\mathrm{Ad}(b_le^{-i\omega_F})X-t\mathrm{Ad}(b_l)X}wb_h\\
&=\left.\frac{\mathrm{d}}{\mathrm{d}t}\right|_{t=0}\mu(b_l)e^{t\mathrm{Ad}(b_l)\vartheta(X)}wb_h\\
&=\left.\frac{\mathrm{d}}{\mathrm{d}t}\right|_{t=0}\mu(b_l)wb_he^{t\mathrm{Ad}(b_l)\vartheta(X)},
\end{aligned}\]
where the last step is because $\mathscr{W}\rtimes\mathscr{B}_h$ centralizes $\mathscr{U}$.
\end{proof}

Comparing (\ref{wbh}), (\ref{bl1}) and (\ref{bl2}) with (\ref{uwbh}), we get
\[\underline{D}_\mathfrak{b}=\underline{D}_v\mathrm{diag}(\vartheta\circ\mathrm{Ad}(b_l)|_{i\mathfrak{u}_\mathbb{R}},\mathrm{Ad}(b_l)|_{\mathfrak{w}_\mathbb{R}\oplus\mathfrak{b}_{h,\mathbb{R}}}),\]
where $\vartheta\circ\mathrm{Ad}(b_l)|_{i\mathfrak{u}_\mathbb{R}}$ and $\mathrm{Ad}(b_l)|_{\mathfrak{w}_\mathbb{R}\oplus\mathfrak{b}_{h,\mathbb{R}}}$ stands for the matrices of the maps with respect to the bases $X_1,...,X_r$ of $i\mathfrak{u}_\mathbb{R}$, $\widetilde{X}_1,...,\widetilde{X}_r$ of $\mathfrak{b}_{l,\mathbb{R}}$ and $X_{r+1},...,X_{2n}$ of $\mathfrak{w}_\mathbb{R}\oplus\mathfrak{b}_{h,\mathbb{R}}$, and $b_l$ is the $b_l$-factor function. Particularly since $\underline{D}_v$ and $\underline{D}_\mathfrak{b}$ both give bases of the tangent space at each point, $\vartheta$ must be an isomorphism, so we can choose $\widetilde{X}_1,...,\widetilde{X}_r$ such that $\vartheta(\widetilde{X}_1)=X_1,...,\vartheta(\widetilde{X}_r)=X_r$. Then the formula above becomes
\[\underline{D}_\mathfrak{b}=\underline{D}_v\mathrm{diag}(\mathrm{Ad}(b_l)|_{i\mathfrak{u}_\mathbb{R}},\mathrm{Ad}(b_l)|_{\mathfrak{w}_\mathbb{R}\oplus\mathfrak{b}_{h,\mathbb{R}}})=\underline{D}_v\mathrm{Ad}(b_l)|_{i\mathfrak{u}_\mathbb{R}\oplus\mathfrak{w}_\mathbb{R}\oplus\mathfrak{b}_{h,\mathbb{R}}}.\]

\begin{Prop}\label{Db2}
\[(\mathrm{Ad}(b_l)|_{i\mathfrak{u}_\mathbb{R}\oplus\mathfrak{w}_\mathbb{R}\oplus\mathfrak{b}_{h,\mathbb{R}}})|_{j_\sigma^{-1}(\Omega)}\in\mathrm{GL}_{2n}(z^*\mathcal{C}^\infty_{\mathrm{dlg},r}(\Omega)),\]
and hence
\[\underline{D}_\mathfrak{b}\in\underline{D}_v\mathrm{GL}_{2n}(\underline{z}^*\mathcal{C}^\infty_{\mathrm{dlg},r}(\Omega)).\]
\end{Prop}

\begin{proof}
It suffices to prove that every $f\in\mathcal{O}_{\mathscr{B}_l}$, viewed as a function on $iC(F)$, satisfies that $(f\circ y_F)|_{j_\sigma^{-1}(\Omega)}\in z^*\mathcal{C}^\infty_{\mathrm{dlg},r}(\Omega)$. We've proved in Proposition \ref{OB} that the composition of $y_F$ with any partial derivative of $f$ lies in $z^*\mathcal{C}^\infty_{\mathrm{lg},r}(\Omega)$. Meanwhile, on $iU(F)$ denote $y_j:=-2\pi il_j(1\leq j\leq r)$, then for every $D\in\underline{D}_z$,
\[D(f\circ y_F)=\sum_{j=1}^r(\frac{\partial f}{\partial y_j}\circ y_F)D(y_j\circ y_F).\]
By induction for every $D\in\mathfrak{U}(\underline{D}_z)$, $D(f\circ y_F)$ is a finite sum $\sum(g\circ y_F)h$, where $g$ are partial derivatives of $f$ while $h$ are finite products $\prod D'(y_j\circ y_F)$ for several $D'\in\mathfrak{U}(\underline{D}_z)$ and $1\leq j\leq r$. So it suffices to prove that
\[(y_j\circ y_F)|_{j_\sigma^{-1}(\Omega)}\in z^*\mathcal{C}^\infty_{\mathrm{dlg},r}(\Omega)\]
for $j=1,...,r$. Now by (\ref{logz}) and (\ref{y'}),
\[\begin{aligned}
y_j\circ y_F&=y_j\circ y_t+y_j\circ y'\circ p_F\\
&=\log{\frac{c}{|z_j|}}+y_j\circ y'\circ p_F.
\end{aligned}\]
$D\log{\frac{c}{|z_j|}}$ is a constant for every $D\in\underline{D}_z$, so $\log{\frac{c}{|z_j|}}$ is a section of $\mathcal{C}^\infty_{\mathrm{dlg},r}$. As for the second term, $y_j\circ y'\circ p_F$ is pulled back from a smooth function on a neighborhood of $\overline{p_\sigma(\Omega)}$, so the proof of Proposition \ref{Dv} also implies that
\[(y_j\circ y'\circ p_F)|_{j_\sigma^{-1}(\Omega)}\in z^*\mathcal{C}^\infty_{\mathrm{dlg},r}(\Omega).\qedhere\]
\end{proof}

\begin{proof}[Proof of Proposition \ref{dmglocal}]
Combining Propositions \ref{Dv} and \ref{Db2}, we see that (\ref{GLdlg}) holds, and then Proposition \ref{dmglocal} follows from Corollary \ref{dmglocalpf}.
\end{proof}

\begin{Cor}\label{fine}
The sheaves $\mathcal{I}_V^i$ are fine sheaves as in \cite[Definition 3.13]{Kodaira05}.
\end{Cor}

\begin{proof}
By Proposition \ref{lgsections} $\mathcal{C}^\infty_{\mathrm{lg},r}$ contains $\mathcal{C}^\infty_{\Delta^n}$ as a subsheaf, and hence so does $\mathcal{C}^\infty_{\mathrm{dlg},r}$. Then Proposition \ref{dmglocal} implies that $\mathcal{C}^\infty_{\mathrm{dmg}}$ contains $\mathcal{C}^\infty_{\mathrm{Sh}_{\mathbb{K},\Sigma}}$, so the $\mathcal{C}^\infty_{\mathrm{dmg}}$-modules $\mathcal{I}_V^i$ are fine sheaves.
\end{proof}

\subsection{Differential forms on a \texorpdfstring{$\sigma$}{}-neighborhood}\label{sheaves3}

\begin{Def}\label{dlgforms}
On $\Delta^{n,r}$, denote $\theta_k:=d\overline{z}_k/\overline{z}_k\,(1\leq k\leq r)$ and $\theta_k:=d\overline{z}_k\,(r<k\leq n)$, and for $I=\{k_1<...<k_i\}\subseteq\{1,...,n\}$ denote $\theta_I:=\theta_{k_i}\!\wedge\!...\!\wedge\!\theta_{k_1}$. Let $\mathcal{A}^{0,i}_{\mathrm{dlg},r}$ be the subsheaf
\[\bigoplus_{I\subseteq\{1,...,n\},|I|=i}\mathcal{C}^\infty_{\mathrm{dlg},r}\theta_I\]
of $j_{n,r*}\mathcal{A}^{0,i}_{\Delta^{n,r}}\cong\bigoplus_{I\subseteq\{1,...,n\},|I|=i}j_{n,r*}\mathcal{C}^\infty_{\Delta^{n,r}}\theta_I$.
\end{Def}

\begin{Prop}\label{dmgforms}
Let $\Omega$ be a $\sigma$-neighborhood of some point in $\mathrm{Sh}_{\mathbb{K},\Sigma}$ with $\sigma$-coordinates $\underline{z}=(z_1,...,z_n)$. Then analogous to (\ref{sheafisom4}), there are isomorphisms
\[(\phi_\sigma^*j_*\mathcal{A}^{0,i}_{\mathrm{Sh}_\mathbb{K}})|_\Omega\cong(j_{\sigma*}\phi_F^*\mathcal{A}^{0,i}_{\mathrm{Sh}_\mathbb{K}})|_\Omega\cong j_{\Omega*}\underline{z}_{n,r}^*\mathcal{A}^{0,i}_{\Delta^{n,r}}\cong\underline{z}^*j_{n,r*}\mathcal{A}^{0,i}_{\Delta^{n,r}},\]
and $(\phi_\sigma^*\mathcal{A}^{0,i}_{\mathrm{dmg}})|_\Omega$ coincides with $(\underline{z}^*\mathcal{A}^{0,i}_{\mathrm{dlg},r})|_\Omega$.
\end{Prop}

\begin{proof}
Let's first reinterpret Proposition \ref{sheafB} (ii): comparing the case $i=0$ and general cases and tracking through (\ref{modisom}), (\ref{ringisom1}) and (\ref{Ai}), we see it is stated that for any open subset $U\subseteq\Omega$ and $(0,i)$-form $\varpi$ on $j_\sigma^{-1}(U)$, $\varpi\in\phi_\sigma^*\mathcal{A}^{0,i}_{\mathrm{dmg}}(U)$ if and only if for every set of vector fields $X_1,...,X_i$ on $j_\sigma^{-1}(U)$ coming from left-invariant vector fields on $\mathscr{B}(\mathbb{R})^\circ$,
\[\varpi(X_1\!\wedge\!...\!\wedge\!X_i)\in\phi_\sigma^*\mathcal{C}^\infty_{\mathrm{dmg}}(U).\]
Recall that elements of $\underline{D}_\mathfrak{b}$ are such vector fields and they span the tangent space at each point, thus the condition above is equivalent to
\[\varpi(X_1\!\wedge\!...\!\wedge\!X_i)\in\phi_\sigma^*\mathcal{C}^\infty_{\mathrm{dmg}}(U),\forall X_1,...,X_i\in\underline{D}_\mathfrak{b}.\]
By Proposition \ref{dmglocal} and (\ref{GLdlg}), this is equivalent to
\[\varpi(Z_1\!\wedge\!...\!\wedge\!Z_i)\in\underline{z}^*\mathcal{C}^\infty_{\mathrm{dlg},r}(U),\forall Z_1,...,Z_i\in\underline{D}_z,\]
As $\{\theta_k\}_{1\leq k\leq n}$ are dual to the anti-holomorphic tangent vector fields in $\underline{D}_z$, the above condition means precisely that $\varpi\in\underline{z}^*\mathcal{A}^{0,i}_{\mathrm{dlg},r}(U)$.
\end{proof}

\section{A fine resolution of the canonical extension}\label{section3}

Now we show that $(\mathcal{I}_V^*,j_*\overline\partial)$ forms a resolution of $\widetilde{V}^{\mathrm{can}}$. We first reduce to the case where $V$ is trivial by trivialising $(\mathcal{I}_V^*,j_*\overline\partial)$ in terms of $\widetilde{V}^{\mathrm{can}}$, then work with several complex variables and prove a Dolbeault lemma with growth conditions.

\subsection{Trivialisations}\label{fineresol1}

Recall that $\widetilde{V}^{\mathrm{can}}$ and $\mathcal{I}_V^i$ are by definition subsheaves of $j_*\widetilde{V}$ and $j_*(\widetilde{V}\otimes_{\mathcal{O}_{\mathrm{Sh}_\mathbb{K}}}\mathcal{A}^{0,i}_{\mathrm{Sh}_\mathbb{K}})$ respectively (in particular $\mathcal{A}^{0,i}_{\mathrm{dmg}}$ is a subsheaf of $j_*\mathcal{A}^{0,i}_{\mathrm{Sh}_\mathbb{K}}$). Besides by Proposition \ref{dmglocal}, $\mathcal{O}_{\mathrm{Sh}_{\mathbb{K},\Sigma}}$ is a subsheaf of $\mathcal{C}^\infty_{\mathrm{dmg}}$.

\begin{Prop}\label{triv}
The natural morphisms \[\widetilde{V}^{\mathrm{can}}\otimes_{\mathcal{O}_{\mathrm{Sh}_{\mathbb{K},\Sigma}}}\mathcal{A}^{0,i}_{\mathrm{dmg}}\rightarrow j_*(\widetilde{V}\otimes_{\mathcal{O}_{\mathrm{Sh}_\mathbb{K}}}\mathcal{A}^{0,i}_{\mathrm{Sh}_\mathbb{K}})\]
identify $\widetilde{V}^{\mathrm{can}}\otimes_{\mathcal{O}_{\mathrm{Sh}_{\mathbb{K},\Sigma}}}\mathcal{A}^{0,i}_{\mathrm{dmg}}$ with $\mathcal{I}_V^i$.
\end{Prop}

The statement is of local nature and trivial for the trivial representation $V=\mathbb{C}$. We access the general case by locally trivialising $j_*(\widetilde{V}\otimes_{\mathcal{O}_{\mathrm{Sh}_\mathbb{K}}}\mathcal{A}^{0,i}_{\mathrm{Sh}_\mathbb{K}})$. For every set $S$, denote by $\underline{S}$ the corresponding constant sheaves on topological spaces.

\begin{Prop}
Let $\Omega\subseteq(U(F)_\mathbb{Z}\backslash D)_\sigma$ be a $\sigma$-neighborhood of a point, then there are isomorphisms
\begin{equation}\label{trivj}
(\phi_\sigma^*j_*(\widetilde{V}\otimes_{\mathcal{O}_{\mathrm{Sh}_\mathbb{K}}}\mathcal{A}^{0,i}_{\mathrm{Sh}_\mathbb{K}}))|_\Omega\cong(\phi_\sigma^*j_*\mathcal{A}^{0,i}_{\mathrm{Sh}_\mathbb{K}})|_\Omega\otimes_{\underline{\mathbb{C}}}\underline{V}
\end{equation}
compatible with the differentials and identifying subsheaves
\begin{equation}\label{trivI}
(\phi_\sigma^*\mathcal{I}_V^i)|_\Omega\cong(\phi_\sigma^*\mathcal{A}^{0,i}_{\mathrm{dmg}})|_\Omega\otimes_{\underline{\mathbb{C}}}\underline{V}
\end{equation}
and (when $i=0$)
\begin{equation}\label{trivcan}
(\phi_\sigma^*\widetilde{V}^{\mathrm{can}})|_\Omega\cong
(\phi_\sigma^*\mathcal{O}_{\mathrm{Sh}_{\mathbb{K},\Sigma}})|_\Omega\otimes_{\underline{\mathbb{C}}}\underline{V}.
\end{equation}
\end{Prop}

\begin{proof}
Denote by $V_{tr}$ a trivial $P_h(\mathbb{C})$-representation of the same dimensions as $V$, then from (\ref{A0i}) and (\ref{partial}) we see that isomorphisms (\ref{trivj}) that are compatible with the differentials will result from a compatible system of $(\mathfrak{p}_o,K_o)$-equivariant isomorphisms
\begin{equation}\label{trivo}
C^\infty(\pi_o^{-1}(\phi_\sigma(U)))^{K_o\mbox{-}\mathrm{fin}}\otimes V\cong C^\infty(\pi_o^{-1}(\phi_\sigma(U)))^{K_o\mbox{-}\mathrm{fin}}\otimes V_{tr}  
\end{equation}
for all open subsets $U\subseteq\Omega$, and by Definition \ref{mgsheaves} we have (\ref{trivI}) if therein
\[C^\infty_{\mathrm{dmg}}(\pi_o^{-1}(\phi_\sigma(U)))^{K_o\mbox{-}\mathrm{fin}}\otimes V\cong C^\infty_{\mathrm{dmg}}(\pi_o^{-1}(\phi_\sigma(U)))^{K_o\mbox{-}\mathrm{fin}}\otimes V_{tr}.\]
By (\ref{ringisom1}) we can replace $C^\infty_*(\pi_o^{-1}(\phi_\sigma(U)))^{K_o\mbox{-}\mathrm{fin}}$ above with $C^\infty_*(\pi_F^{-1}(U))^{K_o\mbox{-}\mathrm{fin}}$.
Note that $g\mapsto a_g^{-1}g$ induces a bijection
$U(F)_\mathbb{Z}\backslash G(\mathbb{R})/A_G(\mathbb{R})^\circ\xrightarrow{\sim}U(F)_\mathbb{Z}\backslash G(\mathbb{R})^1$
and thus $\pi_F^{-1}(U)$ can be viewed as a subset of the latter. Let $\Omega'$ be as in Definition \ref{sigma} and take a $t_P$ as in (\ref{t1}). Consider the map
\begin{equation}\label{tP1}
\begin{gathered}
C^\infty(\pi_F^{-1}(U))\otimes V\rightarrow C^\infty(\pi_F^{-1}(U))\otimes V_{tr},\\
f\mapsto(g\mapsto t_P(a_g^{-1}g)f(g)),
\end{gathered}
\end{equation}
it is clearly a $K_o$-module isomorphism, and also $\mathfrak{p}_o$-equivariant as the computation (\ref{hol3}) shows. By construction $t_P$ is pulled back from $T(F)\backslash\Psi_F^{-1}(\Omega')\subseteq U(F)_\mathbb{C}\backslash G(\mathbb{C})$, in which we claim the image of $\pi_F^{-1}(U)$ is relatively compact. In fact, since $\pi_F^{-1}(U)$ is a $K_o$-bundle over $j_\sigma^{-1}(U)$, it suffices to check the relative compactness of $p_F(j_\sigma^{-1}(U))$ in $T(F)\backslash\Psi_F^{-1}(\Omega')/P_h(\mathbb{C})\cong\Omega'$, but we have \[p_F(j_\sigma^{-1}(U))\subseteq p_\sigma(U)\subseteq p_\sigma(\Omega)\]
and by choice $p_\sigma(\Omega)$ is relatively compact in $\Omega'$. As left-invariant vector fields on $G(\mathbb{C})$ descends to $U(F)_\mathbb{C}\backslash G(\mathbb{C})$, analogous to the proof of Proposition \ref{Dv} we have
\[\forall g\in\mathcal{O}_{P_h(\mathbb{C})},g(t_P(a_{(\cdot)}^{-1}\,\cdot))|_{\pi_F^{-1}(U)}\in C^\infty_{\mathrm{dmg}}(\pi_F^{-1}(U)).\]
Therefore $C^\infty_{\mathrm{dmg}}(\pi_F^{-1}(U))\otimes V$ corresponds to $C^\infty_{\mathrm{dmg}}(\pi_F^{-1}(U))\otimes V_{tr}$.

To deal with $\widetilde{V}^{\mathrm{can}}$ it's more convenient to think of $\mathfrak{p}_h$ and $K_h$. When $i=0$ (\ref{trivj}) amounts to the part \[(C^\infty(\pi_o^{-1}(\phi_\sigma(U))\otimes V)^{K_o}\cong C^\infty(\pi_o^{-1}(\phi_\sigma(U)))^{K_o}\otimes V_{tr}\]
of (\ref{trivo}) and can be transformed into
\[(C^\infty(\pi^{-1}(\phi_\sigma(U))\otimes V)^{K_h}\cong C^\infty(\pi^{-1}(\phi_\sigma(U)))^{K_h}\otimes V_{tr}\]
via $f\mapsto(g\mapsto a_gf(g))$. Let $\widetilde{\pi}_F:U(F)_\mathbb{Z}\backslash G(\mathbb{R})^+\rightarrow U(F)_\mathbb{Z}\backslash D$ be the quotient map, then analogous to (\ref{ringisom1}) we have $C^\infty(\pi^{-1}(\phi_\sigma(U)))\cong C^\infty(\widetilde{\pi}_F^{-1}(U))$ and (\ref{tP1}) will be transformed into
\begin{equation}\label{tP2}
\begin{gathered}
C^\infty(\widetilde{\pi}_F^{-1}(U))\otimes V\xrightarrow{\sim}C^\infty(\widetilde{\pi}_F^{-1}(U))\otimes V_{tr},\\
f\mapsto(g\mapsto t_P(g)f(g)).
\end{gathered}
\end{equation}
On the other hand, by definition (\ref{EVF}) $\mathcal{E}_{V,F}$ is the sheaf
\[U\mapsto(\mathcal{O}(T(F)\backslash\Psi_F^{-1}(U))\otimes V)^{P_h(\mathbb{C})}\]
and a trivialisation $\mathcal{E}_{V,F}|_{\Omega'}\cong\mathcal{O}_{\Omega'}\otimes_{\underline{\mathbb{C}}}\underline{V}$ will arise from $P_h(\mathbb{C})$-equivariant maps
\begin{equation}\label{tP3}
\begin{gathered}
C^\infty(T(F)\backslash\Psi_F^{-1}(U))\otimes V\xrightarrow{\sim}C^\infty(T(F)\backslash\Psi_F^{-1}(U))\otimes V_{tr},\\
f\mapsto(g\mapsto t_P(g)f(g)),
\end{gathered}
\end{equation}
which are well-defined as $t_P$ is $T(F)$-invariant. The identification (\ref{trivcan}) now follows from Definition \ref{Vcan} and the consistency between the formulas (\ref{tP2}) and (\ref{tP3}).
\end{proof}

\subsection{Several complex variables and a Dolbeault lemma}\label{fineresol2}

\begin{Prop}\label{resoln}
$(\mathcal{I}_V^*,j_*\overline\partial)$ makes a resolution of $\widetilde{V}^{\mathrm{can}}$.
\end{Prop}

As $\widetilde{V}^{\mathrm{can}}$ is locally free, Proposition \ref{triv} reduces the problem to the exactness of
\[0\rightarrow\mathcal{O}_{\mathrm{Sh}_\mathbb{K},\Sigma}\rightarrow\mathcal{C}^\infty_\mathrm{dmg}\xrightarrow{j_*\overline\partial}\mathcal{A}^{0,1}_\mathrm{dmg}\xrightarrow{j_*\overline\partial}\mathcal{A}^{0,2}_\mathrm{dmg}\xrightarrow{j_*\overline\partial}...,\]
and in turn Propositions \ref{dmgforms} reduces it to the exactness of
\[0\rightarrow\mathcal{O}_{\Delta^n}\rightarrow\mathcal{C}^\infty_{\mathrm{dlg},r}\xrightarrow{\overline\partial}\mathcal{A}^{0,1}_{\mathrm{dlg},r}\xrightarrow{\overline\partial}\mathcal{A}^{0,2}_{\mathrm{dlg},r}\xrightarrow{\overline\partial}....\]
We deal with the exactness at $\mathcal{O}_{\Delta^n}$ and elsewhere separately. 

\begin{Lem}
$\mathcal{O}_{\Delta^n}$ is the kernel of $\overline\partial:\mathcal{C}^\infty_{\mathrm{dlg},r}\rightarrow\mathcal{A}^{0,1}_{\mathrm{dlg},r}$.
\end{Lem}

\begin{proof}
$\mathcal{O}_{\Delta^n}$ is contained in the kernel as we have $\mathcal{C}^\infty_{\Delta^n}\subseteq\mathcal{C}^\infty_{\mathrm{dlg},r}$. On the other hand we show that $\mathcal{O}_{\Delta^n}$ contains the kernel of
\[\overline\partial:\mathcal{C}^\infty_{\mathrm{lg},r}\rightarrow j_{n,r*}\mathcal{A}^{0,1}_{\Delta^{n,r}}.\]
For any open subset $U\subseteq\Delta^n$ and holomorphic function $f$ on $j_{n,r}^{-1}(U)$ that is bounded by a polynomial in $\log{\frac{1}{|z_1|}},...,\log{\frac{1}{|z_r|}}$, by the Riemann extension theorem \cite[Theorem 1.15]{Kodaira05} the bounded holomorphic function $z_1...z_rf$ on $j_{n,r}^{-1}(U)$ extends to a holomorphic function on $U$, which should vanish along $U\backslash j_{n,r}^{-1}(U)$. Applying the Weierstrass division theorem inductively with Weierstrass polynomials $z_1,...,z_n$ we get that $f$ itself also extends to a holomorphic function on $U$. Therefore $\mathcal{O}_{\Delta^n}$ contains the subsheaf of $j_{n,r*}\mathcal{C}^\infty_{\Delta^{n,r}}$ generated by such $f$, which is exactly the kernel of $\overline\partial:\mathcal{C}^\infty_{\mathrm{lg},r}\rightarrow j_{n,r*}\mathcal{A}^{0,1}_{\Delta^{n,r}}$.
\end{proof}

\begin{Rmk}
So far we've shown that $\widetilde{V}^{\mathrm{can}}$ is the sheafification of the presheaf
\[U\mapsto(C^\infty_{\mathrm{mg}}(\pi_o^{-1}(U))\otimes V)^{(\mathfrak{p}_o,K_o)}.\]
Combining with Proposition \ref{mgsections} we get the degree $0$ case of our main theorem, which has also been proved differently as \cite[Proposition 3.3]{Mumford77}. Besides, one can also use the formula to make an alternative proof of the existence of $\widetilde{V}^{\mathrm{can}}$ when $\mathrm{Sh}_{\mathbb{K},\Sigma}$ is SNC. For these we only need Proposition \ref{mglocal} rather than the full power of Proposition \ref{dmglocal}.
\end{Rmk}

\begin{Lem}\label{dlgr}
On $\Delta^n$ the complex of sheaves
\[\mathcal{A}^{0,0}_{\mathrm{dlg},r}\xrightarrow{\overline\partial}\mathcal{A}^{0,1}_{\mathrm{dlg},r}\xrightarrow{\overline\partial}\mathcal{A}^{0,2}_{\mathrm{dlg},r}\xrightarrow{\overline\partial}...\]
is exact.
\end{Lem}

\begin{Rmk}
The statement of the lemma has appeared in \cite[2.4]{Harris88}, however it seems the arguments given and cited there are insufficient to prove the multi-variable cases of the lemma. In fact to proceed by induction on the number of variables, rather than the single-variable case itself one actually needs an ``effective version'' of it (Corollaries \ref{dlg10} and \ref{dlg11}), which we don't manage to find in the literature. We present here a complete proof of the lemma with all details so as to provide a self-contained reference on the topic.
\end{Rmk}

We make some preparations:

\begin{Lem}\label{HP}
For every $0<\rho<\frac{1}{3}$, suppose $g\in C^\infty(\Delta^*(\rho))$ satisfies
\[|g(z)|\leq c(\log{\frac{1}{|z|}})^N,\forall z\in\Delta^*(\rho)\]
for some $c>0$ and $N\in\mathbb{R}$, then
\[f(z):=\frac{1}{2\pi i}\int_{\Delta^*(\rho)}\frac{g(w)}{\overline{w}(w-z)}\,dw\!\wedge\!d\overline{w}\]
is a solution to $\overline{z}\frac{\partial}{\partial\overline{z}}f=g$ in $C^\infty(\Delta^*(\rho))$ and there is a $d=d(N)>0$ such that
\[|f(z)|\leq\begin{cases}
dc(\log{\frac{1}{|z|}})^{N+1}&(N>-1)\\
dc\log{\log{\frac{1}{|z|}}}&(N=-1),\\
dc&(N<-1)
\end{cases}\forall z\in\Delta^*(\rho).\]
\end{Lem}

\begin{proof}
For the first part of the lemma we adapt the proof of \cite[Lemma 3.1]{Kodaira05}. For every $z_0\in\Delta^*(\rho)$, take an $\varepsilon>0$ such that $\Delta(z_0,3\varepsilon)\subseteq\Delta^*(\rho)$ and a partition of unity $\{\chi_a,\chi_b\}$ subordinate to the open cover $\{U_a:=\overline{\Delta(z_0,\varepsilon)}^c,U_b:=\Delta(z_0,2\varepsilon)\}$ of $\Delta^*(\rho)$. Then formally $f(z)$ is the sum of
\[I_a(z):=\frac{1}{2\pi i}\int_{U_a}\frac{\chi_a(w)g(w)}{\overline{w}(w-z)}\,dw\!\wedge\!d\overline{w}\]
and
\[\begin{aligned}
I_b(z)&:=\frac{1}{2\pi i}\int_{U_b}\frac{\chi_b(w)g(w)}{\overline{w}(w-z)}\,dw\!\wedge\!d\overline{w}\\
&\:=\frac{1}{2\pi i}\int_{\mathbb{C}}\frac{\chi_b(w+z)g(w+z)}{(\overline{w}+\overline{z})w}\,dw\!\wedge\!d\overline{w},
\end{aligned}\]
where $\chi_bg$ is viewed as a smooth function on $\mathbb{C}$ supported in $U_b$. The bound on $g$ implies when $z\in\Delta(z_0,\varepsilon)$, every partial derivative of either integrand with respect to $z,\overline{z}$ is $L^1$ in $w$, and the condition for applying the dominated convergence theorem to interchange the integration with a $\frac{\partial}{\partial z}$ or $\frac{\partial}{\partial\overline{z}}$ is always satisfied, so $I_a,I_b$ and hence $f$ are well-defined smooth functions. Moreover on $\Delta(z_0,\varepsilon)$ we see $I_a$ is holomorphic while it is shown in the cited proof that $\overline{z}\frac{\partial}{\partial\overline{z}}I_b=g$, so $f$ satisfies the equation.

Next for a fixed $z\in\Delta^*(\rho)$, set
\[\Delta_1:=\Delta^*(\frac{|z|}{2}),\Delta_2:=\Delta(z,\frac{|z|}{2})\cap\Delta^*(\rho),\Delta_3:=\Delta^*(\rho)\backslash(\Delta_1\cup\Delta_2)\]
and write $\lambda:=|w|$ for $w\in\Delta^*(\rho)$. Then $|w-z|\geq\frac{|z|}{2}$ for $w\in\Delta_1$, so
\begin{equation}\label{ineq}
|\frac{1}{2\pi i}\int_{\Delta_1}\frac{g(w)}{\overline{w}(w-z)}\,dw\!\wedge\!d\overline{w}|\leq\frac{4c}{|z|}\int_0^\frac{|z|}{2}(\log{\frac{1}{\lambda}})^N\,\mathrm{d}\lambda.
\end{equation}
When $N\leq 0$, the right hand side $\leq 2c(\log{\frac{2}{|z|}})^N<2^{N+1}c(\log{\frac{1}{|z|}})^N$; when $N>0$, we have
\[(\log{\frac{2}{|z|}})^{\lceil N\rceil-N}\int_0^\frac{|z|}{2}(\log{\frac{1}{\lambda}})^N\,\mathrm{d}\lambda\leq\int_0^\frac{|z|}{2}(\log{\frac{1}{\lambda}})^{\lceil N\rceil}\,\mathrm{d}\lambda=\frac{|z|}{2}p_{\lceil N\rceil}(\log{\frac{2}{|z|}}),\]
where $\lceil N\rceil$ is the least integer greater than or equal to $N$ and $p_{\lceil N\rceil}$ is a degree $\lceil N\rceil$ polynomial depending on $\lceil N\rceil$, so there is a $d'=d'(\lceil N\rceil)>0$ such that the right hand side of (\ref{ineq}) $\leq d'c(\log{\frac{1}{|z|}})^N$.

Over $\Delta_2$, we have
\[\begin{aligned}
|\frac{1}{2\pi i}\int_{\Delta_2}\frac{g(w)}{\overline{w}(w-z)}\,dw\!\wedge\!d\overline{w}|&\leq\max_{w\in\Delta_2}|\frac{g(w)}{w}|\cdot\frac{1}{2\pi}\int_{\Delta^*(\frac{|z|}{2})}\frac{|du\!\wedge\!d\overline{u}|}{|u|}\\
&\leq 2c(\log{\frac{2}{|z|}})^N<2^{N+1}c(\log{\frac{1}{|z|}})^N.
\end{aligned}\]

On $\Delta_3$ we have $|w-z|\geq\frac{|w|}{3}$, so
\[\begin{aligned}
|\frac{1}{2\pi i}\int_{\Delta_1}\frac{g(w)}{\overline{w}(w-z)}\,dw\!\wedge\!d\overline{w}|&\leq\frac{3}{2\pi}\int_{\Delta_2\cup\Delta_3}\frac{|g(w)|}{|w|^2}\,|dw\!\wedge\!d\overline{w}|\\
&\leq 6c\int_\frac{|z|}{2}^\rho(\log{\frac{1}{\lambda}})^N\,\frac{\mathrm{d}\lambda}{\lambda},
\end{aligned}\]
while
\[\int_\frac{|z|}{2}^\rho(\log{\frac{1}{\lambda}})^N\,\frac{\mathrm{d}\lambda}{\lambda}<\begin{cases}
\frac{1}{N+1}(\log{\frac{2}{|z|}})^{N+1}<\frac{2^{N+1}}{N+1}(\log{\frac{1}{|z|}})^{N+1}&(N>-1)\\
\log{\log{\frac{2}{|z|}}}<(1+\log_{\log{3}}{2})\log{\log{\frac{1}{|z|}}}&(N=-1)\\
\frac{1}{N+1}(\log{\frac{1}{\rho}})^{N+1}<\frac{1}{N+1}&(N<-1).
\end{cases}\]
Combining these estimates makes the second part of the lemma.
\end{proof}

\begin{Rmk}
The estimates above essentially all come from the d\'{e}monstration of \cite[Lemme 1]{HP86}.
\end{Rmk}

\begin{Lem}\label{harmonic}
$(i)$ For any $0<\rho'<\rho<1$ there is a $d_1=d_1(\rho',\rho)>0$ such that for every harmonic function $h$ on $\Delta(0,\rho)$ bounded by a constant $c>0$, it holds that
\[|\frac{\partial h}{\partial z}(z)|\leq d_1c,\forall z\in\Delta(0,\rho');\]
$(ii)$ For any $0<\rho'<\rho<\frac{1}{3}$ and $N\in\mathbb{R}$ there is a $d_2=d_2(\rho',\rho,N)>0$ such that for every harmonic function $h$ on  $\Delta^*(\rho)$ such that
\[|h(z)|\leq c(\log{\frac{1}{|z|}})^N,\forall z\in\Delta^*(\rho)\]
for some $c>0$, it holds that
\[|z\frac{\partial h}{\partial z}(z)|\leq d_2c(\log{\frac{1}{|z|}})^N,\forall z\in\Delta^*(\rho').\]
\end{Lem}

\begin{proof}
For every $z_0$ in the domain of $h$, take an $\varepsilon>0$ such that $h$ is defined on $\Delta(z_0,2\varepsilon)$, then Poisson's formula
\[h(z)=\frac{1}{2\pi}\int_{|w-z_0|=\varepsilon}\frac{\varepsilon^2-|z-z_0|^2}{|w-z|^2}h(w)\,\mathrm{d}\theta,\forall z\in\Delta(z_0,\varepsilon)\]
implies
\[\frac{\partial h}{\partial z}(z)=\frac{1}{2\pi}\int_{|w-z_0|=\varepsilon}\frac{\varepsilon^2-(w-z_0)(\overline{z}-\overline{z}_0)}{(w-z)^2(\overline{w}-\overline{z})}h(w)\,\mathrm{d}\theta,\forall z\in\Delta(z_0,\varepsilon)\]
and in particular
\[\frac{\partial h}{\partial z}(z_0)=\frac{1}{2\pi}\int_{|w-z_0|=\varepsilon}\frac{h(w)}{(w-z_0)}\,\mathrm{d}\theta.\]
Therefore
\[|\frac{\partial h}{\partial z}(z_0)|\leq\frac{1}{\varepsilon}\max_{|w-z_0|=\varepsilon}{|h(w)|}.\]
In the first case, for $z_0\in\Delta(0,\rho')$ we can take $\varepsilon=\frac{\rho-\rho'}{2}$, so $d_1:=\frac{2}{\rho-\rho'}$ works; in the second case, for $z_0\in\Delta^*(\rho')$ we take $\varepsilon=(\rho-\rho')|z_0|$, then $d_2:=\frac{2^N}{\rho-\rho'}$ satisfies that
\[|z_0\frac{\partial h}{\partial z}(z_0)|<\frac{1}{\rho-\rho'}(\log{\frac{2}{|z_0|}})^N<d_2c(\log{\frac{1}{|z_0|}})^N.\qedhere\]
\end{proof}

\begin{Cor}\label{dlg10}
For every $0<\rho<1$, suppose $g\in C^\infty(\Delta(0,\rho))$ satisfies that for any $p,q\geq 0$,
\[|(\frac{\partial}{\partial z})^p(\frac{\partial}{\partial\overline{z}})^qg|\]
is bounded by a constant $c_{p,q}>0$. Write $\underline{c}_{p,q}:=(c_{p',q'})_{p'\leq p,q'\leq q}$, then for every $0<\rho'<\rho$ there are $\widetilde{c}_{p,q}=\widetilde{c}_{p,q}(\rho',\rho,\underline{c}_{p,q})>0$ depending linearly on $\underline{c}_{p,q}$ such that
\[f(z):=\frac{1}{2\pi i}\int_{\Delta(0,\rho)}\frac{g(w)}{w-z}\,dw\!\wedge\!d\overline{w}\]
satisfies that for any $p,q\geq 0$, $|(\frac{\partial}{\partial z})^p(\frac{\partial}{\partial\overline{z}})^qf|$ is bounded by $\widetilde{c}_{p,q}$ on $\Delta(0,\rho')$.
\end{Cor}

\begin{proof}
\cite[Lemma 3.1]{Kodaira05} implies that $\frac{\partial}{\partial\overline{z}}f=g$, so we just need to bound each $|(\frac{\partial}{\partial z})^pf|$. $\frac{\partial}{\partial z}$ couldn't be directly exchanged with the integration, so instead we induct on $p$. We abbreviate any subscript $p,0$ as $p$.

When $p=0$ the estimate is direct and $\widetilde{c}_0$ is independent of $\rho'$. Suppose we've found $\widetilde{c}_p$, then $|(\frac{\partial}{\partial z})^pf|$ is bounded by $C_1:=\widetilde{c}_p(\frac{\rho'+\rho}{2},\rho,\underline{c}_p)$ on $\Delta(0,\frac{\rho'+\rho}{2})$. On $\Delta(0,\rho)$, define
\[f_{p+1}(z):=\frac{1}{2\pi i}\int_{\Delta(0,\rho)}\frac{(\frac{\partial}{\partial w})^{p+1}g(w)}{w-z}\,dw\!\wedge\!d\overline{w}\]
and
\[\widetilde{f}_p(z):=\frac{1}{2\pi i}\int_{\Delta(0,\rho)}\frac{f_{p+1}(w)}{\overline{w}-\overline{z}}\,dw\!\wedge\!d\overline{w},\]
then we have $\frac{\partial}{\partial z}\widetilde{f}_p=f_{p+1}$ and $\frac{\partial}{\partial\overline{z}}f_{p+1}=(\frac{\partial}{\partial z})^{p+1}g$,
so
\[h:=(\frac{\partial}{\partial z})^pf-\widetilde{f}_p\]
is a harmonic function. Besides $f_{p+1}$ and $\widetilde{f}_p$ are bounded by $C_2:=\widetilde{c}_0(\rho,c_{p+1})$ and $C_3:=\widetilde{c}_0(\rho,C_2)$ respectively, so $h$ is bounded by $C_1+C_3$ on $\Delta(0,\frac{\rho'+\rho}{2})$, and Lemma \ref{harmonic} (i) implies that $|(\frac{\partial}{\partial z})^{p+1}f-f_{p+1}|=|\frac{\partial}{\partial z}h|$ is bounded by $d_1(\rho',\frac{\rho'+\rho}{2})(C_1+C_3)$ on $\Delta(0,\rho')$. Therefore we can take
\[\widetilde{c}_{p+1}:=d_1(\rho',\frac{\rho'+\rho}{2})(C_1+C_3)+C_2.\qedhere\]
\end{proof}

\begin{Cor}\label{dlg11}
In Lemma \ref{HP}, suppose that for any $p,q\geq 0$ there are $c_{p,q}>0$ and $N_{p,q}\in\mathbb{R}$ such that
\[|(z\frac{\partial}{\partial z})^p(\overline{z}\frac{\partial}{\partial\overline{z}})^qg(z)|\leq c_{p,q}(\log{\frac{1}{|z|}})^{N_{p,q}},\forall z\in\Delta^*(\rho).\]
Write $\underline{c}_{p,q}:=(c_{p',q'})_{p'\leq p,q'\leq q},\underline{N}_{p,q}:=(N_{p',q'})_{p'\leq p,q'\leq q}$, then for every $0<\rho'<\rho$ there are $\widetilde{N}_{p,q}\in\mathbb{R}$ depending on $\underline{N}_{p,q}$ and $\widetilde{c}_{p,q}=\widetilde{c}_{p,q}(\rho',\rho,\underline{c}_{p,q},\underline{N}_{p,q})>0$ depending linearly on $\underline{c}_{p,q}$ such that for anny $p,q\geq 0$,
\[|(z\frac{\partial}{\partial z})^p(\overline{z}\frac{\partial}{\partial\overline{z}})^qf(z)|\leq\widetilde{c}_{p,q}\cdot(\log{\frac{1}{|z|}})^{\widetilde{N}_{p,q}},\forall z\in\Delta^*(\rho').\]
\end{Cor}

\begin{proof}
As $\overline{z}\frac{\partial}{\partial\overline{z}}f=g$ it suffices to bound each $|(z\frac{\partial}{\partial z})^pf(z)|$. Again we induct on $p$ and abbreviate any subscript $p,0$ as $p$. When $p=0$ the estimate follows from Lemma \ref{HP}. Suppose we've found $\widetilde{c}_p$ and $\widetilde{N}_p$, then $C_1:=\widetilde{c}_p(\frac{\rho'+\rho}{2},\rho,\underline{c}_p,\underline{N}_p)$ satisfies
\[|(z\frac{\partial}{\partial z})^pf(z)|\leq C_1(\log{\frac{1}{|z|}})^{\widetilde{N}_p},\forall z\in\Delta^*(\frac{\rho'+\rho}{2}).\]
By Lemma \ref{HP}, there is an $f_{p+1}\in C^\infty(\Delta^*(\rho))$ such that $\overline{z}\frac{\partial}{\partial\overline{z}}f_{p+1}=(z\frac{\partial}{\partial z})^{p+1}g$ and
\[|f_{p+1}(z)|\leq C_2(\log{\frac{1}{|z|}})^{N'_{p+1}},\forall z\in\Delta^*(\rho),\]
where $C_2:=d(N_{p+1})c_{p+1}$ and $N'_{p+1}=\max\{N_{p+1}+1,1\}$. Then by the complex conjugate of Lemma \ref{HP} there is an $\widetilde{f}_p\in C^\infty(\Delta^*(\rho))$ such that $z\frac{\partial}{\partial z}\widetilde{f}_p=f_{p+1}$ and
\[|\widetilde{f}_p(z)|\leq C_3(\log{\frac{1}{|z|}})^{N'_{p+1}+1},\forall z\in\Delta^*(\rho)\]
for $C_3:=d(N'_{p+1})C_2$. Define $\widetilde{N}_{p+1}:=\max\{\widetilde{N}_p,N'_{p+1}+1\}$, then in the three inequalities above the exponents of $\log{\frac{1}{|z|}}$ can all be replaced by $\widetilde{N}_{p+1}$, so
\[h:=(z\frac{\partial}{\partial z})^pf-\widetilde{f}_p\]
is bounded by $(C_1+C_3)(\log{\frac{1}{|z|}})^{\widetilde{N}_{p+1}}$ on $\Delta^*(\frac{\rho'+\rho}{2})$. Meanwhile $h$ is harmonic, so by Lemma \ref{harmonic} (ii) we have
\[|(z\frac{\partial}{\partial z})^{p+1}f(z)-f_{p+1}(z)|=|z\frac{\partial h}{\partial z}(z)|\leq d_2\cdot(C_1+C_3)(\log{\frac{1}{|z|}})^{\widetilde{N}_{p+1}},\forall z\in\Delta^*(\rho'),\]
where $d_2=d_2(\rho',\frac{\rho'+\rho}{2},\widetilde{N}_{p+1})$. Therefore we can take $\widetilde{c}_{p+1}:=d_2\cdot(C_1+C_3)+C_2$.
\end{proof}

\begin{Rmk}
Qualitatively, the induction argument in the above proofs has been indicated in \cite[2.3.2]{Harris88}.
\end{Rmk}

\begin{proof}[Proof of Lemma \ref{dlgr}]
We want to show the exactness of the sequence
\[(\mathcal{A}^{0,0}_{\mathrm{dlg},r})_{\underline{z}}\xrightarrow{\overline\partial}(\mathcal{A}^{0,1}_{\mathrm{dlg},r})_{\underline{z}}\xrightarrow{\overline\partial}(\mathcal{A}^{0,2}_{\mathrm{dlg},r})_{\underline{z}}\xrightarrow{\overline\partial}...\]
for each $\underline{z}=(z_1,...,z_n)\in\Delta^n$. Without loss of generality assume $z_1=...=z_s=0$ and $z_{s+1}...z_r\neq 0$, then we observe that the translation by $-\underline{z}$ induces isomorphisms
\[(\mathcal{A}^{0,i}_{\mathrm{dlg},r})_{\underline{z}}\cong(\mathcal{A}^{0,i}_{dlg,s})_{\underline{0}}\]
that are compatible with the differentials, so it suffices to prove for $\underline{z}=\underline{0}=(0,...,0)$.

By Proposition \ref{lgsections} and Example \ref{Dz}, every germ $[\psi]\in(\mathcal{A}^{0,i}_{\mathrm{dlg},r})_{\underline{0}}$ in the kernel of $\overline\partial$ is represented by a $(0,i)$-form
\[\psi=\sum_{I\subseteq\{1,...,n\},|I|=i}\psi_I\theta_I\in\mathcal{A}^{0,i}(\Delta^{n,r}(\rho))\]
for some $0<\rho<\frac{1}{3}$ such that $\overline\partial\psi=0$ and for any $I$ and $D\in\mathfrak{U}(\underline{D}_z)$, $D\psi_I$ is bounded by a polynomial in $\log{\frac{1}{|z_1|}},...,\log{\frac{1}{|z_r|}}$. For brevity we refer to this growth condition as each $D\psi_I$ has \it log growth \rm or $\psi_I$ has \it d-log growth\rm, and note that Corollaries \ref{dlg10} and \ref{dlg11} then state that when $(n,r)=(1,0)$ and $(1,1)$ the d-log growth of $g$ implies the d-log growth of $f$ respectively. We imitate the proof of \cite[Theorem 3.3]{Kodaira05} and induct on the largest integer $m$ that appears in an $I$ with $\psi_I\neq 0$ to show that for every $0<\rho'<\rho$ there is a $(0,i-1)$-form
\[\eta=\sum_{|I'|=i-1}\eta_{I'}\theta_{I'}\in\mathcal{A}^{0,i-1}(\Delta^{n,r}(\rho'))\]
such that $\overline\partial\eta=\psi|_{\Delta^{n,r}(\rho')}$ and each $\eta_{I'}$ has d-log growth. When $m=0$, $\psi=0$, so there is nothing to prove. In general, define
\[\eta_1:=\sum_{I'\subseteq\{1,...,m-1\},|I'|=i-1}\eta_{1,I'}\theta_{I'}\in\mathcal{A}^{0,i-1}(\Delta^{n,r}(\frac{\rho'+\rho}{2})),\]
where
\[\eta_{1,I'}(\underline{z}):=\frac{1}{2\pi i}\int_{\Delta^*(\rho)}\frac{\psi_{I'\cup\{m\}}(...z_{m-1},w,z_{m+1},...,)}{\overline{w}(w-z_m)}\,dw\!\wedge\!d\overline{w}\:(m\leq r),\]
\[\eta_{1,I'}(\underline{z}):=\frac{1}{2\pi i}\int_{\Delta(0,\rho)}\frac{\psi_{I'\cup\{m\}}(...z_{m-1},w,z_{m+1},...,)}{w-z_m}\,dw\!\wedge\!d\overline{w}\:(m>r).\]
For every $z_0\in\Delta^*(\frac{\rho'+\rho}{2})$, take $\varepsilon,\chi_a$ and $\chi_b$ as in the proof of Lemma \ref{HP}, then when $z_m\in\Delta(z_0,\varepsilon)$ the d-log growth of $\psi_{I'\cup\{m\}}$ guarantees any partial derivatives of
\[\frac{\chi_a(w)\psi_{I'\cup\{m\}}(...w...)}{\overline{w}(w-z_m)},\frac{\chi_b(w+z_m)\psi_{I'\cup\{m\}}(...w+z_m...)}{(\overline{w}+\overline{z}_m)w}\]
and
\[\frac{\chi_a(w)\psi_{I'\cup\{m\}}(...w...)}{w-z_m},\frac{\chi_b(w+z_m)\psi_{I'\cup\{m\}}(...w+z_m...)}{w}\]
with respect to $z_1,\overline{z}_1,...,z_n,\overline{z}_n$ are $L^1$ in $w$, so $\eta_1$ is a smooth $(0,i-1)$-form.

As for the d-log growth of $\eta_{1,I'}$, denote $D_m:=z_m\frac{\partial}{\partial z_m},\overline{D}_m:=\overline{z}_m\frac{\partial}{\partial\overline{z}_m}$ when $m\leq r$ and $D_m:=\frac{\partial}{\partial z_m},\overline{D}_m:=\frac{\partial}{\partial\overline{z}_m}$ when $m>r$, then it suffices to check that
\[\forall D''\in\mathfrak{U}(D_m,\overline{D}_m),D'\in\mathfrak{U}(\underline{D}_z\backslash\{D_m,\overline{D}_m\}),D''D'\eta_{1,I'}\textrm{ has log growth}.\]
We can exchange $D'$ with the integration and $D'\psi_{I'\cup\{m\}}$ has log growth, so Corollary \ref{dlg10}, \ref{dlg11} imply the log growth of $D''D'\eta_{1,I'}$.

Now consider $\overline\partial\eta_1=\sum_{|I|=i}(\overline\partial\eta_1)_I\theta_I$. $\overline\partial\psi=0$ implies every $\psi_I$ is holomorphic in $z_{m+1},...,z_n$ and hence so is every $\eta_{1,I'}$, so $(\overline\partial\eta_1)_I=0$ unless $I\subseteq\{1,...,m\}$. Besides, Lemma \ref{HP} and \cite[Lemma 3.1]{Kodaira05} imply that $(\overline\partial\eta_1)_I=\psi_I$ when $m\in I$, so
\[\psi-\overline\partial\eta_1=\sum_{I\subseteq\{1,...,m-1\}}(\psi_I-(\overline\partial\eta_1)_I)\theta_I,\]
where each $\psi_I-(\overline\partial\eta_1)_I$ has d-log growth. By induction hypothesis a $(0,i-1)$-form
\[\eta_2=\sum_{|I'|=i-1}\eta_{2,I'}\theta_{I'}\in\mathcal{A}^{0,i-1}(\Delta^{n,r}(\rho'))\]
is there such that $\overline\partial\eta_2=(\psi-\overline\partial\eta_1)|_{\Delta^{n,r}(\rho')}$ and each $\eta_{2,I'}$ has d-log growth. Then $\eta=\eta_1+\eta_2$ is a $(0,i-1)$-form as desired.
\end{proof}

\begin{Prop}\label{Propdmg}
We have isomorphisms
\begin{equation}\label{dmg2}
H^i(\mathrm{Sh}_{\mathbb{K},\Sigma},\widetilde{V}^{\mathrm{can}})\cong H^i_{(\mathfrak{p}_o,K_o)}(C^\infty_{\mathrm{dmg}}(G)^\mathbb{K}\otimes V).
\end{equation}
\end{Prop}

\begin{proof}
By Proposition \ref{resoln} and Corollary \ref{fine}, $(\mathcal{I}_V^*,j_*\overline\partial)$ makes a fine resolution of $\widetilde{V}^{\mathrm{can}}$, and we've seen in Proposition \ref{mgsections} that the global sections of $(\mathcal{I}_V^*,j_*\overline\partial)$ can be identified the relative Chevalley-Eilenberg complex computing
\[H^*_{(\mathfrak{p}_o,K_o)}(C^\infty_{\mathrm{dmg}}(G)^\mathbb{K}\otimes V).\qedhere\]
\end{proof}

\begin{Rmk}
We see that Proposition \ref{Propdmg} reproves the isomorphisms of cohomology in Proposition \ref{cohisom} when $\Sigma$ and $\Sigma'$ are both SNC.
\end{Rmk}

\subsection{The Hecke action on coherent cohomology}\label{fineresol3}

\begin{Prop}\label{GAf}
(\ref{dmg2}) induces $G(\mathbb{A}_f)$-equivariant isomorphisms
\[\varinjlim_{t_e^*}H^i(\widetilde{V}^{\mathrm{can}}_\mathbb{K})\cong H^i_{(\mathfrak{p}_o,K_o)}(C^\infty_{\mathrm{dmg}}(G)\otimes V_o).\]
\end{Prop}

\begin{proof}
We wish to check for any $g\in G(\mathbb{A}_f)$ and open subgroup $\mathbb{K}'\subseteq g\mathbb{K}g^{-1}$ the commutativity of the following diagram:
\[\begin{tikzcd}
H^i(\widetilde{V}^{\mathrm{can}}_\mathbb{K})\rar{(\ref{dmg2})}[swap]{\sim}\dar{(\ref{tg*})}&H^i_{(\mathfrak{p}_o,K_o)}(C^\infty_{\mathrm{dmg}}(G)^\mathbb{K}\otimes V_o)\dar\\
H^i(\widetilde{V}^{\mathrm{can}}_{\mathbb{K}'})\rar{(\ref{dmg2})}[swap]{\sim}&H^i_{(\mathfrak{p}_o,K_o)}(C^\infty_{\mathrm{dmg}}(G)^{\mathbb{K}'}\otimes V_o),
\end{tikzcd}\]
where the right vertical map arises from
\[\begin{gathered}
C^\infty_{\mathrm{dmg}}(G)^\mathbb{K}\otimes V_o\hookrightarrow C^\infty_{\mathrm{dmg}}(G)^{\mathbb{K}'}\otimes V_o,\\
f\mapsto gf(\cdot\,g).
\end{gathered}\]
Let $j_\Sigma:\mathrm{Sh}_\mathbb{K}\hookrightarrow\mathrm{Sh}_{\mathbb{K},\Sigma}$ and $j_{\Sigma'}:\mathrm{Sh}_{\mathbb{K}'}\hookrightarrow\mathrm{Sh}_{\mathbb{K}',\Sigma'}$ be SNC compactifications such that $\Sigma'$ refines $\Sigma^g$ and abbreviate $t_{g,\Sigma}\circ r_{\Sigma',\Sigma^g}$ as $t_g$, then (\ref{tg*}) is the composition of maps
\[H^i(\mathrm{Sh}_{\mathbb{K},\Sigma},\widetilde{V}^{\mathrm{can}}_{\mathbb{K},\Sigma})\rightarrow H^i(\mathrm{Sh}_{\mathbb{K},\Sigma},t_{g*}\widetilde{V}^{\mathrm{can}}_{\mathbb{K}',\Sigma'})\xrightarrow{\delta}H^i(\mathrm{Sh}_{\mathbb{K}',\Sigma'},\widetilde{V}^{\mathrm{can}}_{\mathbb{K}',\Sigma'}),\]
where the first map comes from the morphism of sheaves $\widetilde{V}^{\mathrm{can}}_{\mathbb{K},\Sigma}\rightarrow t_{g*}\widetilde{V}^{\mathrm{can}}_{\mathbb{K}',\Sigma'}$ adjoint to $t_g^*\widetilde{V}^{\mathrm{can}}_{\mathbb{K},\Sigma}\cong\widetilde{V}^{\mathrm{can}}_{\mathbb{K}',\Sigma'}$, and the second map can be computed as follows: if
\[0\rightarrow\widetilde{V}^{\mathrm{can}}_{\mathbb{K}',\Sigma'}\rightarrow\mathcal{J}^*\]
is an acyclic resolution such that each $t_{g*}\mathcal{J}^i$ is also acyclic, then $t_{g*}\widetilde{V}^{\mathrm{can}}_{\mathbb{K}',\Sigma'}\rightarrow t_{g*}\mathcal{J}^*$ can be viewed as a morphism between complexes of sheaves and $\delta$ is the induced map between hypercohomology:
\[\begin{aligned}
H^i(\mathrm{Sh}_{\mathbb{K},\Sigma},t_{g*}\widetilde{V}^{\mathrm{can}}_{\mathbb{K}',\Sigma'})&=\mathbb{H}^i(t_{g*}\widetilde{V}^{\mathrm{can}}_{\mathbb{K}',\Sigma'})\\
&\rightarrow\mathbb{H}^i(t_{g*}\mathcal{J}^*)\\
&\cong H^i(t_{g*}\mathcal{J}^*(\mathrm{Sh}_{\mathbb{K},\Sigma}))\\
&=H^i(\mathcal{J}^*(\mathrm{Sh}_{\mathbb{K}',\Sigma'}))\cong H^i(\mathrm{Sh}_{\mathbb{K}',\Sigma'},\widetilde{V}^{\mathrm{can}}_{\mathbb{K}',\Sigma'}).
\end{aligned}\]
Furthermore, if $0\rightarrow\widetilde{V}^{\mathrm{can}}_{\mathbb{K},\Sigma}\rightarrow\mathcal{I}^*$ is an acyclic resolution of $\widetilde{V}^{\mathrm{can}}_{\mathbb{K},\Sigma}$ and $\mathcal{I}^*\rightarrow t_{g*}\mathcal{J}^*$ is a morphism of complexes extending $\widetilde{V}^{\mathrm{can}}_{\mathbb{K},\Sigma}\rightarrow t_{g*}\widetilde{V}^{\mathrm{can}}_{\mathbb{K}',\Sigma'}$, then (\ref{tg*}) fits into the following commutative diagram:
\[\begin{tikzcd}
H^i(\mathrm{Sh}_{\mathbb{K},\Sigma},\widetilde{V}^{\mathrm{can}}_{\mathbb{K},\Sigma})\rar{\sim}\dar{(\ref{tg*})}&\mathbb{H}^i(\mathcal{I}^*)\rar{\sim}\dar&H^i(\mathcal{I}^*(\mathrm{Sh}_{\mathbb{K},\Sigma}))\dar\\
H^i(\mathrm{Sh}_{\mathbb{K}',\Sigma'},\widetilde{V}^{\mathrm{can}}_{\mathbb{K}',\Sigma'})\rar{\sim}&\mathbb{H}^i(t_{g*}\mathcal{J}^*)\rar{\sim}&H^i(\mathcal{J}^*(\mathrm{Sh}_{\mathbb{K}',\Sigma'})),
\end{tikzcd}\]
If we can take $\mathcal{I}^*=\mathcal{I}_{V,\mathbb{K},\Sigma}^*,\mathcal{J}^*=\mathcal{I}_{V,\mathbb{K}',\Sigma'}^*$ and the morphism between them to be the restriction of the morphism
\[j_{\Sigma*}(\widetilde{V}_\mathbb{K}\otimes_{\mathcal{O}_{\mathrm{Sh}_\mathbb{K}}}\mathcal{A}^{0,*}_{\mathrm{Sh}_\mathbb{K}})\rightarrow t_{g*}j_{\Sigma'*}(\widetilde{V}_\mathbb{K'}\otimes_{\mathcal{O}_{\mathrm{Sh}_{\mathbb{K}'}}}\mathcal{A}^{0,*}_{\mathrm{Sh}_{\mathbb{K}'}})\]
induced by (\ref{Vo}), then the proposition follows. It remains to check the acyclicity of each $t_{g*}\mathcal{I}_{V,\mathbb{K}',\Sigma'}^i$, but the morphism $\mathcal{C}^\infty_{\mathrm{dmg}}=\mathcal{I}_{\mathbb{C},\mathbb{K},\Sigma}^0\rightarrow t_{g*}\mathcal{I}_{\mathbb{C},\mathbb{K}',\Sigma'}^0$ makes $t_{g*}\mathcal{I}_{\mathbb{C},\mathbb{K}',\Sigma'}^0$ a fine sheaf of rings, and hence every $t_{g*}\mathcal{I}_{V,\mathbb{K}',\Sigma'}^i$ is a fine sheaf.
\end{proof}

\section{A regularization theorem}\label{section4}

The goal of this section is to prove:

\begin{Prop}\label{umg}
The inclusion $C^\infty_{\mathrm{umg}}(G)^\mathbb{K}\subseteq C^\infty_{\mathrm{dmg}}(G)^\mathbb{K}$ induces isomorphisms of $(\mathfrak{p}_o,K_o)$-cohomology with coefficients in $V$, hence by Proposition \ref{Propdmg} we have
\[H^i(\mathrm{Sh}_{\mathbb{K},\Sigma},\widetilde{V}^{\mathrm{can}})\cong H^i_{(\mathfrak{p}_o,K_o)}(C^\infty_{\mathrm{umg}}(G)^\mathbb{K}\otimes V),\]
and by Proposition \ref{GAf} we have $G(\mathbb{A}_f)$-equivariant isomorphisms
\begin{equation}\label{umgGAf}
\varinjlim_{t_e^*}H^i(\widetilde{V}^{\mathrm{can}}_\mathbb{K})\cong H^i_{(\mathfrak{p}_o,K_o)}(C^\infty_{\mathrm{umg}}(G)\otimes V_o).
\end{equation}
\end{Prop}

This is an analogue of Borel's regularization theorem \cite[3.2]{Borel83} and we will imitate Franke's proof \cite[Proof of Theorem 3]{Franke98}. Let's first reduce the problem to the case where $V$ is irreducible: if
\[0\rightarrow V'\rightarrow V\rightarrow V''\rightarrow 0\]
is a short exact sequence of holomorphic representations of $P_h(\mathbb{C})$ and the proposition holds for $V'$ and $V''$, then the long exact sequences of cohomology imply that it also holds for $V$. Next we note that by the following lemma every irreducible $V$ is inflated from a representation of $K_h$, and it is irreducible as a representation of $K_o$ since $A_G(\mathbb{C})$ should act via a character.

\begin{Lem}
Let $V$ be an irreducible finite-dimensional representation of $\mathfrak{p}_h$, then $\mathfrak{p}_-$ acts on $V$ trivially.
\end{Lem}

\begin{proof}
Let $\mathfrak{b}$ be a Borel subalgebra of $\mathfrak{g}$ contained in $\mathfrak{p}_h$, then $\mathfrak{p}_-\subseteq[\mathfrak{b},\mathfrak{b}]$. By Lie's theorem there exists a nonzero vector $v\in V$ which is an eigenvector for every element of $\mathfrak{b}$ and thus annihilated by elements of $\mathfrak{p}_-$. As $V$ is irreducible, it is generated by $v$. Since $\mathfrak{p}_-$ is an ideal of $\mathfrak{p}_h$, it acts trivially on $V$.
\end{proof}

The proof of Proposition \ref{umg} for an irreducible $V$ will proceed as follows: in \ref{reg1} by proving Proposition \ref{RcTc} we reduce the problem to the existence of certain endomorphisms of $C^\infty_{\mathrm{dmg}}(G)^\mathbb{K}$, which will eventually be constructed in \ref{reg3}, while the functional analysis generalities developed in \ref{reg2} will be crucial in verifying that the endomorphisms we get are as desired.

\subsection{An application of the Okamoto-Ozeki formula}\label{reg1}

Fix a $\mathbb{K}$ and denote
\[C^\infty_\star:=C^\infty_\star(G)^\mathbb{K}\]
for $\star=umg,dmg$. Let $C_\mathfrak{g}$ be the Casimir element in $\mathfrak{Z}(\mathfrak{g})\subseteq\mathfrak{U}(\mathfrak{g})$: if $(X_i)$ is any basis of $[\mathfrak{g},\mathfrak{g}]$ and $(X'_i)$ is the dual basis with respect to the Killing form $B(\cdot,\cdot)$ on $\mathfrak{g}$, then $C_\mathfrak{g}=\sum_iX_iX'_i$.

\begin{Prop}\label{RcTc}
Suppose $M'\subseteq M$ are $(\mathfrak{g},K_o)$-modules such that for every $c\in\mathbb{C}$ there are $T_c,R_c\in\mathrm{End}_{(\mathfrak{p}_o,K_o)}(M)$ such that
\begin{equation}\label{incln1}
T_c(M')\subseteq M',R_c(M)\subseteq M'
\end{equation}
and
\begin{equation}\label{hmtpy}
R_c-\mathrm{Id}_M=\frac{1}{2}(C_\mathfrak{g}+c)T_c,
\end{equation}
then for every irreducible (and hence for every) finite-dimensional holomorphic representation $V$ of $P_h(\mathbb{C})$, the inclusion $M'\subseteq M$ induces isomorphisms
\[H^i_{(\mathfrak{p}_o,K_o)}(M'\otimes V)\cong H^i_{(\mathfrak{p}_o,K_o)}(M\otimes V).\]
Particularly if one can find maps $R_c,T_c$ as required for $M=C^\infty_{\mathrm{dmg}}$ and $M'=C^\infty_{\mathrm{umg}}$, then Proposition \ref{umg} holds.
\end{Prop}

\begin{Rmk}\label{Tcpi}
For each $\pi\in\widehat{K}_o$, let $S_\pi\subseteq\widehat{K}_o$ be the finite set of $K_o$-types appearing in $\pi\otimes\mathfrak{g}$, then as the map $M\times\mathfrak{g}\rightarrow M,(m,X)\mapsto Xm$ is $K_o$-equivariant we have
\begin{equation}\label{Spi}
\forall X\in\mathfrak{g},X(M_\pi)\subseteq M_{S_\pi}.
\end{equation}
The assumption in the proposition is then equivalent to that for every $c\in\mathbb{C}$ there exist maps $T_{c,\pi}\in\mathrm{End}_{K_o}(M_\pi)$ parametrized by $\pi\in\widehat{K}_o$ such that if we denote
\[R_{c,\pi}:=\mathrm{Id}_{M_\pi}+\frac{1}{2}(C_\mathfrak{g}+c)T_{c,\pi}\]
and
\[T_{c,S}:=\bigoplus_{\pi\in S}T_{c,\pi}\in\mathrm{End}_{K_o}(M_S)\]
for every $S\subseteq\widehat{K}_o$, then
\begin{equation}\label{incln2}
T_{c,\pi}(M'_\pi)\subseteq M'_\pi,R_{c,\pi}(M_\pi)\subseteq M'_\pi
\end{equation}
and
\begin{equation}\label{intertwine0}
\forall X\in\mathfrak{p}_o,XT_{c,\pi}=T_{c,S_\pi}X
\end{equation}
as maps from $M_\pi$ to $M_{S_\pi}$. In fact, $T_c:=T_{c,\widehat{K}_o}$ and $R_c:=\mathrm{Id}_M+\frac{1}{2}(C_\mathfrak{g}+c)T_c$ will satisfy (\ref{incln1}) and (\ref{hmtpy}).
\end{Rmk}

\begin{Lem}\label{chain}
Let $(A^*, d)$ be a complex of abelian groups, $D^*$ be a subcomplex of it. Suppose there are maps $r:A^*\rightarrow D^*$ and $\varrho:A^*\rightarrow A^{*-1}$ such that $r$ is a chain map, $\varrho(D^*)\subseteq D^{*-1}$ and
\begin{equation}\label{homotopy}
\mathrm{Id}_{A^*}-\iota\circ r=d\circ\varrho+\varrho\circ d,
\end{equation}
where $\iota:D^*\subseteq A^*$ is the inclusion, then $\iota$ induces isomorphisms of cohomology.
\end{Lem}

\begin{proof}
Denote by $\iota_*,r_*$ the maps on cohomology induced by $\iota,r$ respectively, then $\iota_*\circ r_*=\mathrm{Id}_{H^*(A)}$ since (\ref{homotopy}) implies that $\iota\circ r$ is chain homotopic to $\mathrm{Id}_{A^*}$. On the other hand (\ref{homotopy}) also implies that
\[\mathrm{Id}_{D^*}-r\circ\iota=d\circ\varrho+\varrho\circ d\]
as endomorphisms of $D^*$, so $r\circ\iota$ is chain homotopic to $\mathrm{Id}_{D^*}$ and $r_*\circ\iota_*=\mathrm{Id}_{H^*(D)}$. Therefore $\iota_*,r_*$ are inverse to each other and $\iota_*$ is an isomorphism.
\end{proof}

We wish to make Proposition \ref{RcTc} a special case of Lemma \ref{chain}, where $A^*,D^*$ are the relative Chevalley-Eilenberg complexes $C^*_{(\mathfrak{p}_o,K_o)}(M\otimes V)$ and $C^*_{(\mathfrak{p}_o,K_o)}(M'\otimes V)$. Let $r$ be the chain map induced by
\[R_c\otimes 1\in\mathrm{Hom}_{(\mathfrak{p}_o,K_o)}(M\otimes V,M'\otimes V)\]
for some $c$, then it remains to construct $\varrho$. Let $t:A^*\rightarrow A^*$ be the chain map induced by $T_c\otimes 1\in\mathrm{End}_{(\mathfrak{p}_o,K_o)}(M\otimes V)$, then by (\ref{incln1}) it preserves $D^*$. Note that for every $(\mathfrak{p}_o,K_o)$-module $M$,
\[C^i(M)\cong\mathrm{Hom}_{K_o}(\textstyle\bigwedge^i\mathfrak{p}_-,M),\]
and since $\mathfrak{p}_-$ is abelian the differential $d:C^i(M)\rightarrow C^{i+1}(M)$ is given by
\[d\varphi(Y_0\!\wedge\!...\!\wedge\!Y_i)=\sum_{j=0}^i(-1)^jY_j\varphi(Y_0\!\wedge\!...\,\widehat{Y_j}\,...\!\wedge\!Y_i).\]
When $M$ is a $(\mathfrak{g},K_o)$-module, let $X_1,...,X_n$ be an orthonormal basis of $\mathfrak{p}_-$ with respect to the sesquilinear Killing form $B(\cdot,\overline{\,\cdot\,})$ and consider the map
\[d':C^{i+1}(M\otimes V)\rightarrow\mathrm{Hom}(\textstyle\bigwedge^i\mathfrak{p}_-,M\otimes V)\]
defined by
\[d'\varphi(Y_1\!\wedge\!...\!\wedge\!Y_i)=-\sum_{l=1}^n(\overline{X}_l\otimes 1)\varphi(X_l\!\wedge\!Y_1\!\wedge\!...\!\wedge\!Y_i).\]

\begin{Prop}
The image of $d'$ is contained in $C^i(M\otimes V)$.
\end{Prop}

\begin{proof}
It is easy to see that $d'$ is independent of choices of the basis $X_1,...,X_n$. Then since for every $k\in K_o$, $kX_1:=\mathrm{Ad}(k)X_1,...,kX_n$ is another orthonormal basis of $\mathfrak{p}_-$, we have for every $\varphi\in C^{i+1}(M\otimes V)$,
\[\begin{aligned}
d'\varphi(kY_1\!\wedge\!...\!\wedge\!kY_i)&=-\sum_{l=1}^n(k\overline{X}_l\otimes 1)\varphi(kX_l\!\wedge\!kY_1\!\wedge\!...\!\wedge\!kY_i)\\
&=-\sum_{l=1}^n(k\overline{X}_l\otimes 1)k\varphi(X_l\!\wedge\!Y_1\!\wedge\!...\!\wedge\!Y_i)\\
&=-\sum_{l=1}^nk(\overline{X}_l\otimes 1)\varphi(X_l\!\wedge\!Y_1\!\wedge\!...\!\wedge\!Y_i),
\end{aligned}\]
i.e. $d'\varphi\in\mathrm{Hom}_{K_o}(\bigwedge^i\mathfrak{p}_-,M\otimes V)=C^i(M\otimes V)$.
\end{proof}

Thus $d'$ gives a map $A^*\rightarrow A^{*-1}$, and since $M'$ is a $(\mathfrak{g},K_o)$-submodule we have $d(B^*)\subseteq B^{*-1}$. Define $\varrho:=d'\circ t$, then $\varrho(B^*)\subseteq B^{*-1}$. Now we verify (\ref{homotopy}) for an appropriate $c$ and hence finish the proof of Proposition \ref{RcTc}.

\begin{Prop}\label{OO}
The action of $\Box:=dd'+d'd$ on $C^*(M\otimes V)$ is induced by
\[-\frac{1}{2}(C_\mathfrak{g}+c_V)\otimes 1\in\mathrm{End}_{(\mathfrak{p}_o,K_o)}(M\otimes V),\]
where $c_V$ is a scalar determined by $V$. Hence (\ref{homotopy}) follows from (\ref{hmtpy}) when $c=c_V$.
\end{Prop}

\begin{proof}
First note that since $t$ is a chain map,
\[\begin{aligned}
d\circ\varrho+\varrho\circ d&=d\circ d'\circ t+d'\circ t\circ d\\
&=d\circ d'\circ t+d'\circ d\circ t\\
&=\Box\circ t,\\
\end{aligned}\]
so if $\Box$ is as described then the action of $d\circ\varrho+\varrho\circ d$ is induced by
\[-\frac{1}{2}(C_\mathfrak{g}+c_V)T_c\otimes 1\in\mathrm{End}_{(\mathfrak{p}_o,K_o)}(M\otimes V),\]
which by (\ref{hmtpy}) equals $(1-R_c)\otimes 1$ when $c=c_V$, and thus (\ref{homotopy}) holds. Now we compute $\Box\,$:

\it Step 1. \rm For every $\varphi\in C^i(M\otimes V)$,
\[\begin{aligned}
dd'\varphi(Y_1\!\wedge\!...\!\wedge\!Y_i)&=\sum_{j=1}^i(-1)^{j-1}Y_jd'\varphi(Y_1\!\wedge\!...\,\widehat{Y_j}\,...\!\wedge\!Y_i)\\
&=\sum_{j=1}^i\sum_{l=1}^n(-1)^j Y_j(\overline{X}_l\otimes 1)\varphi(X_l\!\wedge\!Y_1\!\wedge\!...\,\widehat{Y_j}\,...\!\wedge\!Y_i)\\
&=-\sum_{j=1}^i\sum_{l=1}^n Y_j(\overline{X}_l\otimes 1)\varphi(Y_1\!\wedge\!...\,\widehat{Y_j}\!\wedge\!X_l\!\wedge\!...\!\wedge\!Y_i),
\end{aligned}\]
and
\[\begin{aligned}
d'd\varphi(Y_1\!\wedge\!...\!\wedge\!Y_i)=&-\sum_{l=1}^n(\overline{X}_l\otimes 1)d\varphi(X_l\!\wedge\!Y_1\!\wedge\!...\!\wedge\!Y_i)\\
=&-\sum_{l=1}^n(\overline{X}_l\otimes 1)X_l\varphi(Y_1\!\wedge\!...\!\wedge\!Y_i)\\
&+\sum_{l=1}^n\sum_{j=1}^i(-1)^{j-1} (\overline{X}_l\otimes 1)Y_j\varphi(X_l\!\wedge\!Y_1\!\wedge\!...\,\widehat{Y_j}\,...\!\wedge\!Y_i)\\
=&-(\sum_{l=1}^n\overline{X}_lX_l\otimes 1)\varphi(Y_1\!\wedge\!...\!\wedge\!Y_i)\\
&+\sum_{l=1}^n\sum_{j=1}^i(\overline{X}_l\otimes 1)Y_j\varphi(Y_1\!\wedge\!...\,\widehat{Y_j}\!\wedge\!X_l\!\wedge\!...\!\wedge\!Y_i),
\end{aligned}\]
where in the last step we use that $\mathfrak{p}_-$ acts on $V$ trivially. Hence
\[\begin{aligned}
\Box\varphi(Y_1\!\wedge\!...\!\wedge\!Y_i)=&-(\sum_{l=1}^n\overline{X}_lX_l\otimes 1)\varphi(Y_1\!\wedge\!...\!\wedge\!Y_i)\\
&+\sum_{j=1}^i\sum_{l=1}^n([\overline{X}_l,Y_j]\otimes 1)\varphi(Y_1\!\wedge\!...\,\widehat{Y_j}\!\wedge\!X_l\!\wedge\!...\!\wedge\!Y_i).
\end{aligned}\]

\it Step 2. \rm Let $X_{n+1},...,X_{n+m}$ be an orthonormal basis of $\mathfrak{k}_{o,\mathbb{R}}\cap[\mathfrak{g},\mathfrak{g}]$ with respect to $-B(\cdot,\cdot)$, then

\begin{Lem}\label{Bianchi}
In $\mathfrak{k}_o\otimes\mathfrak{p}_-$ it holds that for every $Y\in\mathfrak{p}_-$,
\[\sum_{l=1}^n[\overline{X}_l,Y]\otimes X_l=\sum_{p=n+1}^{n+m}X_p\otimes[X_p,Y].\]
\end{Lem}

\begin{proof}
For any $l\leq n<p$, the coefficients of $X_p\otimes X_l$ in the left and right hand sides are both $-B([\overline{X}_l,Y],X_p)=B(\overline{X}_l,[X_p,Y])$.
\end{proof}

By the lemma,
\[\begin{aligned}
&\sum_{j=1}^i\sum_{l=1}^n([\overline{X}_l,Y_j]\otimes 1)\varphi(Y_1\!\wedge\!...\,\widehat{Y_j}\!\wedge\!X_l\!\wedge\!...\!\wedge\!Y_i)\\
=&\sum_{j=1}^i\sum_{p=n+1}^{n+m}(X_p\otimes 1)\varphi(Y_1\!\wedge\!...\!\wedge\![X_p,Y_j]\!\wedge\!...\!\wedge\!Y_i)\\
=&\sum_{p=n+1}^{n+m}(X_p\otimes 1)X_p\varphi(Y_1\!\wedge\!...\!\wedge\!Y_i)\\
=&\sum_{p=n+1}^{n+m}(X_p^2\otimes 1+X_p\otimes X_p)\varphi(Y_1\!\wedge\!...\!\wedge\!Y_i),
\end{aligned}\]
where the second equality is due to the $\mathfrak{k}_o$-equivariance of $\varphi$. Therefore the action of $\Box$ on $C^i(M\otimes V)$ is induced by
\[-\sum_{l=1}^n\overline{X}_lX_l\otimes 1+\sum_{p=n+1}^{n+m}(X_p^2\otimes 1+X_p\otimes X_p)\in\mathrm{End}(M\otimes V).\]

\it Step 3. \rm Recall that $C_\mathfrak{g}=\sum_{l=1}^n(\overline{X}_lX_l+X_l\overline{X}_l)-\sum_{p=n+1}^{n+m}X_p^2$ and denote
\[Z:=\sum_{l=1}^n[X_l,\overline{X}_l]+\sum_{p=n+1}^{n+m}X_p^2\in\mathfrak{U}(\mathfrak{k}_o),\]
then
\begin{equation}\label{Box}
\begin{aligned}
&-\sum_{l=1}^n\overline{X}_lX_l\otimes 1+\sum_{p=n+1}^{n+m}(X_p^2\otimes 1+X_p\otimes X_p)\\
=&-\frac{1}{2}\sum_{l=1}^n(\overline{X}_lX_l+X_l\overline{X}_l-[X_l,\overline{X}_l])\otimes 1\\
&+\frac{1}{2}\sum_{p=n+1}^{n+m}(X_p^2\otimes 1+(X_p\otimes 1+1\otimes X_p)^2-1\otimes X_p^2)\\
=&-\frac{1}{2}(C_\mathfrak{g}\otimes 1-Z+1\otimes Z),
\end{aligned}
\end{equation}
where the term $Z$ means the diagonal action. It is easy to see that $Z$ is independent of choices of the bases $X_1,...,X_n$ and $X_{n+1},...,X_{n+m}$, then since the adjoint action of $K_o$ preserves the Killing form it also preserves $Z\in\mathfrak{U}(\mathfrak{k}_o)$. Thus by Schur's lemma $Z$ acts on the irreducible $K_o$-representation $V$ as a scalar $c_V$.

\it Step 4. \rm Now the right hand side of (\ref{Box}) differs from $-\frac{1}{2}(C_\mathfrak{g}+c_V)\otimes 1$ by $\frac{1}{2}Z$. As $C^i(M\otimes V)=\mathrm{Hom}_{K_o}(\bigwedge^i\mathfrak{p}_-,M\otimes V)$, it suffices to prove that

\begin{Lem}
$Z$ acts trivially on $\bigwedge^i\mathfrak{p}_-$ for every $i$.
\end{Lem}

\begin{proof}
For any $Y_1,...,Y_i\in\mathfrak{p}_-$,
\[\begin{aligned}
\sum_{p=n+1}^{n+m}X_p^2(Y_1\!\wedge\!...\!\wedge\!Y_i)=&\sum_{p=n+1}^{n+m}\sum_{j=1}^i Y_1\!\wedge\!...\!\wedge\![X_p,[X_p,Y_j]]\!\wedge\!...\!\wedge\!Y_i\\
+2&\sum_{p=n+1}^{n+m}\sum_{1\leq j<k\leq i}Y_1\!\wedge\!...\!\wedge\![X_p,Y_j]\!\wedge\!...\!\wedge\![X_p,Y_k]\!\wedge\!...\!\wedge\!Y_i.
\end{aligned}\]
By Lemma \ref{Bianchi}, the first term on the right hand side equals
\[\begin{aligned}
&\sum_{l=1}^n\sum_{j=1}^i Y_1\!\wedge\!...\!\wedge\![[\overline{X}_l,Y_j],X_l]\!\wedge\!...\!\wedge\!Y_i\\
=&-\sum_{l=1}^n\sum_{j=1}^i Y_1\!\wedge\!...\!\wedge\![[X_l,\overline{X}_l],Y_j]\!\wedge\!...\!\wedge\!Y_i\\
=&-\sum_{l=1}^n[X_l,\overline{X}_l](Y_1\!\wedge\!...\!\wedge\!Y_i),
\end{aligned}\]
where the first equality is due to the Jacobi identity and commutativity of $\mathfrak{p}_-$. Parallelly by Lemma \ref{Bianchi} and the Jacobi identity trick, for any $j<k$,
\[\begin{aligned}
&\sum_{p=n+1}^{n+m}Y_1\!\wedge\!...\!\wedge\![X_p,Y_j]\!\wedge\!...\!\wedge\![X_p,Y_k]\!\wedge\!...\!\wedge\!Y_i\\
&=\sum_{l=1}^n Y_1\!\wedge\!...\!\wedge\![[\overline{X}_l,Y_k],Y_j]\!\wedge\!...\!\wedge\!X_l\!\wedge\!...\!\wedge\!Y_i\\
&=\sum_{l=1}^n Y_1\!\wedge\!...\!\wedge\![[\overline{X}_l,Y_j],Y_k]\!\wedge\!...\!\wedge\!X_l\!\wedge\!...\!\wedge\!Y_i,
\end{aligned}\]
i.e. it is both alternating and symmetric in $Y_j$ and $Y_k$ and thus equals $0$. Therefore
\[\sum_{p=n+1}^{n+m}X_p^2(Y_1\!\wedge\!...\!\wedge\!Y_i)=-\sum_{l=1}^n[X_l,\overline{X}_l](Y_1\!\wedge\!...\!\wedge\!Y_i),\]
\[Z(Y_1\!\wedge\!...\!\wedge\!Y_i)=(\sum_{l=1}^n[X_l,\overline{X}_l]+\sum_{p=n+1}^{n+m}X_p^2)(Y_1\!\wedge\!...\!\wedge\!Y_i)=0.\qedhere\]
\end{proof}

\end{proof}

\begin{Rmk}
The computation of $\Box$ was done by Okamoto-Ozeki \cite[4.1]{OO67}. We include a derivation of the result to make the paper better self-contained.
\end{Rmk}

\begin{Rmk}
Proposition \ref{RcTc} can be proved similarly when $\mathfrak{p}_o$ and $V$ are replaced by $\mathfrak{m}_G$ and any finite-dimensional representation of $G(\mathbb{R})$ respectively, using Kuga's formula \cite[(6.9)]{MM63} instead of the Okamoto-Ozeki formula. This, together with the construction of $T_{c,\pi}$, makes Franke's alternative proof of Borel's regularization theorem.
\end{Rmk}

\subsection{Intertwining operators for semigroups of operators}\label{reg2}

The desired maps $T_{c,\pi}\in\mathrm{End}_{K_o}(C^\infty_{\mathrm{dmg},\pi})$ as in Remark \ref{Tcpi} will be constructed using semigroups of operators. We first review and prove some functional analysis preliminary.

\begin{Def}
(i) Let $X,Y$ be vector spaces. A \it linear operator \rm $A$ from $X$ to $Y$, denoted as $A:X\dashrightarrow Y$, is a linear map from a subspace $\mathrm{Dom}(A)\subseteq X$ to $Y$.\\
(ii) If the domain of $A':X\dashrightarrow Y$ contains $\mathrm{Dom}(A)$ and $A'|_{\mathrm{Dom}(A)}=A$, then $A'$ is called an \it extension \rm of $A$.\\
(iii) For $B:Y\dashrightarrow Z$, $BA$ is usually understood to be a linear operator from $X$ to $Z$ with domain
\[\mathrm{Dom}_B(A):=\big\{x\in\mathrm{Dom}(A):Ax\in\mathrm{Dom}(B)\big\}.\]
(iv) When $Y=X$, the \it resolvent set \rm of $A$ is the set of scalars $\lambda$ such that $\lambda-A$ maps $\mathrm{Dom}(A)$ bijectively to $X$.\\
(v) A linear operator $A:X\dashrightarrow Y$ between Banach spaces is said to be \it closed \rm if its graph in $X\times Y$ is closed.\\
(vi) When $X,Y$ are both Hilbert spaces and $\mathrm{Dom}(A)$ is dense in $X$, the \it adjoint \rm of $A$ is the linear operator $A^*:Y\dashrightarrow X$ defined on
\[\mathrm{Dom}(A^*):=\big\{y\in Y:\exists x\in X\textrm{ s.t. }\langle Au,y\rangle=\langle u,x\rangle,\forall u\in\mathrm{Dom}(A)\big\}\]
sending each $y$ to $x\in X$ uniquely determined by the above relation.
\end{Def}

\begin{Rmk}
For every closed subspace $X$ of a Hilbert space $Y$ we see that the inclusion $i_X:X\hookrightarrow Y$ and orthogonal projection $p_X:Y\twoheadrightarrow X$ are adjoints of each other.
\end{Rmk}

\begin{Lem}\label{adj1}
$(i)$ Let $A:X\dashrightarrow Y$ be a linear operator between Hilbert spaces with dense domain, then its adjoint $A^*:Y\dashrightarrow X$ is closed.\\
$(ii)$ For any scalars $c,d\in\mathbb{C}$ it holds that $(cA+d)^*=\overline{c}A^*+\overline{d}$.\\
$(iii)$ If $Z$ is another Hilbert space and $B:Y\dashrightarrow Z$ satisfies that $\mathrm{Dom}(B)\subseteq Y$ and $\mathrm{Dom}_B(A)\subseteq X$ are both dense, then $(BA)^*$ is an extension of $A^*B^*$.\\
$(iv)$ Suppose $\mathrm{Dom}_B(A)=\mathrm{Dom}(A)$, then
\[\mathrm{Dom}_{A^*}(B^*)=\mathrm{Dom}((BA)^*)\cap\mathrm{Dom}(B^*).\]
Particularly when $Z$ is a closed subspace of $Y$ and $B=p_Z$ we have
\[A^*i_Z=(p_ZA)^*.\]
$(v)$ When $X$ is a closed subspace of $Y$ and $A=i_X$, suppose $B=Bi_Xp_X$, i.e.
\begin{equation}\label{decomp}
\mathrm{Dom}(B)=(\mathrm{Dom}(B)\cap X)\oplus(\mathrm{Dom}(B)\cap X^\perp)
\end{equation}
and $B|_{\mathrm{Dom}(B)\cap X^\perp}=0$, then $B^*=i_X(Bi_X)^*$.
\end{Lem}

\begin{proof}
(i) Suppose $(y_n)$ is a sequence in $Y$ such that $y_n\rightarrow y$ and $A^*y_n\rightarrow x$, then for every $u\in\mathrm{Dom}(A)$,
\[\langle Au,y\rangle=\lim_n{\langle Au,y_n\rangle}=\lim_n{\langle u,A^*y_n\rangle}=\langle u,x\rangle,\]
thus $y\in\mathrm{Dom}(A^*)$ and $A^*y=x$.\\
(ii) For every $y\in Y$ the linear functionals $\langle(cA+d)\,\cdot\,,y\rangle$ and $c\,\langle A\,\cdot\,,y\rangle$ on $\mathrm{Dom}(A)$ differ by $\langle\cdot,\overline{d}y\rangle$ and are bounded at the same time, so $(cA+d)^*=\overline{c}A^*+\overline{d}$.\\
(iii) For any $z\in\mathrm{Dom}_{A^*}(B^*)$ and $u\in\mathrm{Dom}_B(A)$,
\[\langle BAu,z\rangle=\langle Au,B^*z\rangle=\langle u,A^*B^*z\rangle,\]
thus $z\in\mathrm{Dom}((BA)^*)$ and $(BA)^*z=A^*B^*z$.\\
(iv) We just need to show that $\mathrm{Dom}((BA)^*)\cap\mathrm{Dom}(B^*)\subseteq\mathrm{Dom}_{A^*}(B^*)$.
For any $z\in\mathrm{Dom}((BA)^*)\cap\mathrm{Dom}(B^*)$ and $u\in\mathrm{Dom}_B(A)$ we have
\[\langle Au,B^*z\rangle=\langle BAu,z\rangle=\langle u,(BA)^*z\rangle,\]
which implies that $B^*z\in\mathrm{Dom}(A^*),z\in\mathrm{Dom}_{A^*}(B^*)$ when $\mathrm{Dom}_B(A)=\mathrm{Dom}(A)$.\\
(v) By (iii) we have $B^*=(Bi_Xp_X)^*$ is an extension of $p_X^*(Bi_X)^*=i_X(Bi_X)^*$ and $(Bi_X)^*$ is an extension of $i_X^*B^*=p_XB^*$, so $B^*$ and $(Bi_X)^*$ have the same domain and $B^*=i_X(Bi_X)^*$.
\end{proof}

Combining (iv) and (v) we get:

\begin{Cor}\label{adj2}
Let $X,Z$ be closed subspaces of a Hilbert space $Y$, $B:Y\dashrightarrow Y$ be a linear operator with dense domain which satisfies (\ref{decomp}) and maps $\mathrm{Dom}(B)\cap X^\perp$ into $Z^\perp$, then
\[B^*i_Z=i_X(p_ZBi_X)^*,\]
i.e. $(p_ZBi_X)^*$ is the restriction of $B^*$ to $\mathrm{Dom}(B^*)\cap Z$.
\end{Cor}

\begin{Lem}\label{intertwine1}
Suppose $A:X\dashrightarrow Y$ is a closed linear operator between Banach spaces, $B_n:X\rightarrow X,C_n:Y\rightarrow Y(n\geq 0)$ are bounded linear maps such that
\begin{itemize}
\item $B_nx\rightarrow B_0x$ for every $x\in X$, $C_ny\rightarrow C_0y$ for every $y\in Y$ and
\item $AB_n$ is an extension of $C_nA$ for every $n\geq 1$,
\end{itemize}
then $AB_0$ is also an extension of $C_0A$.
\end{Lem}

\begin{proof}
First note that for every $n\geq 0$, $\mathrm{Dom}_{C_n}(A)=\mathrm{Dom}(A)$ as $C_n$ is bounded. Then by assumption we have $\mathrm{Dom}(A)\subseteq\mathrm{Dom}_A(B_n)$ for every $n\geq 1$ and for every $x\in\mathrm{Dom}(A)$,
\[B_nx\rightarrow B_0x\textrm{ and }AB_nx=C_nAx\rightarrow C_0Ax,\]
thus the closedness of $A$ implies that $B_0x\in\mathrm{Dom}(A)$ and $AB_0x=C_0Ax$.
\end{proof}

A \it $C_0$-semigroup of operators \rm over a Banach space $X$ is a family of bounded linear maps $S(t):X\rightarrow X\,(t\geq 0)$ such that
\begin{itemize}
\item $S(0)=\mathrm{Id}_X$,
\item $S(t+t')=S(t)S(t')$ for any $t,t'\geq 0$ and
\item $\lim_{t\rightarrow 0^+}S(t)x=x$ for every $x\in X$.
\end{itemize}
(Note that the second and third conditions imply that $S(\cdot)x:[0,+\infty)\rightarrow X$ is continuous for every $x\in X$.) Its \it (infinitesimal) generator \rm $G:X\dashrightarrow X$ is the linear operator defined on
\[\mathrm{Dom}(G):=\big\{x\in X:\lim_{h\rightarrow 0^+}\frac{S(h)x-x}{h}\textrm{ exists}\big\}\]
sending each $x$ to the above limit.

\begin{Lem}\label{TGphi1}
Suppose $\phi\in C_c([0,+\infty),\mathbb{R})$ is $C^1$ on $(0,+\infty)$ and $\phi':(0,+\infty)\rightarrow\mathbb{R}$ has a limit at $0$, then for every $x\in X$,
\[T_{G,\phi}x:=\int_0^{+\infty}\phi(t)S(t)x\,\mathrm{d}t\]
belongs to $\mathrm{Dom}(G)$ and
\[GT_{G,\phi}x=-\phi(0)x-T_{G,\phi'}x.\]
\end{Lem}

\begin{proof}
For every $h>0$,
\[\begin{aligned}
\frac{S(h)T_{G,\phi}x-T_{G,\phi}x}{h}&=\frac{1}{h}\int_0^{+\infty}\phi(t)[S(t+h)-S(t)]x\,\mathrm{d}t\\
&=-\frac{1}{h}\int_0^h\phi(t)S(t)x\,\mathrm{d}t-\int_h^{+\infty}\frac{\phi(t)-\phi(t-h)}{h}S(t)x\,\mathrm{d}t.
\end{aligned}\]
When $h\rightarrow 0^+$, by the continuity of $\phi(t)$ and $S(t)x$ at $t=0$ the first term above has limit $-\phi(0)x$; on the other hand by assumption $\phi'$ is uniformly continuous, therefore by the mean value theorem the second term has limit
\[-\int_0^{+\infty}\phi'(t)S(t)x\,\mathrm{d}t=-T_{G,\phi'}x.\qedhere\]
\end{proof}

\begin{Cor}\label{TGphi2}
$(i)$ In the lemma if $\phi$ is supported in $(0,+\infty)$ and smooth, then
\[\forall x\in X,T_{G,\phi}x\in\bigcap_{k\geq 1}\mathrm{Dom}(G^k).\]
$(ii)$ If $\phi:[0,+\infty)\rightarrow\mathbb{R}$ is compactly supported, identically $1$ in a neighborhood of $0$ and smooth, then for every $x\in X$, $T_{G,\phi}x\in\mathrm{Dom}(G)$ and
\[x+GT_{G,\phi}x\in\bigcap_{k\geq 1}\mathrm{Dom}(G^k).\]
In particular $T_{G,\phi}$ preserves each $\mathrm{Dom}(G^k)$ as well as $\bigcap_{k\geq 1}\mathrm{Dom}(G^k)$.
\end{Cor}

The Hille-Yosida theorem characterises infinitesimal generators of $C_0$ semigroups of operators. We need the following special case of it:

\begin{Lem}\label{HY}
Let $X$ be a Banach space. Suppose the domain of linear operator $G:X\dashrightarrow X$ is dense and there is a $\lambda_0\in\mathbb{R}$ such that every $\lambda>\lambda_0$ belongs to the resolvent set of $G$ and satisfies that
\begin{equation}\label{resolv1}
|(\lambda-G)^{-1}|\leq\frac{1}{\lambda-\lambda_0},
\end{equation}
then $G$ is the generator of a $C_0$-semigroup of operators $(e^{tG})_{t\geq 0}$ over $X$ which can be characterized as follows: for every $\lambda>\lambda_0$ denote
\[G_{\lambda,\lambda_0}:=(\lambda-\lambda_0)^2(\lambda-G)^{-1}-(\lambda-2\lambda_0):X\rightarrow X,\]
then for any $t>0$ and $x\in X$,
\begin{equation}\label{etG}
e^{tG}x=\lim_{\lambda\rightarrow+\infty}e^{tG_{\lambda,\lambda_0}}x,
\end{equation}
where the exponential $e^{tG_{\lambda,\lambda_0}}$ is defined by power series in the usual way. Moreover when $X$ is a Hilbert space (\ref{resolv1}) is equivalent to that for every $x\in\mathrm{Dom}(G)$,
\begin{equation}\label{resolv2}
\mathrm{Re}\,\langle Gx,x\rangle\leq\lambda_0\langle x,x\rangle.
\end{equation}
\end{Lem}

\begin{proof}
The claim that $G$ is the generator of a $C_0$-semigroup of operators is \cite[Theorem 34.7 (ii)]{Lax02}, while the characterisation of the semigroup is part of the proof of the theorem. The last assertion is Lumer and Phillips' reformulation of the Hille-Yosida condition \cite[Lemma 34.10]{Lax02}.
\end{proof}

\begin{Rmk}
It is easy to see \cite[Theorem 34.6]{Lax02} that the generator uniquely determines a $C_0$-semigroup of operators, so in particular $(e^{tG})_{t\geq 0}$ is independent of the choice of $\lambda_0$.
\end{Rmk}

\begin{Lem}\label{intertwine2}
Suppose $A:X\dashrightarrow Y,B:X\dashrightarrow X,C:Y\dashrightarrow Y$ are linear operators between Banach spaces such that
\begin{itemize}
\item $A$ is closed,
\item $B,C$ satisfy the assumptions in Lemma \ref{HY},
\item $\mathrm{Dom}(B)\subseteq\mathrm{Dom}(A)$ and
\item $CA$ is an extension of $AB$,
\end{itemize}
then $Ae^{tB}$ is an extension of $e^{tC}A$ for every $t>0$.
\end{Lem}

\begin{proof}
First note that in Lemma \ref{HY} if $G$ satisfies the assumption for one value of $\lambda_0$ then it also satisfies the assumption when $\lambda_0$ becomes larger. Therefore we can find a $\lambda_0\in\mathbb{R}$ such that
\[\forall x\in X,e^{tB}x=\lim_{\lambda\rightarrow+\infty}e^{tB_{\lambda,\lambda_0}}x\textrm{ and }\forall y\in Y,e^{tC}y=\lim_{\lambda\rightarrow+\infty}e^{tC_{\lambda,\lambda_0}}y.\]
Then by Lemma \ref{intertwine1} it suffices to show that $Ae^{tB_{\lambda,\lambda_0}}$ is an extension of $e^{tC_{\lambda,\lambda_0}}A$ for every $\lambda>\lambda_0$. Again since the exponential of a bounded linear map is defined by power series and can be approximated by polynomials of the map, it suffices to prove that $AB_{\lambda,\lambda_0}^n$ is an extension of $C_{\lambda,\lambda_0}^nA$ for every $n\geq 1$. By induction this follows from the $n=1$ case, which is equivalent to that $A(\lambda-B)^{-1}$ is an extension of $(\lambda-C)^{-1}A$, or that $(\lambda-C)A$ is an extension of $A(\lambda-B)$.
Now since $\mathrm{Dom}(B)\subseteq\mathrm{Dom}(A)$ we have
\[\mathrm{Dom}_A(\lambda-B)=\mathrm{Dom}_A(B)\]
for every $\lambda$, so it suffices to assume that $CA$ is an extension of $AB$.
\end{proof}

\begin{Cor}\label{intertwine3}
Let $X\hookrightarrow Y$ be a injective bounded linear map between Banach spaces, $B:X\dashrightarrow X,C:Y\dashrightarrow Y$ be linear operators satisfying the assumptions in Lemma \ref{HY} such that $C$ is an extension of $B$, then
\[e^{tB}=e^{tC}|_X\]
for every $t>0$.
\end{Cor}

\begin{Cor}\label{intertwine4}
Suppose $A:Y\dashrightarrow X,B:X\dashrightarrow X,C:Y\dashrightarrow Y$ are linear operators between Hilbert spaces with dense domains such that
\begin{itemize}
\item $B^*,C^*$ satisfy the assumptions in Lemma \ref{HY},
\item $\mathrm{Dom}(B^*)\subseteq\mathrm{Dom}(A^*)$,
\item $\mathrm{Dom}_A(C)=\mathrm{Dom}(C)$ and
\item $BA$ is an extension of $AC$,
\end{itemize}
then $(i)$ $A^*e^{tB^*}$ is an extension of $e^{tC^*}A^*$ for every $t>0$, and $(ii)$ $A^*T_{B^*,\phi}$ is an extension of $T_{C^*,\phi}A^*$ for every $\phi$ as in Lemma \ref{TGphi1}.
\end{Cor}

\begin{proof}
For $A^*,B^*,C^*$ we see that the second and third assumption in Lemma \ref{intertwine2} have already been assumed, and by Lemma \ref{adj1} (i) $A^*$ is closed, so it remains to verify that $C^*A^*$ is an extension of $A^*B^*$. By Lemma \ref{adj1} (iii) $C^*A^*$ and $A^*B^*$ have a common extension $(AC)^*$ (which extends $(BA)^*$), and by Lemma \ref{adj1} (iv)
\[\mathrm{Dom}_{C^*}(A^*)=\mathrm{Dom}((AC)^*)\cap\mathrm{Dom}(A^*).\]
On the other hand we have 
\[\mathrm{Dom}_{A^*}(B^*)\subseteq\mathrm{Dom}(B^*)\subseteq\mathrm{Dom}(A^*),\]
so $\mathrm{Dom}_{A^*}(B^*)\subseteq\mathrm{Dom}_{C^*}(A^*)$ and $C^*A^*$ is an extension of $A^*B^*$. Now since by definition $T_{B^*,\phi}$ and $T_{C^*,\phi}$ can be approximated by Riemann sums, (ii) follows from (i) via Lemma \ref{intertwine1}.
\end{proof}

\begin{Cor}\label{intertwine5}
Let $X$ be a unitary representation of a group $K$, $B:X\dashrightarrow X$ be a linear operator with dense domain such that $B^*$ satisfies the assumptions in Lemma \ref{HY} and $B=kBk^{-1}$ for every $k\in K$, then $e^{tB^*}:X\rightarrow X$ is $K$-equivariant for every $t>0$.
\end{Cor}

\subsection{Proof of Proposition \ref{umg}}\label{reg3}

Fix a weight function $\rho:[G]/K_o\mathbb{K}\rightarrow\mathbb{R}_{>0}$ as in Remark \ref{rho} (iii) and consider the weighted $L^2$-spaces
\[L^2_N:=\big\{f:[G]/\mathbb{K}\rightarrow\mathbb{C}:\int_{[G]}\rho^{-N}|f|^2\,dg<\infty\big\}/\!\sim\]
where $f$ are measurable functions and $\sim$ denotes equality almost everywhere, with inner product
\[\langle f_1,f_2\rangle_N:=\int_{[G]}\rho^{-N}f_1\overline{f_2}\,dg,\]
where $N\in\mathbb{Z}_{\geq 0}$ and $dg$ is induced from the Haar measure on $G(\mathbb{A})$. From Remark \ref{rho} (iii) we see that $L^2_N\subseteq L^2_{N'}$ with the inclusion being bounded when $N\leq N'$, and
\[C^\infty_{\mathrm{dmg}}\subseteq\bigcup_{N\geq 0}L^2_N.\]
On the other hand, it follows from the Euclidean case \cite[8.17 Proposition]{Folland99} that $C^\infty_c:=C^\infty_c([G]/\mathbb{K})^{K_o\mbox{-}\mathrm{fin}}$ is dense in every $L^2_N$. As $\rho$ is $K_o$-invariant, each $L^2_N$ is a unitary representation of $K_o$ under right translations and the above inclusions are all $K_o$-equivariant. With a slight abuse of notation we denote by $i_\star$ and $p_\star$ the inclusion of and orthogonal projection to the $\star$-isotypic component or part for $C^\infty_c$ and every $L^2_N$, where $\star$ could be any $\pi\in\widehat{K}_o$ or $S\subseteq\widehat{K}_o$, except that for projection from $L^2_N$ we require $S$ to be finite.

Now we define for any $D\in\mathfrak{U}(\mathfrak{g})$ and $N\geq 0$ a linear operator $D_N:L^2_N\dashrightarrow L^2_N$ as follows: let $D^\dag\in\mathfrak{U}(\mathfrak{g})$ be the \it formal adjoint \rm of $D$, i.e. $D^\dag$ depends antilinearly on $D$ and if $D=X_1...X_k$ with $X_1,...,X_k\in\mathfrak{g}$ then $D^\dag=(-1)^k\overline{X}_k...\overline{X}_1$; denote by $D_c$ the action of $D$ on $C^\infty_c$, then
\[D_N:=(\rho^ND^\dag_c\rho^{-N})^*.\]

\begin{Prop}\label{DN}
$(i)$ $D_{N'}$ is an extension of $D_N$ whenever $N'\geq N$.\\
$(ii)$ $f\in L^2_N$ belongs to $\mathrm{Dom}(D_N)$ if and only if its distributional derivative $Df\in L^2_N$, in which case $D_Nf=Df$. Note that for smooth functions, distributional derivatives are just the usual derivatives.\\
$(iii)$ If finite subsets $S_1,S_2\subseteq\widehat{K}_o$ satisfy $D^\dag(C^\infty_{c,S_2^c})\subseteq C^\infty_{c,S_1^c}$, then for every $N\geq 0$,
\[D_{N,S_1,S_2}:=(p_{S_1}\rho^ND^\dag_c\rho^{-N}i_{S_2})^*:L^2_{N,S_1}\dashrightarrow L^2_{N,S_2}\]
is the restriction of $D_N$ to $\mathrm{Dom}(D_N)\cap L^2_{N,S_1}$.
\end{Prop}

\begin{proof}
As for (i), if $f\in\mathrm{Dom}(D_N)$ then by definition we have
\[\forall h\in C^\infty_c,\langle\rho^ND^\dag(\rho^{-N}h),f\rangle_N=\langle h,D_Nf\rangle_N,\]
hence for every $h\in C^\infty_c$,
\[\begin{aligned}
\langle\rho^{N'}D^\dag(\rho^{-N'}h),f\rangle_{N'}&=\langle\rho^ND^\dag(\rho^{-N'}h),f\rangle_N\\
&=\langle\rho^{N-N'}h,D_Nf\rangle_N\\
&=\langle h,D_Nf\rangle_{N'},
\end{aligned}\]
i.e. $f\in\mathrm{Dom}(D_{N'})$ and $D_{N'}f=D_Nf$. (ii) is essentially by definition: $Df=\widetilde{f}\in L^2_N$ if and only if for every $h\in C^\infty_c$,
\[\int_{[G]}D^\dag(\rho^{-N}h)\overline{f}\,dg=\langle h,\widetilde{f}\rangle_N,\]
while the left hand side is just $\langle\rho^ND^\dag(\rho^{-N}h),f\rangle_N$. Finally (iii) is a special case of Corollary \ref{adj2}.
\end{proof}

\begin{Cor}\label{DNS}
$(i)$ Combining (i) and (iii) we see that $D_{N',S_1,S_2}$ is an extension of $D_{N,S_1,S_2}$ whenever $N'\geq N$.\\
$(ii)$ Combining (ii) and (iii) it follows that $f\in C^\infty_{\mathrm{dmg},S_1}$ lies in $\mathrm{Dom}(D_{N,S_1,S_2})$ if and only if $f$ and $Df$ both belong to $L^2_N$. In particular we have
\[C^\infty_{\mathrm{dmg},S_1}\subseteq\bigcup_{N\geq 0}\mathrm{Dom}(D_{N,S_1,S_2}).\]
Moreover, $D_{N,S_1,S_2}$ coincides with $D$ on $C^\infty_{\mathrm{dmg},S_1}\cap\mathrm{Dom}(D_{N,S_1,S_2})$.
\end{Cor}

When $D=C_\mathfrak{g}$, $D^\dag_c=C_{\mathfrak{g},c}$ is $K_o$-equivariant (because the Killing form is $K_o$-invariant), hence for each finite subset $S_1\subseteq\widehat{K}_o$ we can take $S_2=S_1$; when $D=X\in\mathfrak{g}$, $D^\dag=-\overline{X}$ and dual to (\ref{Spi}) we have
\[\forall\pi\in\widehat{K}_o,\overline{X}(C^\infty_{c,S_\pi^c})\subseteq C^\infty_{c,\{\pi\}^c},\]
thus for each singleton $S_1=\{\pi\}$ we can take $S_2=S_\pi$. For any finite subset $S\subseteq\widehat{K}_o$ and element $\pi\in\widehat{K}_o$ we denote
\[C_{\mathfrak{g},N,S}:=C_{\mathfrak{g},N,S,S}:L^2_{N,S}\dashrightarrow L^2_{N,S},\]
\[C_{\mathfrak{g},N,\pi}:=C_{\mathfrak{g},N,\{\pi\}}:L^2_{N,\pi}\dashrightarrow L^2_{N,\pi},\]
\[X_{N,\pi}:=X_{N,\{\pi\},S_\pi}:L^2_{N,\pi}\dashrightarrow L^2_{N,S_\pi},\]
then Proposition \ref{DN} (iii) implies that $C_{\mathfrak{g},N,S}=\bigoplus_{\pi\in S}C_{\mathfrak{g},N,\pi}$.

\begin{Prop}\label{Sobolev1}
For any $N\geq 0$ and $\pi\in\widehat{K}_o$, $C_{\mathfrak{g},N,\pi}$ satisfies the assumptions in Lemma \ref{HY} and we have
\begin{equation}\label{incln3}
\mathrm{Dom}(C_{\mathfrak{g},N,\pi})\subseteq\mathrm{Dom}(X_{N,\pi})
\end{equation}
for every $X\in\mathfrak{g}$.
\end{Prop}

\begin{proof}
Let $X_1,...,X_n$ and $X_{n+1},...,X_{n+m}$ be bases of $\mathfrak{p}_-$ and $\mathfrak{k}_{o,\mathbb{R}}\cap[\mathfrak{g},\mathfrak{g}]$ as in the proof of Proposition \ref{OO}, $W^1_{N,\pi}\subseteq L^2_{N,\pi}$ be the subspace consisting of $f\in L^2_{N,\pi}$ whose distributional derivative $Xf$ belongs to $L^2_N$ for every $X\in\mathfrak{g}$, then by Proposition \ref{DN} (ii) $W^1_{N,\pi}\subseteq\mathrm{Dom}(X_N)$ for every $X\in\mathfrak{g}$. Equip $W^1_{N,\pi}$ with the inner product
\[\langle f_1,f_2\rangle_{W^1_{N,\pi}}:=\langle f_1,f_2\rangle_N+\sum_{l=1}^n(\langle X_lf_1,X_lf_2\rangle_N+\langle\overline{X}_lf_1,\overline{X}_lf_2\rangle_N),\]
then it becomes a Hilbert space, $C^\infty_{c,\pi}$ is a dense subspace of it and the inclusion $W^1_{N,\pi}\subseteq L^2_{N,\pi}$ is bounded. For every $\lambda\in\mathbb{R}$ consider the sesquilinear form
\[\begin{aligned}
I_{N,\lambda}(f_1,f_2):=&\sum_{l=1}^n(\langle\rho^NX_l(\rho^{-N}f_1),X_lf_2\rangle_N+\langle\rho^N\overline{X}_l(\rho^{-N}f_1),\overline{X}_lf_2\rangle_N)\\
&-\sum_{p=n+1}^{n+m}\langle\rho^NX_p(\rho^{-N}f_1),X_pf_2\rangle_N+\lambda\langle f_1,f_2\rangle_N\\
=&-N\sum_{l=1}^n(\langle X_l(\log{\rho})f_1,X_lf_2\rangle_N+\langle\overline{X}_l(\log{\rho})f_1,\overline{X}_lf_2\rangle_N)\\
&-\sum_{p=n+1}^{n+m}\langle X_pf_1,X_pf_2\rangle_N+(\lambda-1)\langle f_1,f_2\rangle_N+\langle f_1,f_2\rangle_{W^1_{N,\pi}}
\end{aligned}\]
on $W^1_{N,\pi}$, which by (\ref{dlog}) is well-defined and bounded. On the one hand for any $f\in W^1_{N,\pi}$ and $h\in C^\infty_{c,\pi}$ we have
\begin{equation}\label{byparts}
\begin{aligned}
\forall X\in\mathfrak{g},\langle\rho^NX(\rho^{-N}h),Xf\rangle_N&=\int_{[G]}X(\rho^{-N}h)\overline{Xf}\,dg\\
&=-\int_{[G]}\overline{X}X(\rho^{-N}h)\overline{f}\,dg\\
&=-\langle\rho^N\overline{X}X(\rho^{-N}h),f\rangle_N
\end{aligned}
\end{equation}
and hence $I_{N,\lambda}(h,f)=\lambda\langle h,f\rangle_N-\langle\rho^NC_{\mathfrak{g},c}(\rho^{-N}h),f\rangle_N$. Then by definition $\widetilde{f}$ in $L^2_{N,\pi}$ equals $(\lambda-C_{\mathfrak{g},N,\pi})f$ if and only if
\begin{equation}\label{LM}
I_{N,\lambda}(h,f)=\langle h,\widetilde{f}\rangle_N
\end{equation}
holds for all $h\in C^\infty_{c,\pi}$, and by continuity this is equivalent to that (\ref{LM}) holds for all $h\in W^1_{N,\pi}$. On the other hand we see that for every $f\in W^1_{N,\pi}$,
\[\begin{aligned}
I_{N,\lambda}(f,f)=&-N\sum_{l=1}^n(\langle X_l(\log{\rho})f,X_lf\rangle_N+\langle\overline{X}_l(\log{\rho})f,\overline{X}_lf\rangle_N)\\
&-\sum_{p=n+1}^{n+m}\|X_pf\|_N^2+(\lambda-1)\|f\|_N^2+\|f\|_{W^1_{N,\pi}}^2.
\end{aligned}\]
Let $M_1=M_1(\rho)>0$ be a common upper bound of all $|X_l(\log{\rho})|$ and $|\overline{X}_l(\log{\rho})|$, $M_2=M_2(\pi)>0$ be such that $\sum_{p=n+1}^{n+m}\|X_pv\|_\pi^2\leq M_2\|v\|_\pi^2$ for every $v\in\pi$ with $\pi$ thought of as an irreducible unitary representation of $K_o$, then we have
\[\begin{aligned}
\mathrm{Re}\,I_{N,\lambda}(f,f)\geq&-NM_1\sum_{l=1}^n(\|f\|_N\|X_lf\|_N+\|f\|_N\|\overline{X}_lf\|_N)\\
&+(\lambda-M_2-1)\|f\|_N^2+\|f\|_{W^1_{N,\pi}}^2\\
\geq&-\frac{1}{2}\sum_{l=1}^n(\|X_lf\|_N^2+\|\overline{X}_lf\|_N^2)\\
&+(\lambda-nN^2M_1^2-M_2-1)\|f\|_N^2+\|f\|_{W^1_{N,\pi}}^2\\
=&(\lambda-nN^2M_1^2-M_2-\frac{1}{2})\|f\|_N^2+\frac{1}{2}\|f\|_{W^1_{N,\pi}}^2.
\end{aligned}\]
Specially when $\lambda\geq\lambda_0:=nN^2M_1^2+M_2+\frac{1}{2}$ it holds that
\begin{equation}\label{resolv3}
\mathrm{Re}\,I_{N,\lambda}(f,f)\geq\frac{1}{2}\|f\|_{W^1_{N,\pi}}^2
\end{equation}
and thus $I_{N,\lambda}$ satisfies the conditions of the Lax-Milgram theorem \cite[Theorem 6.6]{Lax02}, which implies that every bounded linear functional on $W^1_{N,\pi}$ is of the form $I_{N,\lambda}(\cdot,f)$ for some $f\in W^1_{N,\pi}$. In particular for every $\widetilde{f}\in L^2_{N,\pi}$ there exists $f\in W^1_{N,\pi}$ such that (\ref{LM}) holds for all $h\in W^1_{N,\pi}$, i.e. $\lambda-C_{\mathfrak{g},N,\pi}$ maps $W^1_{N,\pi}\cap\mathrm{Dom}(C_{\mathfrak{g},N,\pi})$ onto $L^2_{N,\pi}$. If we have $\lambda-C_{\mathfrak{g},N,\pi}$ is injective, then it follows that
\begin{equation}\label{incln4}
\mathrm{Dom}(C_{\mathfrak{g},N,\pi})\subseteq W^1_{N,\pi}
\end{equation}
and $\lambda$ belongs to the resolvent set of $C_{\mathfrak{g},N,\pi}$. Notice that analogous to (\ref{byparts}) we also have for any $f\in W^1_{N,\pi}$ and $h\in C^\infty_{c,\pi}$,
\[I_{N,\lambda}(f,h)=\lambda\langle f,h\rangle_N-\langle f,C_{\mathfrak{g},c}h\rangle_N,\]
so the Lax-Milgram theorem also implies that $\lambda-\rho^NC_{\mathfrak{g},N,\pi}\rho^{-N}$ maps onto $L^2_{N,\pi}$. As indicated on \cite[p.193]{Franke98}, one can use Friedrichs mollifiers to approximate every $f\in\mathrm{Dom}(C_{\mathfrak{g},N,\pi})$ by $f_n\in C^\infty_{c,\pi}$ such that $C_{\mathfrak{g},c}f_n\rightarrow C_{\mathfrak{g},N,\pi}f$ in $L^2_{N,\pi}$, and the approximation could be taken stronger such that
\[\rho^Nf_n\rightarrow\rho^Nf\text{ and }\rho^NC_{\mathfrak{g},c}f_n\rightarrow\rho^NC_{\mathfrak{g},N,\pi}f.\]
Hence $(\lambda-\rho^NC_{\mathfrak{g},c}\rho^{-N})(C^\infty_{c,\pi})$ is dense in $L^2_{N,\pi}$, and $\lambda-C_{\mathfrak{g},N,\pi}$ is injective. We see that (\ref{incln4}) implies (\ref{incln3}) as well as that
\[\forall f\in\mathrm{Dom}(C_{\mathfrak{g},N,\pi}),\langle f,C_{\mathfrak{g},N,\pi}f\rangle_N=\lambda\langle f,f\rangle_N-I_{N,\lambda}(f,f),\]
which indicates that the inequality (\ref{resolv2}) follows from (\ref{resolv3}).
\end{proof}

\begin{Cor}\label{Sobolev1.5} 
$(i)$ $C_{\mathfrak{g},N,S_\pi}X_{N,\pi}$ is an extension of $X_{N,\pi}C_{\mathfrak{g},N,\pi}$.\\
$(ii)$ For any $k\geq 1$ and
$X_1,...,X_k\in\mathfrak{g}$,
\[\mathrm{Dom}(C_{\mathfrak{g},N,\pi}^k)\subseteq\mathrm{Dom}(
(X_k)_N(X_{k-1})_N...(X_1)_N),\]
where $(X_l)_N=(X_l)_{N,S_{\pi,l-1},S_{\pi,l}}$ and $S_{\pi,l}$ is the set of $K_o$-types appearing in $\pi\otimes\mathfrak{g}^{\otimes l}$.
\end{Cor}

\begin{proof}
(i) $C_{\mathfrak{g},N,S_\pi}X_{N,\pi}$ and $X_{N,\pi}C_{\mathfrak{g},N,\pi}$ have a common extension $(C_{\mathfrak{g}}X)_{N,\{\pi\},S_\pi}$, and by Lemma \ref{adj1} (iii) the assertion follows from that $\mathrm{Dom}(C_{\mathfrak{g},N,\pi})\subseteq\mathrm{Dom}(X_{N,\pi})$.
(ii) By (i) we have
\[\forall f\in\mathrm{Dom}(C_{\mathfrak{g},N,\pi}^k),X\in\mathfrak{g},X_{N,\pi}f\in\mathrm{Dom}(C_{\mathfrak{g},N,S_\pi}^{k-1}).\]
The conclusion then follows from induction on $k$. 
\end{proof}

\begin{Prop}\label{Sobolev2}
For every $\pi\in\widehat{K}_o$ we have
\[C^\infty_{\mathrm{dmg},\pi}=\bigcap_{k\geq 1}\bigcup_{N\geq 0}\mathrm{Dom}(C_{\mathfrak{g},N,\pi}^k)\]
and
\[C^\infty_{\mathrm{umg},\pi}=\bigcup_{N\geq 0}\bigcap_{k\geq 1}\mathrm{Dom}(C_{\mathfrak{g},N,\pi}^k).\]
\end{Prop}

\begin{proof}
By Corollary \ref{Sobolev1.5} (ii), Proposition \ref{DN} (ii) and the Sobolev embedding theorem \cite[Theorem 5.6.6 (ii)]{Evans98}, every function in $\mathrm{Dom}(C_{\mathfrak{g},N,\pi}^k)$ is $C^r$ whenever $k>r+\frac{1}{2}\dim{\mathfrak{g}/\mathfrak{a}_G}$, hence the 
functions in
\[\bigcap_{k\geq 1}\bigcup_{N\geq 0}\mathrm{Dom}(C_{\mathfrak{g},N,\pi}^k)\]
are all smooth. Besides, the characterisation of $\mathrm{Dom}(C_{\mathfrak{g},N,\pi}^k)$ implies that the spaces
\[\bigoplus_{\pi\in\widehat{K}_o}\bigcap_{k\geq 1}\mathrm{Dom}(C_{\mathfrak{g},N,\pi}^k)\]
for $N\geq 0$ as well as $\bigoplus_{\pi\in\widehat{K}_o}\bigcap_{k\geq 1}\bigcup_{N\geq 0}\mathrm{Dom}(C_{\mathfrak{g},N,\pi}^k)$ are all stable under the action of $\mathfrak{U}(\mathfrak{g})$. On the other hand \cite[Proposition 2]{Franke98} shows that there is an $N_0\geq 0$ depending on $\rho$ such that when $k\geq\dim_\mathbb{R}{\mathrm{Sh}_\mathbb{K}}$, every $f\in\mathrm{Dom}(C_{\mathfrak{g},N,\pi}^k)$ is bounded by a multiple of $\rho^{N/2+N_0}$. Therefore we have
\[\bigcap_{k\geq 1}\bigcup_{N\geq 0}\mathrm{Dom}(C_{\mathfrak{g},N,\pi}^k)\subseteq C^\infty_{\mathrm{dmg},\pi}\]
and
\[\bigcup_{N\geq 0}\bigcap_{k\geq 1}\mathrm{Dom}(C_{\mathfrak{g},N,\pi}^k)\subseteq C^\infty_{\mathrm{umg},\pi}.\]
Inclusions in the other direction clearly follow from Corollary \ref{DNS} (ii).
\end{proof}

Now we construct maps $T_{c,\pi}\in\mathrm{End}_{K_o}(C^\infty_{\mathrm{dmg},\pi})$ for any $c\in\mathbb{C}$ and $\pi\in\widehat{K}_o$. By Lemma \ref{HY} and Proposition \ref{Sobolev1}, for each $N\geq 0$,
\[G_{N,\pi}:=\frac{1}{2}(C_{\mathfrak{g},N,\pi}+c)\]
is the generator of a $C_0$-semigroup of operators $(e^{tG_{N,\pi}})_{t\geq 0}$ over $L^2_{N,\pi}$. Pick a $\phi$ as in Corollary \ref{TGphi2} (ii) and for every $x\in C^\infty_{\mathrm{dmg},\pi}\cap L^2_{N,\pi}$ define
\[T_{c,\pi}x:=T_{G_{N,\pi},\phi}x=\int_0^{+\infty}\phi(t)e^{tG_{N,\pi}}x\,\mathrm{d}t\in L^2_{N,\pi},\]
then it follows from Corollary \ref{DNS} (i) and Corollary \ref{intertwine3} that $T_{c,\pi}$ is well-defined as a map from $C^\infty_{\mathrm{dmg},\pi}$ to $\bigcup_{N\geq 0}L^2_{N,\pi}$, and Corollary \ref{intertwine5} implies that $T_{c,\pi}$ is $K_o$-equivariant.

\begin{Prop}
The image of each $T_{c,\pi}$ is within $C^\infty_{\mathrm{dmg},\pi}$, and $(T_{c,\pi})_{\pi\in\widehat{K}_o}$ satisfies (\ref{incln2}) and (\ref{intertwine0}) for $M=C^\infty_{\mathrm{dmg}}$ and $M'=C^\infty_{\mathrm{umg}}$. Consequently Proposition \ref{umg} follows from Proposition \ref{RcTc} and Remark \ref{Tcpi}.
\end{Prop}

\begin{proof}
By Corollary \ref{TGphi2} (ii) $T_{c,\pi}$ preserves $\mathrm{Dom}(C_{\mathfrak{g},N,\pi}^k)$ for any $N\geq 0$ and $k\geq 1$, hence by Proposition \ref{Sobolev2} it also preserves
$C^\infty_{\mathrm{dmg},\pi}$ and $C^\infty_{\mathrm{umg},\pi}$.
Corollary \ref{DNS} (ii) then implies that for every $x\in C^\infty_{\mathrm{dmg},\pi}\cap\mathrm{Dom}(C_{\mathfrak{g},N,\pi})$,
\[x+\frac{1}{2}(C_\mathfrak{g}+c)T_{c,\pi}x=x+G_{N,\pi}T_{G_{N,\pi},\phi}x,\]
which by Corollary \ref{TGphi2} (ii) lies in $\bigcap_{k\geq 1}\mathrm{Dom}(C_{\mathfrak{g},N,\pi}^k)\subseteq C^\infty_{\mathrm{umg},\pi}$, so the second inclusion in (\ref{incln2}) also holds.

Again by Corollary \ref{DNS} (ii) we have $C^\infty_{\mathrm{umg},\pi}\subseteq\bigcup_{N\geq 0}\mathrm{Dom}(X_{N,\pi})$ for any $\pi\in\widehat{K}_o$ and $X\in\mathfrak{p}_o$, thus to prove (\ref{intertwine0}) it suffices to verify for every $N\geq 0$ that $X_{N,\pi}T_{G_{N,\pi},\phi}$ is an extension of $(\bigoplus_{\pi'\in S_\pi}T_{G_{N,\pi'},\phi})X_{N,\pi}$. Proposition \ref{Sobolev1} implies that
\[G_{N,S_\pi}:=\frac{1}{2}(C_{\mathfrak{g},N,S_\pi}+c)\]
also satisfies the assumptions in Lemma \ref{HY}, and the $C_0$-semigroup of operators it generates is clearly $(\bigoplus_{\pi\in S_\pi}e^{tG_{N,\pi}}:L^2_{N,S_\pi}\rightarrow L^2_{N,S_\pi})_{t\geq 0}$, so we have
\[(\bigoplus_{\pi'\in S_\pi}T_{G_{N,\pi'},\phi})=T_{G_{N,S_\pi},\phi}.\]
By Lemma \ref{adj1} (ii), $G_{N,\pi}$ and $G_{N,S_\pi}$ are the adjoints of
\[B:=\frac{1}{2}(p_\pi\rho^NC_{\mathfrak{g},c}\rho^{-N}i_\pi+\overline{c})\]
and
\[C:=\frac{1}{2}(p_{S_\pi}\rho^NC_{\mathfrak{g},c}\rho^{-N}i_{S_\pi}+\overline{c})\]
respectively. Now we want to apply Corollary \ref{intertwine4} for $A:=-p_\pi\rho^N\overline{X}_c\rho^{-N}i_{S_\pi}$, $B$ and $C$. We see that the first two assumptions there are guaranteed by Proposition \ref{Sobolev1}, and clearly $\mathrm{Dom}_A(C)=\mathrm{Dom}(C)=C^\infty_{c,\pi}$. Since $C_{\mathfrak{g},c}$ and $\rho^{\pm N}$ commute with any $i_\star$ and $p_\star$, it follows from $C_{\mathfrak{g},c}\overline{X}_c=\overline{X}_cC_{\mathfrak{g},c}$ that $BA=AC$. Therefore Corollary \ref{intertwine4} (ii) implies that $X_{N,\pi}T_{G_{N,\pi},\phi}$ is an extension of $T_{G_{N,S_\pi},\phi}X_{N,\pi}$.
\end{proof}

\section{Coherent cohomology and automorphic forms}\label{section5}

Let $\mathcal{A}(G)$ be the space of automorphic forms on $G$. In this section we show that the inclusion $\mathcal{A}(G)^{A_G(\mathbb{R})^\circ}\subseteq C^\infty_{\mathrm{umg}}(G)$ also induces isomorphisms of $(\mathfrak{p}_o,K_o)$-cohomology with coefficients in $V$, hence in (\ref{umgGAf}) we can replace the latter by the former. This is parallel to \cite[Theorem 18]{Franke98} and also relies on the $\mathfrak{Fin}$-acyclicity established in that paper.

For every pair $(\mathfrak{g},K)$ of a complex Lie algebra and a compact Lie group satisfying the conditions in \cite[1.64]{KV95}, denote by $\mathcal{C}(\mathfrak{g},K)$ the category of $(\mathfrak{g},K)$-modules. When $\mathfrak{g}$ is reductive, in \cite{Franke98} Franke considered for every ideal $\mathcal{J}\subseteq\mathfrak{Z}(\mathfrak{g})$ of finite codimension the functor $\mathfrak{Fin}_\mathcal{J}:\mathcal{C}(\mathfrak{g},K)\rightarrow\mathcal{C}(\mathfrak{g},K)$,
\[M\mapsto\mathfrak{Fin}_\mathcal{J}M:=\big\{m\in M:\mathcal{J}^nm=0\textrm{ for some }n\big\}.\]
$\mathfrak{Fin}_\mathcal{J}$ is evidently left-exact and its right derived functors are denoted as $\mathfrak{Fin}_\mathcal{J}^i:=R^i\mathfrak{Fin}_\mathcal{J}$. Franke showed that

\begin{Lem}\label{Franke}
$(i)$ $\mathfrak{Fin}_\mathcal{J}$ preserves injectives.\\
$(ii)$ For any reductive group $G$ over $\mathbb{Q}$ and ideal $\mathcal{J}\subseteq\mathfrak{Z}(\mathfrak{m}_G)$ of finite codimension, $C^\infty_{\mathrm{umg}}(G)$ and $\mathcal{A}(G)^{A_G(\mathbb{R})^\circ}$ are $\mathfrak{Fin}_\mathcal{J}$-acyclic.
\end{Lem}

\begin{proof}
(i) is \cite[Theorem 7(1)]{Franke98}, and the $\mathfrak{Fin}_\mathcal{J}$-acyclicity of $C^\infty_{\mathrm{umg}}(G)^{A_G(\mathbb{R})^\circ}$ follows from \cite[Theorem 16]{Franke98}. The $\mathfrak{Fin}_\mathcal{J}$-acyclicity of $\mathcal{A}(G)^{A_G(\mathbb{R})^\circ}$ is implied by the following lemma.
\end{proof}

\begin{Lem}
If a $(\mathfrak{g},K)$-module $M$ satisfies that $\mathfrak{Z}(\mathfrak{g})m$ is finite-dimensional for every $m\in M$, then $M$ is $\mathfrak{Fin}_\mathcal{J}$-acyclic for every ideal $\mathcal{J}\subseteq\mathfrak{Z}(\mathfrak{g})$ of finite codimension.
\end{Lem}

\begin{proof}
The assumption implies that $M=\bigoplus_{\mathcal{I}\in\mathrm{Spm}\mathfrak{Z}(\mathfrak{g})}\mathfrak{Fin}_\mathcal{I}M$, where $\mathrm{Spm}\,\mathfrak{Z}(\mathfrak{g})$ is the set of maximal ideals of $\mathfrak{Z}(\mathfrak{g})$. Let $D_1,...,D_l$ be a set of generators of $\mathcal{J}$, define
\[M_0:=\bigoplus_{\mathcal{J}\subseteq\mathcal{I}\in\mathrm{Spm}\mathfrak{Z}(\mathfrak{g})}\mathfrak{Fin}_\mathcal{I}M\]
and
\[M_j:=\bigoplus_{\mathcal{I}\in\mathrm{Spm}\mathfrak{Z}(\mathfrak{g}),D_1,...D_{j-1}\in\mathcal{I},D_j\notin\mathcal{I}}\mathfrak{Fin}_\mathcal{I}M,1\leq j\leq l,\]
then $M=\bigoplus_{j=0}^lM_j$, $M_0=\mathfrak{Fin}_\mathcal{J}M_0$ and $D_j$ acts on $M_j$ as an automorphism for $j=1,...,l$. It is shown in the proof of \cite[Theorem 7(2)]{Franke98} that a $(\mathfrak{g},K)$-module $M'$ is $\mathfrak{Fin}_\mathcal{J}$-acyclic if either $M'=\mathfrak{Fin}_\mathcal{J}M'$ or there is a $D\in\mathcal{J}$ acting on $M'$ as an automorphism, so $M_0,...,M_l$ and hence $M$ are all $\mathfrak{Fin}_\mathcal{J}$-acyclic.
\end{proof}

\begin{Cor}\label{acyc}
In Lemma \ref{Franke} (ii), for every compact open subgroup $\mathbb{K}\subseteq G(\mathbb{A}_f)$, $C^\infty_{\mathrm{umg}}(G)^\mathbb{K}$ and $\mathcal{A}(G)^{A_G(\mathbb{R})^\circ\mathbb{K}}$ are also $\mathfrak{Fin}_\mathcal{J}$-acyclic.
\end{Cor}

\begin{proof}
As $\mathbb{K}$ is compact, the $\mathbb{K}$-fixed part of any $(\mathfrak{m}_G,K,\mathbb{K})$-module $M$ is a direct summand of it. If $M$ is $\mathfrak{Fin}_\mathcal{J}$-acyclic as a $(\mathfrak{m}_G,K)$-module, then so is $M^\mathbb{K}$.
\end{proof}

\begin{Lem}
Let $V$ be a finite-dimensional representation of $\mathfrak{p}_h$, then there exists an ideal $\mathcal{J}\subseteq\mathfrak{Z}(\mathfrak{m}_G)$ of finite codimension such that for any $\mathfrak{m}_G$-module $W$ and $\mathfrak{p}_o$-homomorphism $f:V\rightarrow W$, $\mathcal{J}f(V)=0$.
\end{Lem}

\begin{proof}
When $V$ is irreducible, it is annihilated by $\mathfrak{p}_-$ as well as an ideal $\mathcal{I}\subseteq\mathfrak{Z}(\mathfrak{k}_o)$ of codimension 1. As in \cite[1]{Howe00} there is an algebra homomorphism $h_{\mathfrak{p}_o}$ from $\mathfrak{Z}(\mathfrak{m}_G)$ to $\mathfrak{Z}(\mathfrak{k}_o)$ such that $z-h_{\mathfrak{p}_o}(z)\in\mathfrak{U}(\mathfrak{m}_G)\mathfrak{p}_-$ for every $z\in\mathfrak{Z}(\mathfrak{m}_G)$, so we can let  $\mathcal{J}:=h_{\mathfrak{p}_o}^{-1}(\mathcal{I})$.

On the other hand if $V'\subseteq V$ is a subrepresentation and $\mathcal{J}',\mathcal{J}''$ are ideals as desired for $V',V/V'$ respectively, then we can take $\mathcal{J}:=\mathcal{J}'\mathcal{J}''$. In fact, let $W'$ be the $\mathfrak{g}$-submodule of $W$ generated by $f(V')$, then $\mathcal{J}'W'=0$ while the image of $f(V)$ in $W/W'$ is annihilated by $\mathcal{J}''$, so $f(V)$ is annihilated by $\mathcal{J}$. Therefore by induction the lemma holds for every $V$.
\end{proof}

Let $\mathcal{F}=\mathcal{F}_{\mathfrak{m}_G,K_o}^{\mathfrak{p}_o,K_o}:\mathcal{C}(\mathfrak{m}_G,K_o)\rightarrow\mathcal{C}(\mathfrak{p}_o,K_o)$ be the forgetful functor, then the lemma implies that there is an ideal $\mathcal{J}\subseteq\mathfrak{Z}(\mathfrak{m}_G)$ of finite codimension such that
\begin{equation}\label{functors}
\mathrm{Hom}_{(\mathfrak{p}_o,K_o)}(V^*,-)\circ\mathcal{F}\circ\mathfrak{Fin}_\mathcal{J}=\mathrm{Hom}_{(\mathfrak{p}_o,K_o)}(V^*,-)\circ\mathcal{F}.
\end{equation}
We also have:

\begin{Lem}
$(i)$ There are natural isomorphisms
\[R^i\mathrm{Hom}_{(\mathfrak{p}_o,K_o)}(V^*,-)\cong H^i_{(\mathfrak{p}_o,K_o)}(-\otimes V);\]
$(ii)$ $\mathcal{F}$ is exact and preserves injectives.
\end{Lem}

\begin{proof}
(i) We have $\mathrm{Hom}_{(\mathfrak{p}_o,K_o)}(V^*,-)=\mathrm{Hom}_{(\mathfrak{p}_o,K_o)}(\mathbb{C},-)\circ\mathrm{Hom}_\mathbb{C}(V^*,-)$. By \cite[2.53 (c)]{KV95} $\mathrm{Hom}_\mathbb{C}(V^*,-)$ is exact and preserves injectives, thus
\[\begin{aligned}
R^i\mathrm{Hom}_{(\mathfrak{p}_o,K_o)}(V^*,-)&=R^i(\mathrm{Hom}_{(\mathfrak{p}_o,K_o)}(\mathbb{C},-)\circ\mathrm{Hom}_\mathbb{C}(V^*,-))\\
&\cong H^i_{(\mathfrak{p}_o,K_o)}(-)\circ\mathrm{Hom}_\mathbb{C}(V^*,-)\\
&=H^i_{(\mathfrak{p}_o,K_o)}(-\otimes V).
\end{aligned}\]
(ii) $\mathcal{F}$ is clearly exact, and according to \cite[2.57 (d)]{KV95} it preserves injectives.
\end{proof}

\begin{Cor}\label{derived}
There are natural isomorphisms
\[R^i(\mathrm{Hom}_{(\mathfrak{p}_o,K_o)}(V^*,-)\circ\mathcal{F})\cong H^i_{(\mathfrak{p}_o,K_o)}(-\otimes V)\circ\mathcal{F}.\]
\end{Cor}

Finally we can prove the main theorem of this paper:

\begin{Thm}\label{main}
The inclusion $\mathcal{A}(G)^{A_G(\mathbb{R})^\circ\mathbb{K}}\subseteq C^\infty_{\mathrm{umg}}(G)^\mathbb{K}$ induces isomorphisms of $(\mathfrak{p}_o,K_o)$-cohomology with coefficients in $V$, therefore we have isomorphisms
\[H^i(\mathrm{Sh}_{\mathbb{K},\Sigma},\widetilde{V}^{\mathrm{\mathrm{can}}})\cong H^i_{(\mathfrak{p}_o,K_o)}(\mathcal{A}(G)^{A_G(\mathbb{R})^\circ\mathbb{K}}\otimes V)\]
and $G(\mathbb{A}_f)$-equivariant isomorphisms
\[\varinjlim_{t_e^*}H^i(\widetilde{V}^{\mathrm{can}}_\mathbb{K})\cong H^i_{(\mathfrak{p}_o,K_o)}(\mathcal{A}(G)^{A_G(\mathbb{R})^\circ}\otimes V_o).\]
\end{Thm}

\begin{proof}
Let $\mathcal{J}\subseteq\mathfrak{Z}(\mathfrak{m}_G)$ be an ideal of finite codimension making (\ref{functors}) hold. We will show that the inclusions of
\[\mathfrak{Fin}_\mathcal{J}\mathcal{A}(G)^{A_G(\mathbb{R})^\circ\mathbb{K}}=\mathfrak{Fin}_\mathcal{J}C^\infty_{\mathrm{umg}}(G)^\mathbb{K}\]
into $\mathcal{A}(G)^{A_G(\mathbb{R})^\circ\mathbb{K}}$ and $C^\infty_{\mathrm{umg}}(G)^\mathbb{K}$ both induce isomorphisms of $(\mathfrak{p}_o,K_o)$ cohomology with coefficients in $V$. By Lemma \ref{Franke} (i) and Corollary \ref{derived}, for every $(\mathfrak{m}_G,K_o)$-module $M$ (\ref{functors}) gives rise to a Grothendieck spectral sequence
\[E_2^{p,q}=H^p_{(\mathfrak{p}_o,K_o)}(\mathfrak{Fin}_\mathcal{J}^q(M)\otimes V)\Rightarrow H^{p+q}_{(\mathfrak{p}_o,K_o)}(M\otimes V).\]
Corollary \ref{acyc} then implies that when $M=\mathcal{A}(G)^{A_G(\mathbb{R})^\circ\mathbb{K}}$ and $C^\infty_{\mathrm{umg}}(G)^\mathbb{K}$ the spectral sequences degenerate at $E_2$-pages and give the desired isomorphisms.
\end{proof}

\begin{Rmk}
It is pointed out to us by Michael Harris that the result here has been used for \cite[5.3.11]{HZ94} and our work hence justifies the proof of that theorem.
\end{Rmk}

\bf Acknowledgement. \rm
I am grateful to my PhD advisor Richard Taylor for suggesting the problem, the countless discussions we had and proofreading the first draft of the paper. I am indebted to Michael Harris, Olivier Ta\"{i}bi and the anonymous referees for valuable comments, corrections and suggestions about the paper, which I appreciate very much. I wish to thank Vincent Pilloni for organising a groupe de travail on the paper, where some of the comments and corrections arose. I would like to thank William Casselman, Laurent Clozel, Phillip Griffiths, Kai-Wen Lan, Yiannis Sakellaridis, Wilfried Schmid, Peter Scholze, Jesse Silliman, Chris Skinner, Jingwei Xiao, Liang Xiao, Xinyi Yuan and Wei Zhang for sharing their knowledge and thoughts about this work and related topics. This paper was written during my study at Princeton University as well as visits at Beijing International Center for Mathematical Research and Stanford University, and revised when I'm a postdoc at the University of Cambridge. I wish to thank Ruochuan Liu, Jack Thorne and the institutions for their ever kind hospitality. This work was partially supported by the National Science Foundation grant DMS-1252158 and the European Research Council grant agreement No. 714405.

\bibliographystyle{amsalpha}
\bibliography{bibCCSVAF}

\end{document}